\documentclass[10pt,times,letter]{article}

\usepackage{amsmath, amsfonts,amsthm,amssymb,bm}
\usepackage{dsfont}

\usepackage[english]{babel}

\usepackage[normalem]{ulem}

\usepackage{latexsym,amssymb,amsfonts,graphicx}
\usepackage{amsmath,dsfont}
\usepackage{verbatim}
\usepackage{mathrsfs}
\usepackage{bm}
\usepackage{color}
\usepackage{epsfig}

\usepackage{url}
\usepackage{cite}
\usepackage{bm}
\usepackage{natbib}

\usepackage[colorlinks,linkcolor=blue,citecolor=blue,anchorcolor=blue]{hyperref}

\def\esssup_#1{\underset{#1}{\mathrm{ess\,sup\, }}}
\def\essinf_#1{\underset{#1}{\mathrm{ess\,inf\, }}}
\def\argmax_#1{\underset{#1}{\mathrm{arg\,max\, }}}
\def\argmin_#1{\underset{#1}{\mathrm{arg\,min\, }}}

\newtheorem{theorem}{Theorem}[section]
\newtheorem{definition}{Definition}[section]
\newtheorem{proposition}[theorem]{Proposition}
\newtheorem{remark}[theorem]{Remark}
\newtheorem{lemma}[theorem]{Lemma}
\newtheorem{corollary}[theorem]{Corollary}

\newtheorem{assumption}[theorem]{Assumption}
\newtheorem{problem}[theorem]{Problem}
\allowdisplaybreaks[4]

\definecolor{Red}{rgb}{1.00, 0.00, 0.00}
\title{Convergence analysis of controlled particle systems arising in deep learning: from finite to infinite sample size}
\author{
Huafu Liao \thanks{Email: {hfliao@dlut.edu.cn}, School of Mathematical Sciences, Dalian University of Technology.}\and
Alp\'{a}r R. M\'{e}sz\'{a}ros \thanks{Email: alpar.r.meszaros@durham.ac.uk, Department of Mathematical Sciences, University of Durham, Durham, UK.} \and
Chenchen Mou \thanks{Email: chencmou@cityu.edu.hk, Department of Mathematics, City University of Hong Kong.} \and
Chao Zhou\thanks{Email: matzc.nus.edu.sg@gmail.com, Department of Mathematics and Risk Management Institute, National University of Singapore.}
}
\date{\vspace{-5ex}}

\numberwithin{equation}{section}

\begin{document}
\maketitle
\begin{abstract}
This paper deals with a class of neural SDEs and studies the limiting behavior of the associated sampled optimal control problems as the sample size grows to infinity. The neural SDEs with $N$ samples can be linked to the $N$-particle systems with centralized control. We analyze the Hamilton--Jacobi--Bellman equation corresponding to the $N$-particle system and establish regularity results which are uniform in $N$. The uniform regularity estimates are obtained by the stochastic maximum principle and the analysis of a backward stochastic Riccati equation. Using these uniform regularity results, we show the convergence of the minima of the objective functionals and optimal parameters of the neural SDEs as the sample size $N$ tends to infinity. The limiting objects can be identified with suitable functions defined on the Wasserstein space of Borel probability measures. Furthermore, quantitative convergence rates are also obtained.

\vspace{0.3 cm}


\vspace{0.3 cm}

\end{abstract}

\section{Introduction}
In recent years, neural networks have been shown very effective for modeling complicated data sets. For situations where large amounts of samples are observed, it is important to ensure the convergence of training outcomes as the number of samples goes to infinity, i.e., the generality of the neural network. Such problems are studied in \cite{Chizat,E19}, and \cite{Mei2018}. Motivated by these studies, in this work we investigate a class of {\it neural SDEs} (see \cite{Liu2020}), which generalize {\it neural ODEs} (see e.g. \cite{E17,Haber17,Chen2018}). As the volume of samples goes to infinity, we establish quantitative results on the convergence of both the minima of the objective functionals and the optimal parameters.

Our research is concerned with the following neural SDEs
\begin{align}\label{sample-dyn-0}
\left\{\begin{aligned}
    &dX^{\theta,i}_N(t) = f\bigg(t,\theta(t),X^{\theta,i}_N(t),\frac{1}{N}\sum_{j=1}^N\delta_{X^{\theta,j}_N(t)}\bigg)dt + \sigma dW^0(t),\\
&X^{\theta,i}_N(0)=x_i,\ i=1,\ldots,N,
\end{aligned}\right.
\end{align}
where $\theta:[0,T]\to\Theta$ is a stochastic process that represents the trainable parameters (valued in a given control set $\Theta$). Here $f:[0,T]\times\Theta\times\mathbb R\times\mathcal P_{2}(\mathbb R)\to\mathbb R$ is a nonlinear function that governs the feed forward dynamics. $T>0$ is a given time horizon, $(W^{0}(t))_{t\in[0,T]}$ is a given Brownian motion on $\mathbb R$ with intensity $\sigma\in\mathbb R$, and the neural SDEs are initiated with samples $x_i\in \mathbb R,\ i=1,\ldots,N$. The standing assumptions on the data and the set up for the specific sampled optimal control problems, including the description of the objective function is given in Section \ref{The model and the problems} (see \eqref{obj-func-2} and \eqref{def-V-N-O}).

The neural SDE \eqref{sample-dyn-0} describes deep learning from a dynamical systems viewpoint. The dynamical system approach to deep learning was proposed in \cite{Haber17,E17}, and studied later in \cite{E19,Jabir,Larsson,ThGe2018,Bo22,Cohen}, etc. See also, for instance, \cite{Chen2018,Ruthotto2020,Baudelet2023,Lauriere2021}, for the application of such an approach. Following \cite{E19,E17,Haber17,QLi2017}, we explain the intuition behind this approach to deep learning as follows. For models such as residual networks, recurrent neural networks and normalizing flows, etc., given $N$ inputs $x=(x_1,\ldots,x_N)$, the typical feed-forward propagation with $K$ layers can be presented as
\begin{align*}
x_i(t+1)=x_i(t)+\epsilon f\big(x_i(t),\theta(t)\big),\quad x_i(0)=x_i,\quad t=0,1,\ldots,K-1,\quad i=1,2,\ldots,N.
\end{align*}
Here $x(0)$, $x(K)\in\mathbb R^{dN}$ represent the input layer and output layer, respectively, and $\theta(t)$ is the control (training) parameters. The goal of such learning process is to tune the trainable parameters $\theta(t)$, $t=0,1,\ldots,K-1$, so that the terminal state $x(K)$ minimizes/maximizes a given objective function. For example, in supervised learning, the objective function usually requires that $g\big(x_i(K)\big)$ most { resemble} certain label $y_i$ for $i=1,\ldots,N$, see \cite{QLi2017}; in the context of neural networks for optimal control, where $f\big(x,\theta(t)\big)$ is understood as a feedback function with $\theta(t)$ to be determined, the aim is to maximize the expectation of the utility $\frac1N\sum_{i=1}^NU\big(x_i(K)\big)$, see \cite{Hure21,Hure22}. As the layer number $K$ tends to infinity, after an appropriate rescaling $\epsilon K=T$, it is straightforward to present the above iteration via an ODE
\begin{align*}
\dot{x}_i(t)=f\big(x_i(t),\theta(t)\big),\quad x_i(0)=x_i,\quad t\in[0,T],\ i=1,2,\ldots,N,
\end{align*}
where the terminal state is $x(T)$. Here in a more general setting, we incorporate the feed-forward propagation with a systemic noise as well as the empirical distribution of inputs. The resulting continuous idealization is then naturally generalized to the neural SDE \eqref{sample-dyn-0}. A typical situation where the empirical distribution enters the feed-forward propagation is batch normalization (see \cite{IoffeSzegedy2015}). For example, we may take the dynamics in \eqref{sample-dyn-0} to be
\begin{align*}
  f(t,\theta,x,\mu)=\tilde f\left(\theta,\frac{x-\int y\mu(dy)}{\sqrt{\int y^2\mu(dy)+\epsilon}}\right)
\end{align*}
for some function $\tilde f$, where $\epsilon>0$ is a given parameter.
We note here that the controlled particle system \eqref{sample-dyn-0} is different from the usual mean field type, typically studied in the literature, in the sense that every particle $X^{\theta,i}_N(t)$ in the system shares the {\it same control} $\theta(t)$ rather than having their own $\theta^i(t)$.

In summary, our results make it possible to quantitatively analyze the convergence of training outcomes obtained from neural SDEs with $N$ samples when $N\to+\infty$. Speaking with the terms of optimal control theory, our problem relates to the convergence of the value functions and optimal controls of the particle systems  with centralized controls, i.e., the propagation of chaos or the law of large numbers. As the number of particles grows to infinity, we explore sufficient conditions that ensure the aforementioned convergence. Such convergences are possible thanks to the presence of an $L^2$-regularizer in the objective functional. Furthermore, quantitative results on the convergence rate are also obtained. Above all, our exploration can be interpreted as conditions ensuring the generalization of models in deep learning. Specifically, an adaptation of our results to the analysis on supervised learning is carried out in Section \ref{adaptation}.

Propagation of chaos on controlled particle systems {have been extensively studied} recently. To name a few, we refer to \cite{Lacker17, Cardaliaguet19, Cardaliaguet2022, Carmona2022,Daudin2023,Djete22,Gangbo21,Gangbo22,Mayorga23,Tangpi2022,Tangpi2023,MengzhenLi2023}, see also \cite{Coghi16,Chassagneux22,Germain2022,Ren2022,Ren2023}, as well as the references therein for the ones with uncontrolled particle systems. The limit of the value functions in the aforementioned convergence is a function whose state variable is a probability measure. For literature on such limits, see for instance \cite{Wei2017,Wu2020,Gangbo21,Gangbo22,Mayorga23}.

As mentioned before, our model \eqref{sample-dyn-0} is significantly different from the ones above in terms of the form of control, which thus results in a very different Hamilton--Jacobi--Bellman (HJB) equation. Although similar models are studied in \cite{Kou2016,E2016,E20171,H-L2017,Hure21,Hure22}, their emphasis is on the analysis of the corresponding algorithm. The models and results in \cite{Bo22,Yu2022,E19} are the closest ones to the present paper, where the convergence of both value functions and optimal controls are investigated. In \cite{Bo22,Yu2022} the law of large numbers is obtained where there is no quantitative results. In \cite{E19} on the other hand, the authors focus on the deterministic control and obtain quantitative results on large deviations, but the state dynamics $f$ therein is required to be independent of the distribution of particles. Here, we study models with more general state dynamics $f$ that could depend on the distribution of particles and obtain quantitative results. More specifically, besides the law of large numbers, we further show the corresponding convergence rate: as the sample size $N$ grows to infinity, the minima of the objective functional, i.e. $V_N$ and the optimal feedback function $\theta^*_N$ converge, at certain rates, to a value function and a feedback function whose state variable is the empirical measure of the samples. As a result, we show that the optimal parameters also converge at certain
rates. We obtain two kinds of convergence results: the short time convergence and the global convergence, both accompanied by a precise convergence rates.

The HJB equation written for the value function $V_{N}$ associated to our main control problem, i.e. Problem \ref{problem}, formally reads as 
\begin{align}\label{intro:1}
\left\{\begin{aligned}
    &\partial_t V_{N}+\frac{\sigma^2}2\sum_{i,j=1}^N\partial^2_{x_ix_j}V_N+\inf_{\theta\in\Theta}\bigg\{\frac\lambda2\big|\theta\big|^2+\sum_{i=1}^Nf\bigg(t,\theta,x_i,\frac1N\sum_{j=1}^N\delta_{x_j}\bigg)\partial_{x_i}V_N+\frac1 N\sum_{i=1}^NL(x_i)\bigg\}=0,\\
&\quad(t,x)\in[0,T)\times\mathbb R^N,\\
&V_N(T,x_1,\ldots,x_N)=\frac1N\sum_{i=1}^NU(x_i),\quad x\in\mathbb R^N,
\end{aligned}\right.
\end{align}
where $L$ and $U$ are the suitably chosen loss function and final cost function, respectively, and the control functions are valued in some control set $\Theta$ and $\sigma\in\mathbb R$ and $\lambda>0$ are given further parameters.

To study the convergence of the value functions and optimal feedback functions of Problem \ref{problem}, the main ingredients are the regularity results on the HJB equations which are uniform and decay suitably with respect to $N$ -- the dimension of the input variables (Theorem \ref{veri} and Theorem \ref{veri-1}). This idea is in the spirit of \cite{Gangbo22} where the control is of the mean field type. However, in contrast to this, in our problem we are faced with the same common control for each particle and the dynamics of each particle could be nonlinear with respect to the control variable. Hence the resulting Hamiltonian is significantly different in structure. Such Hamiltonian requires additional analysis where the a priori estimates on $\nabla_x V_N$ and the regularizer play important roles. Moreover, in our problem there is a common Brownian motion in the dynamics of each particle. Therefore the method in \cite{Gangbo22} is no longer applicable directly in our situation. Instead, here we rely more on a probabilistic approach to analyze the HJB equation and to obtain the desired regularity results. 

The {\it first main contribution} of this paper is the uniform (in $N$) estimates on the degenerate PDE systems describing $V_N$, as well as $\nabla_x V_N$, $\nabla^2_{xx}V_N$. Such uniform estimates will imply the convergence rate of $V_N(t,x)$ and the corresponding feedback functions $\theta^*_N(t,x)$.  In order to obtain the desired uniform estimates, we apply the nonlinear Feymann--Kac representation and focus on the stochastic processes corresponding to $V_N$, $\nabla_x V_N$ and $\nabla^2_{xx}V_N$, respectively. Because of the degenerate nature of the problem, we need to introduce regularizations at several levels: these will be via involving non-degenerate idiosyncratic noise as well as some suitable cut-off procedures to handle the growth properties of the data. Our estimates will turn out to be independent of these regularization parameters. It is well known that $\nabla_x V_N$  corresponds to the adjoint process in the stochastic maximum principle. As a result of this, we apply the stochastic maximum principle and obtain that each entry of $\nabla_x V_N$ decays at the rate of $O(N^{-1})$. However, the analysis of the systems involving $\nabla^2_{xx}V_N$ is more subtle. It turns out that the suitable approximations of $\nabla^2_{xx}V_N$ introduced above, are related to matrix-valued processes ($Y_t$ in \eqref{riccati-0-0} and \eqref{riccati}) that satisfy backward stochastic Riccati equations. We first conduct an analysis on the Hamitonian based on the a priori estimates on $\nabla_xV_N$. Then local in time estimates on the processes $Y_t$ are obtained via the contraction mapping principle: the $(i,j)-$entry of $Y_t$ has a decay rate of $O(\delta_{ij}N^{-1}+N^{-2})$. As for the global estimates, we make further suitable convexity assumptions on the data and analyze the eigenvalues of $Y_t$ utilizing the Riccati (i.e. quadratic) feature of the corresponding BSDE. Under these extra assumptions, each eigenvalue of $Y_t$ decays at the rate of $O(N^{-1})$ for arbitrary long time horizon $T$. These convexity assumptions are similar in spirit to displacement convexity (used in \cite{BENSOUSSAN2019,Carmona2015,Gangbo22}), however, they are not covered by the existing literature (not even by the displacement monotonicity conditions introduced in \cite{Gangbo2022-1, BanMesMou}), because the state dynamics given by $f$ is allowed to have a measure dependence. We note here that such measure dependence of $f$ has been investigated in \cite{BENSOUSSAN2023, Djete22,Lacker17,Mou2022MFGC} within the framework of standard mean field games and control.

Our {\it second main contribution} is the convergence analysis on $V_N(t,x)$ and $\theta_N(t,x)$ on a quantitative level. We use a variational approach to show that $V_N$ and $\theta^*_N$ are both finite dimensional projection of certain functions $\mathcal V$ and $\theta^*$ whose state variables are probability measures. Furthermore, thanks to the previous uniform estimates, we show that, both $\mathcal V$ and $\theta^*$ are Lipschitz continuous with respect to their state variables. Under  our two sets of different assumptions, the previous results hold for a short time horizon or global in time, respectively. Such convergence of $V_N(t,x)$ and $\theta_N(t,x)$ has two major implications on neural SDEs. First, the convergence $V_N(t,x)$ translates to the convergence of minima of objective functionals. Second, the convergence of $\theta_N(t,x)$ would yield pathwise convergence results that translate to the convergence of optimal parameters obtained via neural SDEs (see Proposition \ref{cauchy-sequence-0-0} and Proposition \ref{cauchy-sequence-1-0}).
\medskip

{\bf Some concluding remarks.} The limit function $\mathcal V$ is formally associated to a second order HJB equation set on the Wasserstein space $\mathcal P_{2}(\mathbb R)$. This formally reads as 
\begin{align}\label{intro:2}
\left\{\begin{aligned}
    &\partial_t\mathcal V(t,\mu)+\frac{\sigma^2}2\left\{\int_{\mathbb R}\partial_{y\mu}\mathcal V(t,\mu)(y)\mu(dy)+\int_{\mathbb R^{2}}\partial_{\mu\mu}\mathcal V(t,\mu)(y,y')\mu(dy)\mu(dy')\right\}\\
    &\qquad\  +\inf_{\theta\in\Theta}\left\{\frac\lambda2\big|\theta\big|^2+\int_{\mathbb R}f\left(t,\theta,y,\mu\right)\partial_{\mu}\mathcal V(t,\mu)(y)\mu(dy)+\int_{\mathbb R}L(y)\mu(dy)\right\}=0,\ \ (t,\mu)\in[0,T)\times\mathcal P_{2}(\mathbb R),\\
&\mathcal V(T,\mu)=\int_{\mathbb R}U(y)\mu(dy),\quad \mu\in\mathcal P_{2}(\mathbb R).
\end{aligned}\right.
\end{align}
We would like to underline at this stage that studying the quantitative decay estimates with respect to $N$ of second order spacial derivatives of $V_{N}$ (that we perform in this paper) results in the fact that $\partial_{\mu}\mathcal V$ exists and it is Lipschitz continuous in a suitable sense. The very same analysis that we perform on these objects could be pushed further, to study quantitative third order derivative estimates for $V_{N}$, which would result in twice differentiability of $\mathcal V$, and hence in the fact that $\mathcal V$ is a classical solution to the HJB equation \eqref{intro:2}. This would be very much in the flavor of the $C^{2,1,w}(\mathcal P_{2}(\mathbb R))$ type estimates from \cite{Gangbo22}. { When $\sigma=0$, as a consequence of our results, we actually establish the well-posedness of \eqref{intro:2} with solutions admitting derivatives that induce Lipschitz continuous maps in the Wasserstein space. We refer to the discussion after Corollary \ref{extend-derivative} for more details. Roughly speaking, the solution in this case is regular enough to let every term in \eqref{intro:2} to be meaningful in the classical sense, complementing the results on viscosity solutions in \cite{E19}.}

\medskip

The specific choice for $L, U$ and $f$ in the above setting is motivated by the concrete applications in deep neural networks we have described above. In our analysis, in  fact one would be able to allow more general measure dependent functions in \eqref{intro:2}.

\medskip

We would like to emphasize once more that connections between equations of type \eqref{intro:1} and \eqref{intro:2}, and the corresponding quantitative rates of convergence as $N\to+\infty$ have received a great attention in the {  past couple of years. For a non-exhaustive list of works we refer to \cite{Cardaliaguet2022, Daudin2023, Car2023, Gangbo21, Mayorga23, Swiech24}. The results of these papers differ significantly from ours and their motivation is quite different. Let us spend some time describing the main differences and similarities between our approach and the ones previously considered in the literature. The common feature of the previous works is to start with control problems and HJB equations in the infinite dimensional setting of probability measures, and then project the viscosity solutions onto finite dimensional subspaces (typically the subspace of empirical measures), then study the discrepancy between such projections and the value functions associated to the finite dimensional control/particle systems. The corresponding convergence analysis then can be divided into two distinct categories, essentially based on the notion of viscosity solutions used in the infinite dimensional setting. 

First, in \cite{Cardaliaguet2022, Daudin2023, Car2023} --- and in related subsequent works ---   the authors obtain quantitative convergence rates based on Lipschitz and semi-concavity estimates on the infinite dimensional value functions. The rates are proved to be actually sharper in regions where the infinite dimensional value functions are more regular. Classical solutions to the infinite dimensional PDEs, with suitable derivative bounds are known to give rise to sharp convergence rates (cf. \cite{Cardaliaguet19}). For these methods it was indispensable though the presence of non-degenerate idiosyncratic noise and parabolic regularity.  

Second, the line of works \cite{Gangbo21, Mayorga23, Swiech24} considered the presence of purely common noise or deterministic problems, and they have relied on the `Hilbertian lifting' of the infinite dimensional HJB equations. This approach, however, gave only qualitative convergence results in general.

In our models, we consider purely common noise coming from centralized control problems (meaning that all particles share the same control process). Centralized controls do not enter the framework of the previously mentioned works. While these models are relevant in learning applications, to the best our knowledge only \cite{E19} has considered such kind of convergence questions in the case of deterministic problems. Furthermore, the above mentioned papers relied on viscosity solution techniques and regularization procedures for semi-concave and Lipschitz continuous functions defined on the Wasserstein space. Our approach is completely different from such methods, and our analysis is based on finite dimensional quantitative analysis and a careful combination of parabolic PDE techniques and stochastic analysis of FBSDE systems. As such, it relies on obtaining uniform in $N$ derivative decay estimates, and obtains the limiting functions via a compactness arguments. Such a method nowadays is referred to as a {\it bottom-up approach}.

After the original posting of our manuscript, methods of bottom-up nature became instrumental in tackling other important problems of similar nature:  we refer to \cite{CirRed, CirJacRed} for the quantitative convergence of Nash equilibria in mean field games and non-exchangeable models and to the works \cite{Liao24, JacMes:25, JacMes:26, LamMes} for building solutions to master equations and quantitative convergence problems in the case of so-called mean field games of controls. As this has been done in \cite{JacMes:26} for a different model, for the last building block in establishing that the limit value function $\mathcal{V}$ is a classical solution to \eqref{intro:2} in the case when $\sigma\neq 0$, we would need the establishment of uniform decay estimates on third order derivatives of $V_N$. While this should be possible, we believe that such a tedious analysis would significantly deviate attention from the original motivation of our paper, which is the convergence of $(V_N)_{N\in\mathbb N}$ and $(\theta_N)_{N\in\mathbb N}$. Therefore, we have elected to not include these further arguments in the current manuscript. It worth mentioning again that in the case when $\sigma = 0$, $\mathcal V$ is indeed a classical solution to \eqref{intro:2}.

In our paper, in addition to the value functions, the convergence rate of control (training) parameters $\theta_N$ constitutes a large part of our analysis, whereas most of the above mentioned works focus on the convergence rate of value functions $V_N$ only.
}

\medskip

\medskip

The remainder of the paper is organized as follows. In Section \ref{The model and the problems} we describe the model, the main problem of interest and our main results. In Section \ref{Auxiliary problems and corresponding uniform estimates} we first introduce the auxiliary problems and study the regularity of the corresponding value functions. Then we establish the estimate on the derivatives of the value function as well as the verification results associated to the original problem. In Section \ref{Discussion on convergence rate} we show that the value function $V_N$ in Problem \ref{problem} is the finite dimensional projection of a function $\mathcal V$ whose state variable is in the space of probability measures, and establish the results on the convergence rate. { In Section \ref{adaptation} we adapt our results to the analysis on the training with monotone loss functions.}
\section{The model problem, standing assumptions and main results}\label{The model and the problems}

Let $T>0$ be a given time horizon. Let $(\Omega,\mathbb P,\mathcal F,\mathbb F)$ be an augmented filtered probability space satisfying the usual conditions, where $\mathbb F=(\mathcal F_t)_{t\in[0,T]}$ is the natural filtration generated by a sequence of independent Brownian motions $\{W^i\}_{i=0}^\infty$.

The {$[0,T]\ni t\mapsto X^{\theta,i}_N(t),\ i=1,\ldots,N,$} in \eqref{sample-dyn-0} is a sequence of controlled diffusion processes coupled with the common noise $W^0(t)$ and the mean field term $\frac{1}{N}\sum_{j=1}^N\delta_{X^{\theta,j}_N(t)}$. The control {$[0,T]\ni t\mapsto\theta(t)$} in \eqref{sample-dyn-0}, which is understood as the weight process in deep learning, is shared among the dynamics of all $X^{\theta,i}_N(t)$.

Next we formulate the dynamic version of the optimization problem. For $x_1,x_2,\ldots,$ consider the admissible set $\mathcal U^{ad}_t$ consisting of the tuple $(\Omega,\mathbb P,\mathcal F,\{W^i\}_{i=0}^{+\infty},\theta)$ satisfying the following:
\begin{itemize}
\item $(\Omega,\mathbb P,\mathcal F)$ is a complete probability space.
\item $\{W^i\}_{i=0}^N$ are independent Brownian motions defined on $(\Omega,\mathbb P,\mathcal F)$ with $W^i(s)=0$ almost surely and $\mathcal F_t:=\sigma\big(W^i(u),u\in[s,t],i\geq0\big)$ augmented by all the $\mathbb P$-null sets in $\mathcal F$.
  \item $\theta$ is an $\{\mathcal F_t\}_{s\leq t\leq T}$-adapted process on $(\Omega,\mathbb P,\mathcal F)$ and $\theta(s)\in\Theta$, $s\in[t,T]$.
  \item For each $N\geq1$, $\big(\{X^{\theta,i}_N\}_{i=1}^N,W^0,\theta\big)$ solves \eqref{sample-dyn} on $(\Omega,\mathbb P,\mathcal F,\{\mathcal F_t\}_{s\leq t\leq T})$ where the pathwise uniqueness holds:
  \begin{align}\label{sample-dyn}
\left\{\begin{aligned}
    &dX^{\theta,i}_N(t) = f\bigg(t,\theta(t),X^{\theta,i}_N(t),\frac{1}{N}\sum_{j=1}^N\delta_{X^{\theta,j}_N(t)}\bigg)dt + \sigma dW^0(t),\\
&X^{\theta,i}_N(s)=x_i,\ i=1,\ldots,N.
\end{aligned}\right.
\end{align}
\end{itemize}
When there is no ambiguity, we use $\theta$ to denote the admissible control. We remark that the role of $\{W^i\}_{i=1}^{+\infty}$ will become evident when considering the approximated problem \eqref{sample-dyn-1}.

Given $s\in[0,T]$, a control $\theta\in\mathcal U^{ad}_s$ and $N$ inputs $(x_1,\ldots,x_N)\in\mathbb R^N$, we can further define the objective function {$J_N(s,\cdot):\mathcal U^{ad}_s\times\mathbb R^N\to\mathbb R$} as follows
\begin{align}\label{obj-func-2}
J_N(s,\theta,x_1,\ldots,x_N):=\mathbb E\bigg[\frac1 N\sum_{i=1}^N\int_s^TL\big(X^{\theta,i}_N(t)\big)dt+\frac1N\sum_{i=1}^NU\big(X^{\theta,i}_N(T)\big)+\frac\lambda2\int_s^T\left|\theta(t)\right|^2dt\bigg],
\end{align}
The third term on the right hand side of \eqref{obj-func-2} is the regularizer.  In applications, we may also choose the regularizer to be $\frac\lambda2\int_s^T\left|\theta(t)-\eta_t\right|^2dt$ where $(\eta_t)_{t\in[0,T]}$ is a deterministic reference process. 
It is straightforward but notationally cumbersome to generalize our results to the case where $\Theta=\mathbb R^d$ and $x_i\in\mathbb R^m$. For the ease of notations and convenience in this paper we choose $d=m=1$.

In our analysis we consider the space of Borel probability measures, supported in Euclidean spaces $\mathbb{R}^m$. We work on the specific subset of these measures, which have finite second moment, and denote this by $\mathcal P_2(\mathbb{R}^m)$. We equip this subset with the classical $2$-Wasserstein distance $\mathcal W_2$. 

Here we make the following technical assumptions on parameters.

\begin{assumption}\label{assumption-1}
Assume that
  \begin{enumerate}
		\item The function $[0,T]\times{\Theta}\ni(t,\theta)\mapsto f(t,\theta,0,\delta_{\{0\}})$ is continuous, where $\Theta=\mathbb R$;
		\item the function {$f:[0,T]\times\Theta\times\mathbb R\times{\mathcal P(\mathbb R)}\to\mathbb R$} is such that $\partial_tf$ is bounded and has bounded derivatives with respect to $(\theta,x,\mu)$ up to the second order;
\item For $\varphi\in\{L, U\}$, {$\varphi\geq0$, and} there exist constants $C^\varphi_{11}$, $C^\varphi_{10}$, $C^\varphi_{20}$ such that
\begin{align}\label{assumption-L-U}
|\varphi'(x)|\leq C^\varphi_{11}|x|+C^\varphi_{10},\ |\varphi''(x)|\leq C^\varphi_{20}.
\end{align}
{ \item There exist positive constants $\alpha\in(0,2),\ C_g>0$ such that $|L(x)|+|U(x)|\leq C_g(|x|^{2-\alpha}+1),\ x\in\mathbb R.$}
\end{enumerate}
\end{assumption}
{ We will always suppose Assumption \ref{assumption-1} throughout the whole paper.} In Assumption \ref{assumption-1}, by derivative with respect to $\mu$ we mean the intrinsic derivative, the so-called Wasserstein derivative (see for instance \cite[Definition 2.2.2]{Cardaliaguet19} or \cite[Chapter 5]{Carmona2018-I} and the discussion therein). In particular, when we say differentiability with respect to the measure variable, we always mean the so-called fully $C^1$, $C^2$, etc. classes (see \cite[Chapter 5]{Carmona2018-I}). In what follows we use the notation $\partial_\mu$ to denote this intrinsic Wasserstein derivative. We denote by $\tilde x$ the new variable arising after applying $\partial_\mu$, and we display this after the measure variable as $\partial_\mu g(\mu,\tilde x)$, for any $g \in C^1(\mathcal P_2(\mathbb{R}^M)).$ Under the prescribed framework, we consider the following optimization problem:
\begin{problem}\label{problem}
  Minimizing \eqref{obj-func-2} over $\mathcal U^{ad}_t$.
\end{problem}
As is mentioned in the introduction, the optimization of $J_N$ under the constraint \eqref{sample-dyn} can be understood as a learning process with neural SDEs/ODEs. For a direct example, consider the case where $\sigma=0$ and the optimal control $\theta^*_N(t)$ is deterministic, optimizing $J_N$ can be viewed as using neural network to choose the best feedback function of the form $f\big(t,\theta(t),\cdot\big)$ for optimal control problems according to the sampled objective functional $J_N$, where $(x_1,\ldots,x_N)$ are the sample inputs drawn from a certain distribution. Here $\theta(t)$ is the parameter to be determined, which is part of the feedback function. See \cite{Hure21,Hure22} for similar discrete time models. Please see Section \ref{adaptation} for more interpretation on $\theta^*_N(t)$ in another concrete example.
Denote the value function to Problem \ref{problem} by
\begin{align}\label{def-V-N-O}
V_N(t,x_1,\ldots,x_N):=\inf_{\theta\in\mathcal U^{ad}_t}J_N(t,\theta,x_1,\ldots,x_N),
\end{align}
and $\theta^*_N(t,x_1,\ldots,x_N)$ one of the optimal feedback functions (if exists). Suppose that
\begin{align*}
\mathcal W_2\bigg(\frac1N\sum_{k=1}^N\delta_{x_k},\mu\bigg)\ \longrightarrow\ 0\quad\text{as}\quad N\to+\infty,
\end{align*}
where $\mu\in\mathcal{P}_2(\mathbb R)$ is a given probability measure. We are interested in establishing the quantitative convergence results of the value functions $V_N(t,x_1,\ldots,x_N)$, optimal feedback functions $\theta^*_N(t,x_1,\ldots,x_N)$ as well as optimal parameters $\theta^*_N(t)=\theta^*_N\big(t,X^{\theta^*_N,1}_N(t),\ldots,X^{\theta^*_N,N}_N(t)\big)$ when $N\to+\infty$. Under suitable assumptions, we give positive answers to these convergence questions in Section \ref{Discussion on convergence rate}. Towards that end, our first finding is that $V_N$ is the finite projection of a local Lipschitz function $\mathcal V$ defined on probability space. Such $\mathcal V$ is first defined on the set of empirical measures then extended to $\mathcal P_2(\mathbb R)$.
\begin{definition}\label{def-emp-meas}
For samples $x_1,\ldots,x_N\in\mathbb{R}$, denote by $\mu^N$ the corresponding empirical measure $\mu^N=\frac1N\sum_{i=1}^N\delta_{x_i}.$ Define
\begin{align}\label{def-U}
{\cal V}(t,\mu^N):=V_N(t,x_1,\ldots,x_N).
\end{align}
\end{definition}
We will show in Lemma \ref{well-define} that Definition \ref{def-emp-meas} is meaningful, i.e., if $\frac1N\sum_{i=1}^N\delta_{x_i}=\frac1M\sum_{i=1}^M\delta_{y_i}$ for some $N,M\in\mathbb{N}$, then $V_N(t,x_1,\ldots,x_N)=V_M(t,y_1,\ldots,y_M).$ Moreover, we show in Section \ref{Discussion on convergence rate} that $\mathcal V$ is locally Lipschitz, which implies the convergence rate of $V_N$ as summarized in Corollary \ref{VN-con-rate}. The Lipschitz continuity is formulated as follows. { We refer to the description of the additional assumptions Hypothesis (R) and Hypothesis (R1) in Subsection \ref{sec:32}.}
\begin{theorem}\label{extend-U}
Suppose {\bf Hypothesis (R)}. Let $\mu_1$ and $\mu_2$ be two empirical measures, then for each $t\in[0,T]$,
\begin{align*}
    |{\cal V}(t,\mu_1)-{\cal V}(t,\mu_2)|&\leq\tilde C_{71}\mathcal W_2(\mu_1,\mu_2)+{\tilde C_{72}\bigg[\mathcal W^2_2(\mu_1,\mu_2)+\bigg(\int_{\mathbb R}y^2\mu_1(dy)+\int_{\mathbb R}y^2\mu_2(dy)\bigg)\mathcal W_2(\mu_1,\mu_2)\bigg]},
\end{align*}
where the constants $\tilde C_{71},\ \tilde C_{72}$ are defined via those in Proposition \ref{1st-deri-2-1},
\begin{align}\label{tilde-C-7}
\tilde C_{71}=\tilde C_2(C^L_{11}+C^U_{11}+C^L_{10}+C^U_{10}),\quad\tilde C_{72}=\frac{\tilde C_2(C^L_{11}+C^U_{11})}2.
\end{align}
As a result, ${\cal V}(t,\cdot)$ can be uniquely extended to a local Lipschitz function on $\mathcal P_2(\mathbb R)$.
\end{theorem}
\begin{corollary}\label{VN-con-rate}
Suppose {\bf Hypothesis (R)} or {\bf Hypothesis (R1)}. Let $\mu^N:=\frac1N\sum_{i=1}^N\delta_{x_i}\to\mu$ in $(\mathcal P_2(\mathbb R),\mathcal W_{2}),$ as $N\to+\infty$. Then
\begin{align*}
\lim_{N\to+\infty}V_N(t,x_1,\ldots,x_N)={\cal V}(t,\mu),
\end{align*}
at a rate
\begin{align*}
\quad|V_N(t,x_1,\ldots,x_N)-{\cal V}(t,\mu)|\notag & \leq\tilde C_{71}\mathcal W_2(\mu^{N},\mu)\\
&+\tilde C_{72}\left[\mathcal W^2_2(\mu^N,\mu)+\left(\int_{\mathbb R}x^2\mu^N(dx)+\int_{\mathbb R} x^2\mu(dx)\right)\mathcal W_2(\mu_1,\mu_2)\right].
\end{align*}
\end{corollary}
As will be shown in Lemma \ref{well-define-1}, for $t\in[0,T]$ and $\mu^N=\frac1N\sum_{i=1}^N\delta_{x_i}$, ${\cal V}$ is differentiable at $(t,\mu^N)$. Moreover,
\begin{align*}
 \partial_\mu{\cal V}(t,\mu^N,x_i)=N\partial_{x_i}V_N(t,x_1,\ldots,x_N),\ i=1,\ldots,N.
\end{align*}
Therefore the optimal feedback function for $V_N$ can be regarded as the finite projection of $\theta^*: \ [0,T]\times\mathcal P_2(\mathbb R)\mapsto \mathbb R$, where for empirical measure $\mu^N$ its definition is
\begin{align*}%\label{measure-theta-0}
\theta^*(t,\mu^N)&:=\theta^*_N\big(t,x,\nabla_xV_N(t,x)\big)\notag\\
  &=\argmin_{\theta\in\Theta}\bigg\{\frac\lambda2|\theta|^2+\int_{\mathbb R}f(t,\theta,y,\mu^N)\partial_\mu {\cal V}(t,\mu^N,y)\mu^N(dy)\bigg\}.
\end{align*}
In Section \ref{Discussion on convergence rate} under suitable assumptions, $\theta^*$ is proved to be Lipschitz continuous for arbitrary time horizon $T$.
\begin{theorem}\label{lip-lift-theta-1}
  Suppose the same assumptions as in Proposition \ref{dY/dx-1}.
  Let $\mu^N=\frac1N\sum_{i=1}^N\delta_{x_i}$, $\nu^N=\frac1N\sum_{i=1}^N\delta_{y_i}$. Then for $T>0$,
  \begin{align}\label{theta-convergence-1}
|\theta^*(t,\mu^N)-\theta^*(t,\nu^N)|\leq\tilde C_9\mathcal W_2(\mu^N,\nu^N),\ t\in[0,T],
\end{align}
for some $\tilde C_9=\tilde C_9(f,\lambda^{\frac12},T,L,U,(\lambda-\lambda_0)^{-\frac12})$.
\end{theorem}
One implication of the above is the quantitative estimates on optimal feedback functions.
\begin{corollary}\label{theta-con-rate-1}
Suppose that the assumptions of Theorem \ref{lip-lift-theta-1} take place and suppose that $\mu^N:=\frac1N\sum_{i=1}^N\delta_{x_i}\to\mu$ in $(\mathcal P_2(\mathbb R),\mathcal W_{2}),$ as $N\to+\infty$. Then for $T>0$,
\begin{align*}
\lim_{N\to+\infty}\theta^*_N\big(t,x,\nabla_xV_N(t,x)\big)=\theta^*(t,\mu),\ t\in[0,T],
\end{align*}
at a rate
\begin{align*}
&\quad|\theta^*_N\big(t,x,\nabla_xV_N(t,x)\big)-\theta^*(t,\mu)|\leq\tilde C_9\mathcal W_2(\mu^N,\mu).
\end{align*}
\end{corollary}
For the case where the distribution of samples converge, Theorem \ref{lip-lift-theta-1} yields the %pathwise 
convergence rate of the empirical measures supported on the optimal path for each sample; and above all, the %pathwise 
convergence rate of optimal controls. %, which are also understood as the optimal parameters obtained from neural networks. 
From the viewpoint of neural networks, such convergence corresponds to the generality of neural networks.
\begin{proposition}\label{cauchy-sequence-1-0}
 Let $X^*_N = (X^{1,*}_N(s),\dots,X^{N,*}_N(s))_{s\in[0,T]}$, $N\geq1$ be the optimal path in Problem \ref{problem} with $t=0$ and initial values $x^{(1)}_N,x^{(2)}_N,\ldots,x^{(N)}_N$. Suppose that the assumptions of Theorem \ref{lip-lift-theta-1} take place and suppose that $\frac1N\sum_{i=1}^N\delta_{x^{(i)}_N}\to\mu\ \text{in}\ (\mathcal P_2(\mathbb R),\mathcal W_{2})$, as $N\to+\infty$. Then for $T>0$, there exists an adapted limit process $(\theta^*,\mu^*)$, where $\theta^*(s)\in\Theta$ and $\mu^*(s)\in\mathcal P_1(\mathbb R)$, $0\leq s\leq T$, such that $\mu^*(0)=\mu$ and 
 \begin{align}\label{cauchy-sequence-0-1}
 &\max_{s\in[0,T]}\mathcal W_1(\mu^*_N(s),\mu^*(s))\leq\hat C_9\mathcal W_1\big(\mu^*_N(0),\mu(0)\big)\ a.s.,\notag\\
 &|\theta^*_N\big(s,X^*_N(s),\nabla_xV_N(s,X^*_N(s))\big)-\theta^*(s)|\leq\hat C_9\mathcal W_1(\mu^*_N(0),\mu(0))\ a.s.
\end{align}
Here $\mu^*_N(s):=\frac1N\sum_{i=1}^N\delta_{X^{i,*}_N(s)}$ and $\hat C_9>0$ is a constant independent of $N$.
\end{proposition}
\begin{proof}
 The proof is the same as in Proposition \ref{cauchy-sequence-0-0}. %Here we use Theorem \ref{lip-lift-theta-1} instead to obtain the analogy of \eqref{theta-convergence-1}.
\end{proof}

We also note here that Theorem \ref{lip-lift-theta-1}, Corollary \ref{theta-con-rate-1} and Proposition \ref{cauchy-sequence-1-0} have short time counter parts that will be discussed in Section \ref{Discussion on convergence rate}.

\section{The auxiliary problems and corresponding uniform estimates}\label{Auxiliary problems and corresponding uniform estimates}
{In order to study the aforementioned convergence as well as the convergence rate, we establish uniform derivative estimates on $V_N$ as the number of variables increases, which is different from the usual PDE estimates. Our results include the uniform estimates on the first and the second order derivatives of $V_N$. These estimates are used in Section \ref{Discussion on convergence rate}. It turns out (as we will see in the next section) that the estimates on the first order derivatives yield the convergence rate of $V_N(t,x_1,\ldots,x_N)$, while the estimates on the second order derivatives yield the convergence rate of $\theta^*_N(t,x_1,\ldots,x_N)$.}
\subsection{The auxiliary problems and the estimates on the first order derivatives}
To solve \eqref{def-V-N-O}, the dynamic programming principle yields the HJB equation
\begin{align}
\left\{\begin{aligned}
    &\partial_tV_N+\frac{\sigma^2}2\sum_{i,j=1}^N\partial^2_{x_ix_j}V_N+\inf_{\theta\in\Theta}\bigg\{\frac\lambda2\big|\theta\big|^2+\sum_{i=1}^Nf\bigg(t,\theta,x_i,\frac1N\sum_{j=1}^N\delta_{x_j}\bigg)\partial_{x_i}V_N+\frac1 N\sum_{i=1}^NL(x_i)\bigg\}=0,\\
&\quad(t,x)\in[0,T)\times\mathbb R^N,\\
\label{HJB-V-N}&V_N(T,x_1,\ldots,x_N)=\frac1N\sum_{i=1}^NU(x_i),\quad x\in\mathbb R^N.
\end{aligned}\right.
\end{align}
The equation \eqref{HJB-V-N} is degenerate parabolic, as the Fourier symbol of the second order differential operator is given by
\begin{align*}
\frac{\sigma^2}2\sum_{i,j=1}^N\xi_i\xi_j=\frac{\sigma^2}2\left(\sum_{i=1}^N\xi_i\right)^2.
\end{align*}
Hence the classical solution to \eqref{HJB-V-N} is not guaranteed by standard results.

In order to study \eqref{HJB-V-N}, we introduce the following auxiliary equation with parameters $R=(R_1,R_2)$ and $\varepsilon$:
\begin{align}\label{HJB-V-N-1}
\left\{\begin{aligned}&\partial_tV^{\varepsilon,R}_N+\frac{\sigma^2}2\sum_{i,j=1}^N\partial^2_{x_ix_j}V^{\varepsilon,R}_N+\frac{\varepsilon^2}2\sum_{i=1}^N\partial^2_{x_ix_i}V^{\varepsilon,R}_N\\
&\qquad +\inf_{\theta\in\Theta_{R_2}}\bigg\{\frac\lambda2\big|\theta\big|^2+\sum_{i=1}^Nf\bigg(t,\theta,x_i,\frac1N\sum_{j=1}^N\delta_{x_j}\bigg)\partial_{x_i}V^{\varepsilon,R}_N + 
\frac1 N\sum_{i=1}^NL_{R_1}(x_i)\bigg\}=0,\\
&V^{\varepsilon,R}_N(T,x_1,\ldots,x_N)=\frac1N\sum_{i=1}^NU_{R_1}(x_i),
\end{aligned}\right.
\end{align}
where $\Theta_{R_2}:=\Theta\cap B_{R_2}(0)$ and for $\varphi\in\{L,U\}$ we have defined the smooth truncated version $\varphi_{R_1}$ satisfying
\begin{enumerate}
  \item $\varphi_{R_1}(x)=\varphi(x)$ on $x\in B_{R_1}(x)$, $|\varphi_{R_1}(x)|\leq|\varphi(x)|$;
  \item $\varphi_{R_1},\ \nabla_x\varphi_{R_1},\ \nabla^2_x\varphi_{R_1}$ are bounded;
  \item The derivatives satisfy
  \begin{align}\label{trun-deri}
    |\varphi'_{R_1}(x)|\leq C^\varphi_{11}|x|+C^\varphi_{10},\ |\varphi''_{R_1}(x)|\leq C^\varphi_{20}.
  \end{align}
\end{enumerate}
These derivative bounds and growth rates on the truncated functions can be guaranteed because of the main assumptions on $L,U$, which we imposed in Assumption \ref{assumption-1}.

The equation above corresponds to the auxiliary optimization problem with the underlying training processes
\begin{align}\label{sample-dyn-1}
  &dX^{\varepsilon,\theta,i}_N(t) = f\bigg(t, \theta(t),X^{\varepsilon,\theta,i}_N(t),\frac{1}{N}\sum_{j=1}^N\delta_{X^{\varepsilon,\theta,j}_N(t)}\bigg)dt+\varepsilon dW^i(t)+\sigma dW^0(t),\ i=1,\ldots,N,
\end{align}
and the admissible set $\mathcal U^{ad}_{s,R_2}$ consists of $\theta\in\mathcal U^{ad}_s$ with $|\theta(t)|\leq R_2$, $t\in[s,T]$, as well as the objective function $J^{\varepsilon,R_1}_N(s,\cdot): \mathcal U^{ad}_{s,R_2}\times\mathbb R^N\to\mathbb R$ is defined as
\begin{align*}
  J^{\varepsilon,R_1}_N(s,\theta,x_1,\ldots,x_N):=\mathbb E\bigg[\frac1 N\sum_{i=1}^N\int_s^TL_{R_1}\big(X^{ \varepsilon,\theta,i}_N(t)\big)dt+\frac1N\sum_{i=1}^NU_{R_1}\big(X^{ \varepsilon,\theta,i}_N(T)\big)+\frac\lambda2\int_s^T\left|\theta(t)\right|^2dt\bigg].
\end{align*}
Suppose that Assumption \ref{assumption-1} takes place. Then we have
\begin{align}\label{def-V-N}
V^{\varepsilon,R}_N(s,x_1,\ldots,x_N)=\inf_{\theta\in\mathcal U^{ad}_{s,R_2}}J^{\varepsilon,R_1}_N(s,\theta,x_1,\ldots,x_N).
\end{align}
{ According to \eqref{obj-func-2}, we may focus on the training processes satisfying
\begin{align*}
 \mathbb E\bigg[\int_s^T|\theta(t)|^2dt\bigg]<+\infty.
\end{align*}
In view of the variational representation, the estimates in Lemma \ref{moment-X-i} below, as well as the dominated convergence theorem,}  it is straightforward to show the following convergences
\begin{align}
\label{approximation-1}&\lim_{R_2\to+\infty}V^{\varepsilon,R}_N(s,x_1,\ldots,x_N)=\inf_{\theta\in\mathcal U^{ad}_s}J^{\varepsilon,R_1}_N(s,\theta,x_1,\ldots,x_N)=:V^{\varepsilon,R_1}_N(s,x_1,\ldots,x_N),
\end{align}
For the training processes in \eqref{sample-dyn-1} we introduce another objective functional $J^{\varepsilon}_N(s,\cdot): \mathcal U^{ad}_s\times\mathbb R^N\to\mathbb R$ { and related value function} defined as
\begin{align*}
 &J^\varepsilon_N(s,\theta,x_1,\ldots,x_N):={\mathbb E\bigg[\frac1 N\sum_{i=1}^N\int_s^TL\big(X^{\varepsilon,\theta,i}_N(t)\big)dt+\frac1N\sum_{i=1}^NU\big(X^{\varepsilon,\theta,i}_N(T)\big)+\frac\lambda2\int_s^T\left|\theta(t)\right|^2dt\bigg]},\notag\\
  &\quad{ V^\varepsilon_N(s,x_1,\ldots,x_N):=\inf_{\theta\in\mathcal U^{ad}_s}J^\varepsilon_N(s,\theta,x_1,\ldots,x_N).}
\end{align*}
After some modification of standard results on parabolic PDEs (that we detail below), we can show that the HJB equation \eqref{HJB-V-N-1} admits a solution {$V^{\varepsilon,R}_N\in C^{1+\frac\gamma2,2+\gamma}_{loc}\left([0,T)\times\mathbb R^N\right)\cap C\left([0,T]\times\mathbb R^N\right)$}. In this section, we establish uniform estimates on $V^{\varepsilon,R}_N$ and its first order derivatives, especially uniform in $(\varepsilon,N)$. Different from the usual PDE estimates, the estimates here are  focused more on the dimension of variables since the dimension, {which corresponds to the number of samples}, is now changing. We begin with the existence and uniqueness of classical solution to \eqref{HJB-V-N-1}.

\begin{lemma}\label{wp-epsln}
  The HJB equation \eqref{HJB-V-N-1} admits a unique bounded solution $V^{\varepsilon,R}_N\in C^{1+\frac\gamma2,4+\gamma}_{loc}\left([0,T)\times\mathbb R^N\right)\cap C\left([0,T]\times\mathbb R^N\right)$ where $0<\gamma<1$ and $\partial_tV^{\varepsilon,R}_N$, $\partial_{x_i}V^{\varepsilon,R}_N$, $\partial^2_{x_ix_j}V^{\varepsilon,R}_N$, $1\leq i,j\leq N$ are bounded.
\end{lemma}
\begin{proof}
  Notice that $L_{R_1}$ and $U_{R_1}$ as well as their derivatives are all bounded. According to Theorem 4.4.3, Theorem 4.7.2 and Theorem 4.7.4 in \cite{Krylov1980}, the value function $V^{\varepsilon,R}_N$ defined in \eqref{def-V-N} is the weak solution (in the distributional sense) to \eqref{HJB-V-N-1}, furthermore, $V^{\varepsilon,R}_N$ and its weak derivatives $\partial_tV^{\varepsilon,R}_N$, $\partial_{x_i}V^{\varepsilon,R}_N$, $\partial^2_{x_ix_j}V^{\varepsilon,R}_N$, $1\leq i,j\leq N$ are all bounded. Note also that for $(t,x)\in(0,T)\times\mathbb R^N$
  \begin{align}\label{regularity}
    \partial_tV^{\varepsilon,R}_N(t,x)+\frac{\sigma^2}2\sum_{i,j=1}^N\partial^2_{x_ix_j}V^{\varepsilon,R}_N(t,x)+\frac{\varepsilon^2}2\sum_{i=1}^N\partial^2_{x_ix_i}V^{\varepsilon,R}_N(t,x)=g(t,x),
  \end{align}
  where
  \begin{align*}
  g(t,x):=-\inf_{\theta\in\Theta_{R_2}}\bigg\{\frac\lambda2\big|\theta\big|^2+\sum_{i=1}^Nf\bigg(t,\theta,x_i,\frac1N\sum_{j=1}^N\delta_{x_j}\bigg)\partial_{x_i}V^{\varepsilon,R}_N+\frac1 N\sum_{i=1}^NL_{R_1}(x_i)\bigg\}.
\end{align*}
As is shown above, $\partial^2_{x_ix_j}V^{\varepsilon,R}_N$ is bounded. Moreover, we have from Corollary 4.7.8 of \cite{Krylov1980} that for $0<\gamma<1$, $\nabla_xV^{\varepsilon,R}_N(t,x)$ is $\frac\gamma2$-H\"older with respect to $t$ (uniformly in $x$). Hence $g(t,x)$ is locally Lipschitz continuous with respect to $x$ and H\"older continuous with respect to $t$. Let us view $V^{\varepsilon,R}_N$ as the solution to PDE \eqref{regularity} with constant coefficients, where the terminal conditions are the same as \eqref{HJB-V-N-1}. Standard results then yield that $\partial_tV^{\varepsilon,R}_N$, $\partial_{x_i}V^{\varepsilon,R}_N$, $\partial^2_{x_ix_j}V^{\varepsilon,R}_N\in C^{\frac\gamma2,\gamma}_{loc}\left([0,T)\times\mathbb R^N\right)$, $1\leq i,j\leq N$.

  As for the uniqueness, we can use the stochastic control interpretation to \eqref{HJB-V-N-1} and show that any solution $V^{\varepsilon,R}_N$ equals the value function in \eqref{def-V-N} by the standard verification results.
\end{proof}
Notice that at the moment the bound on $V^{\varepsilon,R}_N,\ \partial_{x_i}V^{\varepsilon,R}_N$, $\partial^2_{x_ix_j}V^{\varepsilon,R}_N$, $1\leq i,j\leq N$ might depend on {$\varepsilon$, $R$ and $N$}. Before establishing uniform estimates with respect to {$\varepsilon$, $R$ and $N$}, we need a refined analysis on the sample path, where the constants are independent of the dimension $N$ of $x=(x_1,\ldots,x_N)\in\mathbb R^N$ and the estimates admits a term of averaging due to the mean field effect.
\begin{lemma}\label{moment-X-i}
  Let $x_i\in\mathbb R$, $1\leq i\leq N$, $\theta\in\mathcal U^{ad}_{t_0,R_2}$ and $X^{\varepsilon,\theta,i}_N,$ $i=1,\dots,N$, be the {associated sample path in \eqref{sample-dyn-1}}. Then there exists a constant {$\tilde C_1=\tilde C_1(f,T)$ (depending only on $f,T$, independent of $N,\varepsilon,\sigma,R_1,R_2$)}, increasing in $T$, such that
\begin{align}\label{moment-X-i-1}
    {\mathbb E|X^{\varepsilon,\theta,i}_N(t)|^2}\leq\tilde C_1\bigg(1+|x_i|^2+\mathbb E\int_{t_0}^T\big|f\big(s,\theta(s),0,\delta_{\{0\}}\big)\big|^2ds+\frac1N\sum_{j=1}^N|x_j|^2\bigg).
\end{align}
{ And
\begin{align}
 \label{approximation-1-1}&\lim_{R_1\to+\infty}V^{\varepsilon,R_1}_N(t,x_1,\ldots,x_N)=V^\varepsilon_N(t,x_1,\ldots,x_N),\\
 \label{approximation-2}&\lim_{\varepsilon\to0}V^\varepsilon_N(t,x_1,\ldots,x_N)=V_N(t,x_1,\ldots,x_N).
\end{align}}
\end{lemma}
\begin{proof}
For $x_1,\ldots,x_N\in\mathbb R$, denote by
\begin{align*}
&\tilde f_i\big(s,\theta,x_1,\ldots,x_N\big):=f\bigg(s,\theta,x_i,{\frac{1}{N}\sum_{k=1}^N\delta_{x_k}}\bigg),
\end{align*}
then
\begin{align*}
  &\partial_{x_j}\tilde f_i\big(s,\theta,x_1,\ldots,x_N\big)=\delta_{ij}f_x\bigg(s,\theta,x_i,{\frac{1}{N}\sum_{k=1}^N\delta_{x_k}}\bigg)+{\frac1N\partial_\mu f\bigg(s,\theta,x_i,\frac{1}{N}\sum_{k=1}^N\delta_{x_k},x_j\bigg)},
\end{align*}
where $\delta_{ij}$ stands for the Kronecker symbol. { Define the set $M_N(C)\subset\mathbb R^{N\times N}$ with a parameter $C>0$ where
\begin{align}\label{definition-MNC}
  A\in M_N(C)\quad\text{if and only if}\quad|A_{ij}|\leq C(\delta_{ij}+N^{-1}),\ 1\leq i,j\leq N.
\end{align}
In view of the above and the mean value theorem, there exist a constant $C_1=C_1(f)$ and matrix process $\big(A_N(s)\big)_{s\in[0,T]}$ satisfying
\begin{align*}
 A_N(s)\in M_N(C_1),\ \big|(A_N(s))_{ij}\big|\leq C_1(\delta_{ij}+N^{-1}),\ 1\leq i,j\leq N,\ s\in[0,T],
\end{align*}
such that
\begin{align*}
&{X^{\varepsilon,\theta}_N(t)=x+\int_{t_0}^tf\big(s,\theta(s),0,\delta_{\{0\}}\big)\mathbf 1ds+\int_{t_0}^tA_N(s)X^{\varepsilon,\theta}_N(s)ds}\notag\\
&\quad+\varepsilon\mathbf W^N(t)+\sigma\mathbf1W^0(t),
\end{align*}
where
\begin{align*}
\mathbf W^N(t):=(W^1(t),\ldots,W^N(t))^\top, \mathbf1:=(1,\ldots,1)^\top.
\end{align*}
Solving the linear SDE above, we have}
\begin{align}\label{X-rep}
    &X^{{ \varepsilon,\theta},i}_N(t)=\big(\Phi^+_N(t)x\big)_i+\int_{t_0}^tf\big(s,\theta(s),0,\delta_{\{0\}}\big)\big(\Phi^+_N(t)\Phi^-_N(s)\mathbf 1\big)_ids\notag\\
    &+\varepsilon\bigg(\int_{t_0}^t\Phi^+_N(t)\Phi^-_N(s)d\mathbf W^N(s)\bigg)_i+\sigma\int_{t_0}^t\big(\Phi^+_N(t)\Phi^-_N(s)\mathbf 1\big)_idW^0(s),
\end{align}
where the matrix valued processes $\Phi^{\pm}_N(s)$ solve
\begin{align*}
\Phi^+_N(t)&=I_N+\int_{t_0}^tA_N(s)\Phi^+_N(s)ds,\quad\Phi^-_N(t)=I_N-\int_{t_0}^t\Phi^-_N(s)A_N(s)ds.
\end{align*}
Note that
\begin{align*}
 \frac d{dt}\left[\Phi^-_N(t)\Phi^+_N(t)\right]=0\quad\text{and}\quad\Phi^-_N(t_0)\Phi^+_N(t_0)=I_N,
\end{align*}
thus $\Phi^-_N(t)\Phi^+_N(t)=\Phi^-_N(t_0)\Phi^+_N(t_0)=I_N$.

  According to Lemma \ref{forward-scaling-eqn}, $\Phi^{\pm}_N(s)\in M_N(C_2)$ for some $C_2=C_2(f,T)$ because {$A_N(s)\in M_N(C_1)$}. Moreover, $\Phi^+_N(t)\Phi^-_N(s)\in M_N(C_2)$ due to Lemma \ref{multiply-close}. An application of Burkholder--Davis--Gundy inequality (see e.g. \cite{Yong1999}) to the $i$-th component in \eqref{X-rep} gives the estimate \eqref{moment-X-i-1}.
  
  Next we turn to \eqref{approximation-1-1}. Denote by $\theta^\delta$ the $\delta$-optimal strategy for $V^\varepsilon_N(t,x_1,\ldots,x_N)$. In view of \eqref{moment-X-i-1}, the growth condition on $L,\ U$, and the dominated convergence theorem, we obtain
  \begin{align*}
  \limsup_{R_1\to+\infty}V^{\varepsilon,R_1}_N(t,x_1,\ldots,x_N)\leq\lim_{R_1\to+\infty}J^{\varepsilon,R_1}_N(t,\theta^\delta,x_1,\ldots,x_N)=J^{\varepsilon}_N(t,\theta,x_1,\ldots,x_N)\leq V^\varepsilon_N(t,x_1,\ldots,x_N)+\delta.
  \end{align*}
Hence $\limsup_{R_1\to+\infty}V^{\varepsilon,R_1}_N(t,x_1,\ldots,x_N)\leq V^\varepsilon_N(t,x_1,\ldots,x_N).$ It remains to show the inverse inequality.

Denote by $\theta^{R_1,\delta}$ the $\delta$-optimal strategy for $V^{\varepsilon,R_1}_N(t,x_1,\ldots,x_N)$, in view of \eqref{moment-X-i-1} and the growth conditions in Assumption \ref{assumption-1},
\begin{align*}
 &\quad J^{\varepsilon,R_1}_N(t_0,\theta^{R_1,\delta},x_1,\ldots,x_N)\\
 &\geq\mathbb E\bigg[-\frac CN\sum_{i=1}^N\int_{t_0}^T\big|X^{\varepsilon,\theta^{R_1,\delta},i}_N(t)\big|^{2-\alpha}dt-\frac CN\sum_{i=1}^N\big|X^{\varepsilon,\theta^{R_1,\delta},i}_N(T)\big|^{2-\alpha}-C+\lambda\int_{t_0}^T|\theta^{R_1,\delta}(t)|^2dt\bigg]\\
 &\geq-C\bigg(1+|x_i|^2+\mathbb E\int_{t_0}^T\big|f\big(s,\theta^{R_1,\delta}(s),0,\delta_{\{0\}}\big)\big|^2ds+\frac1N\sum_{j=1}^N|x_j|^2\bigg)^{\frac{2-\alpha}2}+\mathbb E\bigg[\lambda\int_{t_0}^T|\theta^{R_1,\delta}(t)|^2dt\bigg]\\
&\geq-\mathbb E\bigg[C\int_{t_0}^T|\theta^{R_1,\delta}(t)|^2dt\bigg]^{\frac{2-\alpha}2}-C_x+\mathbb E\bigg[\lambda\int_{t_0}^T|\theta^{R_1,\delta}(t)|^2dt\bigg]\geq\frac\lambda2\mathbb E\bigg[\int_{t_0}^T|\theta^{R_1,\delta}(t)|^2dt\bigg]-\hat C_x,
\end{align*}
 where the first inequality in the last line above is obtained via the Taylor's expansion, for the last inequality we have used Young's inequality  and $C_x,\ \hat C_x$ are positive constants depending on $x=(x_1,\ldots,x_N)$. Since $\theta^{R_1,\delta}$ is $\delta$-optimal, $V^{\varepsilon,R_1}_N(t,x_1,\ldots,x_N)$ is bounded below uniformly in $R_1$ by $-\hat C_x+\delta$, and we pick a subsequence along which the limit $\lim_{R_1\to+\infty}V^{\varepsilon,R_1}_N(t,x_1,\ldots,x_N)$ exists. Hence 
 \begin{align*}
  &\quad\limsup_{R_1\to+\infty}\mathbb E\bigg[\int_{t_0}^T|\theta^{R_1,\delta}(t)|^2dt\bigg]\leq2\lambda^{-1}\limsup_{R_1\to+\infty}\big(J^{\varepsilon,R_1}_N(t_0,\theta^{R_1,\delta},x_1,\ldots,x_N)+\hat C_x\big)\\
  &\leq2\lambda^{-1}\lim_{R_1\to+\infty}\big(V^{\varepsilon,R_1}_N(t_0,x_1,\ldots,x_N)+\delta+\hat C_x\big).
 \end{align*}
In view of the above estimates, \eqref{moment-X-i-1} as well as the growth conditions on $L,\ U$, we get that the families $\big(L_{R_1}(X^{\varepsilon,\theta^{R_1,\delta},i}_N(t))\big)_{0\leq\varepsilon<1,R_1>0}$ and $\big(U_{R_1}(X^{\varepsilon,\theta^{R_1,\delta},i}_N(t))\big)_{0\leq\varepsilon<1,R_1>0}$ are both uniformly integrable. Hence
\begin{align*}
 \lim_{R_1\to+\infty}\mathbb E\bigg[\frac1 N\sum_{i=1}^N\int_s^T|L_{R_1}-L|\big(X^{\varepsilon,\theta^{R_1,\delta},i}_N(t)\big)dt+\frac1N\sum_{i=1}^N|U_{R_1}-U|\big(X^{\varepsilon,\theta^{R_1,\delta},i}_N(T)\big)\bigg]=0.
\end{align*}
Therefore
\begin{align*}
  &\quad\delta+\lim_{R_1\to+\infty}V^{\varepsilon,R_1}_N(t,x_1,\ldots,x_N)\geq\liminf_{R_1\to+\infty}J^{\varepsilon,R_1}_N(s,\theta^{R_1,\delta},x_1,\ldots,x_N)=\liminf_{R_1\to+\infty}J^{\varepsilon}_N(s,\theta^{R_1,\delta},x_1,\ldots,x_N)\\
  &\geq V^\varepsilon_N(t,x_1,\ldots,x_N).
  \end{align*}
Therefore, altogether we have obtained
\begin{align*}
V^\varepsilon_N(t,x_1,\ldots,x_N) \le \delta+\lim_{R_1\to+\infty}V^{\varepsilon,R_1}_N(t,x_1,\ldots,x_N) \le \delta+ V^\varepsilon_N(t,x_1,\ldots,x_N).
\end{align*}
Thus, by the arbitrariness of $\delta$ we obtained 
$$
\lim_{R_1\to+\infty}V^{\varepsilon,R_1}_N(t,x_1,\ldots,x_N) = V^\varepsilon_N(t,x_1,\ldots,x_N),
$$
for any convergent subsequence. So, the convergence does not depend on the particular subsequence ant it must take place for the full family $(V^{\varepsilon,R_1}_N(t,x_1,\ldots,x_N))_{R_1>0}$.

  Hence \eqref{approximation-1-1} follows. The equality \eqref{approximation-2} can be obtained using the similar approach and estimates to the above.
\end{proof}

\begin{remark}
Take $\theta(t)\equiv0$ (which is admissible since $0\in\Theta$), then \eqref{moment-X-i-1} can be rephrased as
\begin{align}\label{moment-X-i-2}
{\mathbb E|X^{\varepsilon,0,i}_N(t)|^2}\leq\tilde C_1\bigg(1+|x_i|^2+\frac1N\sum_{j=1}^N|x_j|^2\bigg).
\end{align}
\end{remark}
Based on Lemma \ref{moment-X-i}, we can go on with the estimates on the first derivatives. {In the context below, the values of constants {$C_k$, $\tilde C_k$, $k\geq1$,} might vary, but their dependence on the model parameters remains the same.}

For $(t,p,q,z,\theta)\in[0,T]\times\mathbb R^N\times\mathbb R^N\times\mathbb R^{N\times N}\times\mathbb R^N$, define the following Hamiltonian
\begin{align}
  \label{H-N-R-0}H^{R_1}_N(t,x,p,\theta)&:=\lambda\big|\theta\big|^2+\sum_{i=1}^Nf\bigg(t,\theta,x_i,{\frac1N\sum_{j=1}^N\delta_{x_j}}\bigg)p_i+\frac1 N\sum_{i=1}^NL_{R_1}(x_i),
\end{align}
as well as, for later use,
\begin{align}
  \label{H-N-R-1}H_N(t,x,p,\theta)&:=\lambda\big|\theta\big|^2+\sum_{i=1}^Nf\bigg(t,\theta,x_i,{\frac1N\sum_{j=1}^N\delta_{x_j}}\bigg)p_i+\frac1 N\sum_{i=1}^NL(x_i).
\end{align}
With the preparation above, we show the following estimates on $V^{\varepsilon,R}_N$ in \eqref{HJB-V-N-1}.
\begin{lemma}\label{1st-deri}
  Let $x_i\in\mathbb R$, { $t\in[0,T],$} $1\leq i\leq N$ {and $V^{\varepsilon,R}_N$ be the solution to \eqref{HJB-V-N-1}}. Then there exists a constant $\tilde C_2=\tilde C_2(f,\lambda^{-\frac12},T)$, increasing in $T,\lambda^{-\frac12}$, independent of {$N$, $\sigma$, $\varepsilon$ and $R$} such that for $1\leq i\leq N,$
{\begin{align}\label{1st-deri-scaling-1}
    \big|\partial_{x_i}V^{\varepsilon,R}_N(t,x)\big|&\leq\frac{\tilde C_2(C^L_{11}+C^U_{11})}N\bigg(1+|x_i|^2+\frac1N\sum_{j=1}^N|x_j|^2\bigg)^\frac12+\frac{\tilde C_2(C^L_{10}+C^U_{10})}N.
\end{align}
}
\end{lemma}

\begin{proof}
We begin by showing the existence of a constant $\hat C_2=\hat C_2(f,T)$ such that { for $(t,x)\in[0,T]\times\mathbb R^N,$}
\begin{align}\label{1st-deri-scaling}
    \left|\partial_{x_i}V^{\varepsilon,R}_N(t,x)\right|&\leq\frac{\hat C_2(C^L_{11}+C^U_{11})}N\bigg(1+|x_i|^2+{\mathbb E\bigg[\int_0^T|f\big(s,\theta^*(s),0,\delta_{\{0\}}\big)|ds}\bigg]+\frac1N\sum_{j=1}^N|x_j|^2\bigg)^{\frac12}\notag\\
&\quad+\frac{\hat C_2(C^L_{10}+C^U_{10})}N,
\end{align}
where $\theta^*$ is the optimal control process.

  It suffices to show the existence of such $\hat C_2$ for $t=0$. For other $t\in[0,T]$ the proof and $\hat C_2$ can be deduced in the same way.

  In view of Lemma \ref{wp-epsln}, the HJB equation \eqref{HJB-V-N-1} admits a classical solution. Following the standard verification procedure (see e.g. \cite{Fleming2006}), one can show the existence of an optimal control (in the weak sense) and the corresponding optimal path. Hence we may denote by $\theta^*(t)$ and $(X^*_N(t),Y^*_N(t))$ the optimal control, optimal path as well as the adjoint process. According to the stochastic maximum principle { (see Theorem 3.2 in \cite[Chapter 3]{Yong1999})}, { the optimal state process together with the adjoint equation satisfy the system} 
\begin{align}\label{ad-eqn}\left\{\begin{aligned}
dY^{*,i}_N(t)&=-\partial_{x_i}H^{R_1}_N\left(t,X^*_N(t),Y^*_N(t),\theta^*(t)\right)dt+\sum_{j=0}^NZ^{ij}_N(t)dW^j(t),\\
dX^{*,i}_N(t)&=\partial_{p_i}H^{R_1}_N\left(t,X^*_N(t),Y^*_N(t),\theta^*(t)\right)dt+\varepsilon dW^i(t)+\sigma dW^0(t),\\
X^*_N(0)&=x,\quad Y^{*,i}_N(T)=\frac1NU'_{R_1}\left(X^{*,i}_N(T)\right).\end{aligned}\right.
\end{align}
Here $Y^*_N,\ X^*_N\in\mathbb R^N$, $Z_N\in\mathbb R^{N\times N}$, and recall that $H^{R_1}_N(t,x,p,\theta)$ is given in \eqref{H-N-R-0}. In view of the relationship between the maximum principle and HJB equation for stochastic optimal control problems (see { Theorem 3.1 in} \cite[Chapter 5]{Yong1999}), we have $Y^{*,i}_N(0)=\partial_{x_i}V^{\varepsilon,R}_N({ 0},x),\ 1\leq i\leq N.$

Rewrite \eqref{ad-eqn} in the following manner:
\begin{align}\label{ad-eqn-1}\left\{\begin{aligned}
dY^{*,i}_N(t)&=-\bigg[\sum_{j=1}^NA^N_{ij}(t)Y^{*,j}_N(t)+\frac1NL'_{R_1}(X^{*,i}_N(t))\bigg]dt+\sum_{j=0}^NZ^{ij}_N(t)dW^j(t)\\
  dX^{*,i}_N(t)&=f\bigg(t,\theta^*(t),X^{*,i}_N(t),\frac1N\sum_{j=1}^N\delta_{X^{*,j}_N(t)}\bigg)dt+\varepsilon dW^i(t)+\sigma dW^0(t),\\
X^*_N(0)&=x,\quad Y^{*,i}_N(T)=\frac1NU'_{R_1}\big(X^{*,i}_N(T)\big),\end{aligned}\right.
\end{align}
where for $1\leq i,j\leq N$,
\begin{align}\label{def-AN}
  A^N_{ij}(t)&:=\delta_{ij}f_x\bigg(t,\theta^*(t),X^{*,j}_N(t),\frac1N\sum_{k=1}^N\delta_{X^{*,k}_N(t)}\bigg)+{\frac1N\partial_\mu f\bigg(t,\theta^*(t),X^{*,j}_N(t),\frac1N\sum_{k=1}^N\delta_{X^{*,k}_N(t)},X^{*,i}_N(t)\bigg)}.
\end{align}
Consider the matrix valued processes $\Phi^{\pm}_N(t)\in\mathbb R^{N\times N}$ solving
\begin{align}
\label{Phi-2}\Phi^+_N(t)&=I_N-\int_0^tA^N(s)\Phi^+_N(s)ds,\\
\label{Phi-3}\Phi^-_N(t)&=I_N+\int_0^t\Phi^-_N(s)A^N(s)ds.
\end{align}
Then
\begin{align}\label{discount-dyn}
  \Phi^-_N(t)Y^*_N(t)=\frac1N\mathbb E_t\big[\Phi^-_N(T)U'_{R_1}\big(X^*_N(T)\big)\big]+\frac1N\mathbb E_t\bigg[\int_t^T\Phi^-_N(s)L'_{R_1}\big(X^*_N(s)\big)ds\bigg],
\end{align}
where
\begin{align*}
  &U'_{R_1}(X^*_N(T)):=\left(U'_{R_1}(X^{*,1}_N(T)),\ldots,U'_{R_1}(X^{*,N}_N(T))\right)^\top,\notag\\ &L'_{R_1}(X^*_N(s)):=\left(L'_{R_1}(X^{*,1}_N(s)),\ldots,L'_{R_1}(X^{*,N}_N(s))\right)^\top.
\end{align*}
In particular,
\begin{align*}
  Y^*_N(0)=\frac1N\mathbb E\left[\Phi^-_N(T)U'\left(X^*_N(T)\right)\right]+\frac1N\mathbb E\left[\int_0^T\Phi^-_N(s)L'_{R_1}\left(X^*_N(s)\right)ds\right],
\end{align*}
and thus
\begin{align}\label{QN0}
  Y^{*,i}_N(0)^2\leq\frac2{N^2}\big|\mathbb E\big(\Phi^-_N(T)U'_{R_1}(X^*_N(T))\big)_i\big|^2+{\frac{2T}{N^2}\int_0^T\big|\mathbb E\big(\Phi^-_N(s)L'_{R_1}(X^*_N(s))\big)_i\big|^2ds}.
\end{align}
According to \eqref{def-AN}, $A_N\in M_N(C_1)$ with $C_1=C_1(f)$. In view of Lemma \ref{forward-scaling-eqn} and \eqref{Phi-3}, for $A_N(t)\in M_N(C_1)$, it follows that
{\begin{align*}%\label{Phi-bound}
\mathbb E\Phi^-_N(s)\in M_N(C_2),\ C_2=C_2(f,T),\ s\in[0,T].
\end{align*}}
Note \eqref{assumption-L-U} and \eqref{trun-deri}, for the $i$-th entry of $\Phi^-_N(T)U'_{R_1}\left(X^*_N(T)\right)$:
\begin{align}\label{QN0-1}
&\quad\big|\mathbb E\big(\Phi^-_N(T)U'_{R_1}(X^*_N(T))\big)_i\big|^2\notag\\
  &\leq2\mathbb E\bigg[\big(\Phi^-_N(T)\big)^2_{ii}\bigg]\mathbb E\bigg[U'_{R_1}\big(X^{*,i}_N(T)\big)^2\bigg]+2\mathbb E\bigg[\sum_{\substack{j=1\\j\neq i}}^N\big(\Phi^-_N(T)\big)^2_{ij}\bigg]\mathbb E\bigg[\sum_{\substack{j=1\\j\neq i}}^NU'_{R_1}\big(X^{*,j}_N(T)\big)^2\bigg]\notag\\
  &\leq C_2\mathbb E\bigg[U'_{R_1}\big(X^{*,i}_N(T)\big)^2\bigg]+\frac{C_2}N\mathbb E\bigg[\sum_{\substack{j=1\\j\neq i}}^NU'_{R_1}\big(X^{*,j}_N(T)\big)^2\bigg]\notag\\
  &\leq C_2(C^U_{11})^2\bigg(1+\mathbb E\big[|X^{*,i}_N(T)|^2\big]+\frac1N\sum_{j=1}^N\mathbb E\big[|X^{*,j}_N(T)|^2\big]\bigg)+C_2(C^U_{10}+1)^2.
\end{align}
Similarly,
\begin{align}\label{QN0-2}
&\quad\big|\mathbb E\big(\Phi^-_N(s)L'_{R_1}(X^*_N(s))\big)_i\big|^2\notag\\
  &\leq C_2(C^U_{11})^2\bigg(1+\mathbb E\big[|X^{*,i}_N(s)|^2\big]+\frac1N\sum_{j=1}^N\mathbb E\big[|X^{*,j}_N(s)|^2\big]\bigg)+C_2(C^U_{10}+1)^2.
\end{align}
In view of Lemma \ref{moment-X-i},
\begin{align}\label{moment-X-star}
  \mathbb E\big[\big|X^{*,i}_N(t)\big|\big]\leq\tilde C_1\bigg(|x_i|^2+\mathbb E\int_0^t{|f\big(s,\theta^*(s),0,\delta_{\{0\}}\big)|}ds+\frac1N\sum_{j=1}^N|x_j|^2\bigg).
\end{align}
Plugging \eqref{QN0-1}, \eqref{QN0-2} and \eqref{moment-X-star} into \eqref{QN0}, we obtain \eqref{1st-deri-scaling}.

To further prove \eqref{1st-deri-scaling-1}, it suffices to prove that there exist a constant $\check C=\check C(f,\lambda^{-\frac12},T)$ (increasing in $\lambda^{-\frac12}$) such that
\begin{align*}
\mathbb E\int_0^t|f\big(s,\theta^*(s),0,\delta_{\{0\}}\big)|ds\leq\check C\bigg(1+\frac1N\sum_{i=1}^N|x_i|^2\bigg)^{\frac12}.
\end{align*}
In fact,
\begin{align*}
  \mathbb E\int_0^t|f\big(s,\theta^*(s),0,\delta_{\{0\}}\big)|ds\leq\int_0^t|f\big(s,0,0,\delta_{\{0\}}\big)|ds+\|f_\theta\|_\infty\mathbb E\int_0^t|\theta^*(s)|ds.
\end{align*}
And we notice that
\begin{align}\label{uni-est-theta}
  \mathbb E\int_0^t|\theta^*(s)|ds&\leq T^\frac12\mathbb E\bigg(\int_0^T|\theta^*(s)|^2ds\bigg)^\frac12\leq (2T)^\frac12\lambda^{-\frac12}J_N(\theta^*,0,x_1,\ldots,x_N)^{\frac12}\notag\\
  &\leq (2T)^\frac12\lambda^{-\frac12}J_N(0,0,x_1,\ldots,x_N)^{\frac12}\leq\check C\bigg(1+\frac1N\sum_{i=1}^N|x_i|^2\bigg)^{\frac12},
\end{align}
where the last inequality holds because of \eqref{moment-X-i-2}. Hence we may deduce \eqref{1st-deri-scaling-1} from the estimates above.
\end{proof}
{The next lemma shows that}, thanks to Lemma \ref{1st-deri}, we may drop the {parameter $R_2$ in \eqref{HJB-V-N-1}} and consider
\begin{align}\label{glb-1}
\left\{\begin{aligned}&\partial_tV^{\varepsilon,R_1}_N+\frac{\sigma^2}2\sum_{i,j=1}^N\partial^2_{x_ix_j}V^{\varepsilon,R_1}_N+\frac{\varepsilon^2}2\sum_{i=1}^N\partial^2_{x_ix_i}V^{\varepsilon,R_1}_N+\tilde H^{R_1}_N(t,x,\nabla_xV^{\varepsilon,R_1}_N)=0,\\
&V^{\varepsilon,R_1}_N(T,x_1,\ldots,x_N)=\frac1N\sum_{i=1}^NU_{R_1}(x_i),
\end{aligned}\right.
\end{align}
where for $(t,x,p)\in[0,T]\times\mathbb R^N\times\mathbb R^N$,
\begin{align}\label{glb-H}
  \tilde H^{R_1}_N(t,x,p):=\inf_{\theta\in\mathbb R}\bigg\{\frac\lambda2\big|\theta\big|^2+\sum_{i=1}^Nf\bigg(t,\theta,x_i,\frac1N\sum_{j=1}^N\delta_{x_j}\bigg)p_i+\frac1 N\sum_{i=1}^NL_{R_1}(x_i)\bigg\}.
\end{align}
\begin{lemma}\label{1st-deri-2-0}
The equation \eqref{glb-1} admits a unique classical solution $V^{\varepsilon,R_1}_N\in C^{1+\frac\gamma2,2+\gamma}_{loc}\left([0,T)\times\mathbb R^N\right)\cap C\left([0,T]\times\mathbb R^N\right)$ where for $0<\gamma<1$ and $V^{\varepsilon,R_1}_N$, $\partial_tV^{\varepsilon,R_1}_N$, $\partial_{x_i}V^{\varepsilon,R_1}_N$, $\partial^2_{x_ix_j}V^{\varepsilon,R_1}_N$, $1\leq i,j\leq N$ are bounded. Moreover, the derivatives $\partial_{x_i}V^{\varepsilon,R_1}_N$, $1\leq i\leq N$, satisfy
\begin{align}\label{1st-deri-scaling-2-0}
    \big|\partial_{x_i}V^{\varepsilon,R_1}_N(t,x)\big|&\leq\frac{\tilde C_2(C^L_{11}+C^U_{11})}N\bigg(1+|x_i|^2+\frac1N\sum_{j=1}^N|x_j|^2\bigg)^\frac12+\frac{\tilde C_2(C^L_{10}+C^U_{10})}N,
\end{align}
where the constant $\tilde C_2$ is from Lemma \ref{1st-deri}.
\end{lemma}
\begin{proof}
  We recall Lemma \ref{wp-epsln} saying that \eqref{HJB-V-N-1} admits classical solutions $V^{\varepsilon,R}_N$. Moreover, since $L_{R_1}$ and $U_{R_1}$ both have bounded derivatives, in view of Lemma \ref{1st-deri}, $\big|\nabla_xV^{\varepsilon,R}_N\big|$ is bounded by a constant independent of $R_2$. Therefore, for sufficiently large $R_2$, we have
  \begin{align}\label{large-R2}
    &\quad\inf_{\theta\in\Theta_{R_2}}\bigg\{\frac\lambda2\big|\theta\big|^2+\sum_{i=1}^Nf\bigg(t,\theta,x_i,\frac1N\sum_{j=1}^N\delta_{x_j}\bigg)\partial_{x_i}V^{\varepsilon,R}_N(t,x)+\frac1 N\sum_{i=1}^NL_{R_1}(x_i)\bigg\}\notag\\
    &=\inf_{\theta\in\mathbb R}\bigg\{\cdots\bigg\}=\tilde H^{R_1}_N\big(t,x,\nabla_xV^{\varepsilon,R}(t,x)\big).
  \end{align}
  In other words, for those $R_2$ satisfying \eqref{large-R2}, $V^{\varepsilon,R}_N$ solves \eqref{glb-1}. Choose an arbitrary $R_2$ such that $V^{\varepsilon,R}_N$ satisfies \eqref{large-R2} and denote it by $V^{\varepsilon,R_1}_N$. We thus have by Lemma \ref{wp-epsln} that $V^{\varepsilon,R}_N\in C^{1+\frac\gamma2,2+\gamma}_{loc}\left([0,T)\times\mathbb R^N\right)\cap C\left([0,T]\times\mathbb R^N\right)$ and $V^{\varepsilon,R}_N, \partial_tV^{\varepsilon,R}_N$, $\partial_{x_i}V^{\varepsilon,R}_N$, $\partial^2_{x_ix_j}V^{\varepsilon,R}_N$ are bounded. We can show the uniqueness via the variational arguments described in Lemma \ref{wp-epsln}. We can also obtain \eqref{1st-deri-scaling-2-0} from Lemma \ref{1st-deri} since it is satisfied by any $V^{\varepsilon,R}_N$.
\end{proof}

\subsection{The estimates on the second order derivatives}\label{sec:32}
In this section we establish uniform estimates on the second order derivatives $\partial^2_{x_ix_j}V^{\varepsilon,R_1}_N$, $1\leq i,j\leq N$ for the solution to  \eqref{glb-1} where the parameter $R_2$ has been dropped. {To do so, our idea is to formally take the derivatives with respect to $x_i$ and $x_j$ in \eqref{glb-1} and obtain the PDE system on $\partial^2_{x_ix_j}V^{\varepsilon,R_1}_N$, $1\leq i,j\leq N$. The above differentiation requires further analysis on the differentiability of Hamiltonian in \eqref{glb-1}. It then turns out that the aforementioned analysis involves the uniform estimates on the first order derivatives in \eqref{1st-deri-scaling-2-0}.} We can see from \eqref{1st-deri-scaling-2-0} that the first order derivatives therein are only locally bounded in general. In our forthcoming analysis, we propose some technical assumptions so as to deal with this non-global boundedness.

Let $\Gamma_1,\ \Gamma_2\in\mathbb R$. Denote by $\mathcal{A}_{N,\Gamma_1,\Gamma_2}$ the set consisting of $p=(p_1,\ \ldots,\ p_N)$, $x=(x_1,\ \ldots,\ x_N)\in\mathbb R^N$ satisfying
\begin{align}\label{hypo-regu-1}
  -G_N(x,\Gamma_1)\leq p_i\leq G_N(x,\Gamma_2),\quad1\leq i\leq N,
\end{align}
where { $G_N:\mathbb{R}^N\times\mathbb{R}\to\mathbb{R}$ is defined as}
\begin{align*}
 G_N(x,\Gamma):=\frac{\Gamma(C^L_{11}+C^U_{11})}N\bigg(1+|x_i|^2+\frac1N\sum_{j=1}^N |x_j|^2\bigg)^\frac12+\frac{\Gamma(C^L_{10}+C^U_{10})}N,\quad\Gamma\in\mathbb R
\end{align*}
and the constants $C_{11}^U,C_{10}^U, C_{11}^L,C_{10}^L$ are defined in Assumption \ref{assumption-1}.
In other words,
\begin{align}\label{def-set-AN}
\mathcal{A}_{N,\Gamma_1,\Gamma_2}:=\left\{(x,p)\in\mathbb R^N\times\mathbb R^N:(x,p)\ \text{satisfies}\ \eqref{hypo-regu-1}\right\}.
\end{align}

We assume that the following hold in the remaining of the paper.

\noindent{\bf Hypothesis (R).} Suppose the following
\begin{enumerate}
\item  There exists $\lambda_0>0$, such that for any $(\theta,x,p)\in\Theta\times { \mathcal{A}_{N,\Gamma_1,\Gamma_2}}$, $N\geq1$
\begin{align}\label{hypo-regu-2}
{ \lambda_0+\sum_{i=1}^N\partial^2_{\theta\theta}f\bigg(t,\theta,x_i,\frac1N\sum_{j=1}^N\delta_{x_j}\bigg)p_i\geq0.}
\end{align}
\item There exists $C^Q>0$ such that  for any $(\theta,x,p)\in\Theta\times \mathcal A_N$, $N\geq1$ and for $\varphi\in\{|\partial^2_{x\theta}f|,|\partial^2_{xx}f|\}$ and $\phi\in\{|\partial^2_{x\mu}f|,\ |\partial^2_{\theta\mu}f|\}$
\begin{align}\label{hypo-regu-3}
C^Q>\varphi\bigg(t,\theta,x_i,\frac1N\sum_{j=1}^N\delta_{x_j}\bigg)\cdot Np_i+{\phi\bigg(t,\theta,x_i,\frac1N\sum_{j=1}^N\delta_{x_j},x_i\bigg)}\cdot Np_i.
\end{align}
\item The coefficient $\lambda$ is taken such that $\lambda>\lambda_0$.
\end{enumerate}
\begin{remark}
In terms of {\bf Hypothesis (R)}, we have the following { comments:}
\begin{enumerate}
\item{ We emphasize that the hypothesis depends on $(\mathcal{A}_{N,\Gamma_1,\Gamma_2})_{N\geq1}$ with the parameters $\Gamma_1,\Gamma_2\in\mathbb R$ from \eqref{hypo-regu-1}. It is important to notice that the hypothesis eventually aims to impose a fine localized property on the set $\big\{\big(x,\nabla_xV^{\varepsilon,R_1}_N(t,x)\big):(t,x)\in[0,T]\times\mathbb R^N\big\}$.
In view of Lemma \ref{1st-deri-2-0}, for $(t,x)\in[0,T]\times\mathbb R^N$ and all $R_1>0$, $\big(x,\nabla_xV^{\varepsilon,R_1}_N(t,x)\big)\in\mathcal A_{N,\tilde C_2,\tilde C_2}$ in general. Therefore, a natural choice for $(\Gamma_1,\Gamma_2)$ is $(\tilde C_2,\tilde C_2)$. However, we will provide a class of examples when $V^{\varepsilon,R_1}_N(t,x)$ turns out to be increasing in every component of $x$, in which case we obtain a finer description of $\big(x,\nabla_xV^{\varepsilon,R_1}_N(t,x)\big)\in\mathcal A_{N,0,\tilde C_2}$, and in particular one can naturally choose $\Gamma_1 =0$. 
\item We also note that if $\partial^2_{\theta\theta}f\geq0$, then any $\lambda>0$ will fulfil {\bf Hypothesis (R)} when $(\Gamma_1,\Gamma_2)=(0,\tilde C_2)$. Lastly, such a formulation of $\mathcal{A}_{N,\Gamma_1,\Gamma_2}$ will serve for a more flexible description of the `localized' properties on $\big\{\big(x,\nabla_xV^{\varepsilon,R_1}_N(t,x)\big):(t,x)\in[0,T]\times\mathbb R^N\big\}$ in Lemma \ref{1st-deri-2-0-1} and Proposition \ref{dY/dx-1} below. }
    \item For an LQ model with uncontrolled diffusion, $\partial^2_{x\theta}f,\partial^2_{xx}f,\partial^2_{x\mu}f,\partial^2_{\theta\mu}f=0$ and \eqref{hypo-regu-2}, \eqref{hypo-regu-3} holds trivially.
    \item For $f$, $L$, $U$ with bounded derivatives, $C^L_{11}+C^U_{11}=0$ and \eqref{hypo-regu-2}, \eqref{hypo-regu-3} holds trivially.
    \item Given {\bf Hypothesis (R)} with parameters $\Gamma_1,\Gamma_2$, the lower bound of $\lambda$ can be expressed by the following:
    \begin{align*}
     \lambda>\sup_{N\geq1}\sup_{\substack{t\in[0,T],\\(\theta,x,p)\in\mathcal{A}_{N,\Gamma_1,\Gamma_2}}}\bigg\{-\sum_{i=1}^N\partial^2_{\theta\theta}f\big(t,\theta,x_i,\frac1N\sum_{j=1}^N\delta_{x_j}\big)p_i\bigg\}.
     \end{align*}
     It worth noting also that by taking the two canonical choices $(\Gamma_1,\Gamma_2)=(\tilde C_2,\tilde C_2)$ or $(\Gamma_1,\Gamma_2)=(0,\tilde C_2)$, and since $\tilde C_2$ is an increasing function of $\lambda^{-\frac12}$ and $\partial_{\theta\theta}^2f$ is uniformly bounded, the supremization problem in the previous sentence will always give the asymptotic behavior for $\lambda$ in the form of
     $$
     \lambda^{\frac32}\ge C(T,f),
     $$
    for some $C(T,f)>0$, uniformly on the set of empirical measures with uniformly bounded second moments.
\end{enumerate}
\end{remark}

Given \eqref{hypo-regu-2}, for $(x,p)\in \mathcal A_N$, the corresponding $H^{R_1}_N(t,x,p,\theta)$ in \eqref{H-N-R-0} is strictly convex in $\theta$. Hence the unique minimizer $\theta^{R_1}_N\in\Theta$ can be defined as a function of $(t,x,p)$ in such a way that
\begin{align}\label{theta-*}
\theta^{R_1}_N(t,x,p):=\arg\min_{\theta\in\Theta}H^{R_1}_N(t,x,p,\theta).
\end{align}
In light of the definition above, an optimal control $\theta^*(t)$ in \eqref{ad-eqn} can be represented as
\begin{align*}
\theta^*(t)=\theta^{R_1}_N(t,X^*_N(t),Y^*_N(t)).
\end{align*}
Thanks to {\bf Hypothesis (R)}, we can now show the Lipschitz continuity of the feedback function $\theta^{R_1}_N(t,x,p)$.
\begin{lemma}\label{theta-Lip}
  Suppose that {\bf Hypothesis (R)} {  with parameters $\Gamma_1,\ \Gamma_2$ in \eqref{hypo-regu-1} takes place}. Then $\theta^{R_1}_N(t,x,p)$ is locally Lipschitz continuous with respect to $(x,p)\in \mathcal A_{N, \Gamma_1,\Gamma_2}$ with derivatives
\begin{align}\label{theta-weak-deri}
    \big|\partial_{x_k}\theta^{R_1}_N(t,x,p)\big|\leq\frac{{(\lambda-\lambda_0)^{-1}}C^Q}N,\ \big|\partial_{p_k}\theta^{R_1}_N(t,x,p)\big|\leq{(\lambda-\lambda_0)^{-1}}\|f_\theta\|_\infty,\ k=1,\ldots,N.
\end{align}
\end{lemma}
\begin{proof}
We postpone the proof of this result to Appendix \ref{sec-App}.
\end{proof}
We have the following estimates on the coefficients based on Lemma \ref{theta-Lip}.
\begin{lemma}\label{coeff.}
   Suppose that the assumptions of Lemma \ref{theta-Lip} take place. Then there exists a constant $${\tilde C_3=\tilde C_3(f,\lambda^{\frac12},T,L,(\lambda-\lambda_0)^{-1})},$$ increasing in $T$, ${\lambda^{-\frac12},\ (\lambda-\lambda_0)^{-1}}$, such that for $(x,p)\in \mathcal A_N$,
  \begin{align}\label{hamilton-001}
   &\big|\partial_{x_i}\tilde H^{R_1}_N(t,x,p)\big|,\big|\partial^2_{x_ip_j}\tilde H^{R_1}_N(t,x,p)\big|\leq\tilde C_3N^{-1},\quad\big|\partial^2_{x_ix_j}\tilde H^{R_1}_N(t,x,p)\big|\leq\tilde C_3N^{-1}(\delta_{ij}+N^{-1}),\notag\\
    &\big|\partial_{p_i}\tilde H^{R_1}_N( t,x,p)\big|,\big|\partial^2_{p_ip_j}\tilde H^{R_1}_N(t,x,p)\big|\leq\tilde C_3,\quad1\leq i,j\leq N.
  \end{align}
\end{lemma}
\begin{proof}
  Recall 
  \eqref{glb-H} and \eqref{theta-*},
  \begin{align*}
    \tilde H^{R_1}_N(t,x,p)=H^{R_1}_N\big(t,x,p,\theta^{R_1}_N(t,x,p)\big),\quad(x,p)\in \mathcal A_N.
  \end{align*}
  Hence we can obtain the above estimates via \eqref{theta-weak-deri}.
\end{proof}

\subsubsection{Short time estimates}\label{sec-short-time}
As is mentioned before, with the preparation above, we may take partial derivatives in \eqref{glb-1} and derive the equation satisfied by $V^{\varepsilon,kl}_N:=\partial^2_{x_kx_l}V^{\varepsilon,R_1}_N$. We begin with a regularity results which validates the differentiation.
\begin{lemma}\label{1st-deri-2-0-1}
  Let $\Gamma^0_1,\ \Gamma^0_2\in\mathbb{R}$ be given constants such that $\big(x,\nabla_xV^{\varepsilon,R_1}_N(t,x)\big)\in\mathcal A_{N,\Gamma^0_1,\Gamma^0_2}$ for all $\varepsilon,R_1>0,\ (t,x)\in[0,T]\times\mathbb R^N$. Suppose that {\bf Hypothesis (R)} holds with $(\Gamma_1,\Gamma_2)=(\Gamma^0_1,\Gamma^0_2)$ in \eqref{hypo-regu-1}. 
 The equation \eqref{glb-1} admits a unique classical solution $V^{\varepsilon,R_1}_N\in C\left([0,T]\times\mathbb R^N\right)$ where $V^{\varepsilon,R_1}_N, \partial_{x_i}V^{\varepsilon,R_1}_N,\partial^2_{x_ix_j}V^{\varepsilon,R_1}_N, 1\leq i,j\leq N$ are bounded. Moreover for $0<\gamma<1$, $V^{\varepsilon,R_1}_N, \partial_{x_i}V^{\varepsilon,R_1}_N,\partial^2_{x_ix_j}V^{\varepsilon,R_1}_N\in C^{1+\frac\gamma2,2+\gamma}_{loc}\left([0,T)\times\mathbb R^N\right), 1\leq i,j\leq N$.
\end{lemma}
\begin{proof}
 In Lemma \ref{1st-deri-2-0} we have shown that the solution to \eqref{glb-1} $V^{\varepsilon,R_1}_N\in C([0,T]\times\mathbb R^N)$ has bounded derivatives $V^{\varepsilon,R_1}_N, \partial_{x_i}V^{\varepsilon,R_1}_N,\partial^2_{x_ix_j}V^{\varepsilon,R_1}_N, 1\leq i,j\leq N$. 
 In order to show the higher regularity of $V^{\varepsilon,R_1}_N$, let $R_2$ be sufficiently large and take $\partial_{x_i}$ ($1\leq i\leq N$) in \eqref{regularity}  to obtain the linear PDE satisfied by $\partial_{x_i}V^{\varepsilon,R_1}_N$. Notice that when $R_2$ is sufficiently large,
 \begin{align*}
  \partial_{x_i}g(t,x)=\partial_{x_i}\big(\tilde H_N(x,\nabla_xV^{\varepsilon,R_1}_N)\big)\in C^{\frac\gamma2,\gamma}_{loc}\left([0,T)\times\mathbb R^N\right).
 \end{align*}
So we have by the standard results on linear PDE that $\partial_{x_i}V^{\varepsilon,R_1}_N\in C^{1+\frac\gamma2,2+\gamma}_{loc}\left([0,T)\times\mathbb R^N\right), 1\leq i\leq N$. In view of Lemma \ref{coeff.}, we may repeat the previous procedure once more, i.e., take $\partial^2_{x_ix_j}$ ($1\leq i,j\leq N$) in \eqref{regularity} and show that $\partial^2_{x_ix_j}V^{\varepsilon,R_1}_N\in C^{1+\frac\gamma2,2+\gamma}_{loc}\left([0,T)\times\mathbb R^N\right), 1\leq i,j\leq N$.
\end{proof}

Denote by $V^{\varepsilon,R_1,kl}_N=\partial^2_{x_kx_l}V^{\varepsilon,R_1}_N$, $1\leq k,l\leq N$. By direct calculation, applying $\partial^2_{x_kx_l}$ to the equation \eqref{glb-1}, one obtains 
\begin{align}\label{glb-2}
\left\{\begin{aligned}\partial_tV^{\varepsilon,R_1,kl}_N&+\frac{\sigma^2}2\sum_{i,j=1}^N\partial^2_{x_ix_j}V^{\varepsilon,kl}_N+\frac{\varepsilon^2}2\sum_{i=1}^N\partial^2_{x_ix_i}V^{\varepsilon,kl}_N+\partial^2_{x_kx_l}\tilde H^{R_1}_N(t,x,\nabla_xV^{\varepsilon,R_1}_N)\\
&+\sum_{i=1}^N\partial_{p_i}\tilde H^{R_1}_N(t,x,\nabla_xV^{\varepsilon,R_1}_N)\partial_{x_i}V^{\varepsilon,R_1,kl}_N+\sum_{i,j=1}^N\partial^2_{p_ip_j}\tilde H^{R_1}_N(t,x,\nabla_xV^{\varepsilon,R_1}_N)V^{\varepsilon,R_1,ki}_NV^{\varepsilon,R_1,jl}_N\\
& +\sum_{i=1}^N\partial^2_{x_lp_i}\tilde H^{R_1}_N (t,x,\nabla_xV^{\varepsilon,R_1}_N)V^{\varepsilon,R_1,ki}_N+ \sum_{i=1}^N\partial^2_{x_kp_i}\tilde H^{R_1}_N (t,x,\nabla_xV^{\varepsilon,R_1}_N)V^{\varepsilon,R_1,li}_N\\
&=0,\\
&V^{\varepsilon,R_1,kl}_N(T,x_1,\ldots,x_N)=\frac{\delta_{kl}}NU''_{R_1}(x_k),\quad1\leq k,l\leq N.
\end{aligned}\right.
\end{align}
The equation above enables us to arrive to the results on the second order derivatives via nonlinear Feynman--Kac representation (see Section 7.4.2 in \cite{Yong1999}). In the current subsection we present the estimates on the second order derivatives for short time.
\begin{proposition}\label{dY/dx-1-0}
  Suppose that the assumptions of Lemma \ref{1st-deri-2-0-1} take place.
  There exists a constant {$\tilde c=\tilde c(f,\lambda^{-\frac12},L,U,(\lambda-\lambda_0)^{-1})$} and {$\tilde C_4=\tilde C_4(f,L,U,(\lambda-\lambda_0)^{-1})$, $\tilde C_4$ increasing in $(\lambda-\lambda_0)^{-1}$}, such that for $T<\tilde c$, PDE \eqref{glb-2} admits a unique bounded solution satisfying for $1\leq i,j\leq N$,
  \begin{align}\label{short-time-2nd-deri}
    \big|V^{\varepsilon,R_1,ij}_N(t,x)\big|\leq\tilde C_4N^{-1}(\delta_{ij}+N^{-1}),\quad(t,x)\in[0,T]\times\mathbb R^N.
  \end{align}
\end{proposition}
\begin{proof}
For $x=(x_1,\ldots,x_N)\in\mathbb R^N$, consider
\begin{align}\label{SDE-X-i}
  \left\{\begin{aligned}&dX^i(t)=\partial_{p_i}\tilde H^{R_1}_N\big(t,X(t),\nabla_xV^{\varepsilon,R_1}_N(t,X(t))\big)dt+\sigma dW^i(t)+\varepsilon dW^0(t),\ t\in[t_0,T],\\
  &X^i(t_0)=x_i,\ 1\leq i\leq N,\end{aligned}\right.
\end{align}
as well as
\begin{align}\label{def-matrix-Y}
  Y^{kl}(t)=V^{\varepsilon,R_1,kl}_N\big(t,X(t)\big),\quad1\leq k,l\leq N.
\end{align}
According to Lemma \ref{1st-deri-2-0}, we can deduce the existence of a constant $C$ depending on $R_1$ such that
\begin{align*}
  \big|\nabla_xV^{\varepsilon,R_1}_N(t,X(t))\big|\leq C.
\end{align*}
The estimate above and the first order condition associated to \eqref{glb-H} yields
\begin{align}\label{linear-growth-0}
  \big|\partial_{p_i}\tilde H^{R_1}_N\big(t,X_t,\nabla_xV^{\varepsilon,R_1}_N(t,X(t))\big)\big|\leq C(1+|X_t|).
\end{align}
Hence SDE \eqref{SDE-X-i} admits a weak solution satisying
\begin{align*}
  \mathbb E\bigg[\max_{0\leq t\leq T}|X(t)|^m\bigg]\leq C(1+|x|^m),\quad m\geq1.
\end{align*}
In view of Lemma \ref{1st-deri-2-0},
\begin{align*}
  |Y^{kl}(t)|\leq C,\quad1\leq k,l\leq N,
\end{align*}
where the constant $C$ might depend on $\varepsilon$ and $R_1$.

In view of \eqref{glb-2} and the estimates above, we can infer from the nonlinear Feynman--Kac representation that the matrix process $Y(t)$ satisfies the {backward stochastic Riccati equation}
\begin{align}\label{riccati-0-0}
  &Y(t)=\mathbb E_t\bigg\{\frac1N\tilde U(T)+\int_t^T\bigg[\nabla^2_{xx}\tilde H^{R_1}_N\big(s,X(s),\nabla_xV^{\varepsilon,R_1}_N(s,X(s))\big)+Y(s)\nabla^2_{xp}\tilde H^{R_1}_N\big(s,X(s),\nabla_xV^{\varepsilon,R_1}_N(s,X(s))\big)\notag\\
  &+\nabla^2_{xp}\tilde H^{R_1}_N\big(s,X(s),\nabla_xV^{\varepsilon,R_1}_N(s,X(s))\big)Y(s)+Y(s)\nabla^2_{pp}\tilde H^{R_1}_N\big(s,X(s),\nabla_xV^{\varepsilon,R_1}_N(s,X(s))\big)Y(s)\bigg]ds\bigg\},
\end{align}
where the matrix $\tilde U(T)$ is given by
\begin{align}\label{tilde-UT}
  \tilde U^{ij}(T)=\delta_{ij}U''\big(X^i(T)\big),\quad1\leq i,j\leq N.
\end{align}
Next, define the mapping from the set of adapted matrix processes to itself
\begin{align*}
  \Phi:\quad L^\infty\big(\Omega;C([0,T];\mathbb R^{N\times N})\big)\ \longrightarrow L^\infty\big(\Omega;C([0,T];\mathbb R^{N\times N})\big),\quad{\Phi(Y)=\tilde Y,}
\end{align*}
such that for $t\in[0,T]$,
\begin{align*}
  \tilde Y(t)=\mathbb E_t\bigg[\frac1N\tilde U(T)+\int_t^T\big[\nabla^2_{xx}\tilde H^{R_1}_N(s,X(s),\nabla_sV^{\varepsilon,R_1}_N(s,X(s)))+Y(s)\nabla^2_{xp}\tilde H^{R_1}_N(s,X(s),\nabla_xV^{\varepsilon,R_1}_N(s,X(s)))\notag\\
  +\nabla^2_{xp}\tilde H^{R_1}_N(s,X(s),\nabla_sV^{\varepsilon,R_1}_N(s,X(s)))Y(s)+Y(s)\nabla^2_{pp}\tilde H^{R_1}_N(s,X(s),\nabla_sV^{\varepsilon,R_1}_N(s,X(s)))Y(s)\big]ds\bigg].
\end{align*}

We can see that $Y_t$ in \eqref{def-matrix-Y} is a fixed point of $\Phi$. Next we show that such fixed point is unique. In fact, let $Y^*_t$ and $Y^{**}_t$ be two bounded fixed points. And consider their norm of the following form
\begin{align*}
  \big\|Y^*\big\|=\max_{0\leq t\leq T}\big\|Y^*(t)\big\|_\infty:=\max_{0\leq t\leq T}\max_{1\leq i\leq N}\sum_{j=1}^N\big|Y^{*,ij}(t)\big|\leq C,\quad\max_{0\leq t\leq T}\big\|Y^{**}(t)\big\|_\infty\leq C.
\end{align*}
Then for $t\in[T-\delta,T]$ and $\tilde C=\tilde C_3$ depending only on $\tilde C_3$ from \eqref{hamilton-001},
\begin{align*}
  &\quad\big\|Y^*(t)-Y^{**}(t)\big\|_\infty\notag\\
  &\leq\mathbb E_t\bigg[\int_t^T\big(\big\|Y^*(s)\big\|_\infty+\big\|Y^{**}(s)\big\|_\infty\big)\big\|\nabla^2_{pp}\tilde H^{R_1}_N\big(s,X(s),\nabla_xV^\varepsilon_N(s,X(s))\big)\big\|_\infty\big\|Y^*(s)-Y^{**}(s)\big\|_\infty ds\bigg]\notag\\
  &\quad+2\mathbb E_t\bigg[\int_t^T\big\|\nabla^2_{xp}\tilde H^{R_1}_N(s,X(s),\nabla_xV^\varepsilon_N(s,X(s)))\big\|_\infty\big\|Y^*(s)-Y^{**}(s)\big\|_\infty ds\bigg]\notag\\
  &\leq 2(C+1)\tilde C\mathbb E_t\bigg[\int_t^T\big\|Y^*(s)-Y^{**}(s)\big\|_\infty ds\bigg]\leq2(C+1)\tilde C\delta\max_{T-\delta\leq s\leq T}\big\|Y^*(s)-Y^{**}(s)\big\|_\infty\quad a.s..
\end{align*}
Choose $2(C+1)\tilde C\delta<1$, then the inequality above implies that $Y^*(t)=Y^{**}(t)$ on $t\in[T-\delta,T]$. Repeat the above procedure, we can show that $Y^*(t)=Y^{**}(t)$ on $t\in[T-\delta,T]$, $[T-2\delta,T-\delta]$ and after finite times repetitions we obtain $Y^*(t)=Y^{**}(t)$ on $t\in[0,T]$. The uniqueness above thus { tells us that} $Y$ in \eqref{def-matrix-Y} is the only bounded fixed point of $\Phi$.

To continue, define the closed subset $\mathcal B(N,K)$ of adapted matrix processes in such a way that $Y\in\mathcal B(N,K)$ if and only if
\begin{align}\label{cal B-N-1}
    \max_{t\in[0,T]}|Y^{ij}(t)|\leq KN^{-1}(\delta_{ij}+N^{-1})\quad {\rm{a.s.}},
\end{align}
where the constant $K>0$ is to be determined.

We claim that for appropriate $K$ and $\tilde c$ (independent of $N$), $\Phi$ is invariant on $\mathcal B(N,K)$, and $\Phi$ is a contraction mapping on $\mathcal B(N,K)$ with $T<\tilde c$.

Let $Y^{(1)}$ and $Y^{(2)}$ be two inputs from $\mathcal B(N,K)$ and $\tilde Y^{(1)}$ and $\tilde Y^{(2)}$ be the associated outputs.
\begin{align*}
  &\quad\big\|\tilde Y^{(1)}(t)-\tilde Y^{(2)}(t)\big\|_\infty=\max_{1\leq i\leq N}\sum_{j=1}^N\big|\tilde Y^{(1),ij}(t)-\tilde Y^{(2),ij}(t)\big|\notag\\
  &\leq\mathbb E_t\bigg[\int_t^T\big(\big\|Y^{(1)}(s)\big\|_\infty+\big\|Y^{(2)}(s)\big\|_\infty\big){\big\|\nabla^2_{pp}\tilde H^{R_1}_N(s,X(s),\nabla_xV^\varepsilon_N(s,X(s)))\big\|_\infty}\big\|Y^{(1)}(s)-Y^{(2)}(s)\big\|_\infty ds\bigg]\notag\\
  &\quad+2\mathbb E_t\bigg[\int_t^T\big\|\nabla^2_{xp}\tilde H^{R_1}_N\big(s,X_s,\nabla_xV^\varepsilon_N(s,X(s))\big)\big\|_\infty\big\|Y^{(1)}(s)-Y^{(2)}(s)\big\|_\infty ds\bigg]\notag\\
  &\leq(2K+1)\tilde CT\max_{0\leq s\leq T}\big\|Y^{(1)}(s)-Y^{(2)}(s)\big\|_\infty\quad {\rm{a.s.}},
\end{align*}
where $\tilde C$ is increasing in $T$ because by Lemma \ref{coeff.} the constant $\tilde C_3$ is increasing in $T$. Let's further fix the parameter $T$ in $\tilde C$ to be $T=1$ and obtain {$\tilde C=\tilde C(f,\lambda^{-\frac12},L,(\lambda-\lambda_0)^{-1})$}. Hence for $T<1$,
\begin{align*}
  \max_{0\leq t\leq T}\big\|\tilde Y^{(1)}(t)-\tilde Y^{(2)}(t)\big\|_\infty\leq(2K+1)\tilde CT\max_{0\leq t\leq T}\big\|Y^{(1)}(t)-Y^{(2)}(t)\big\|_\infty\quad {\rm{a.s.}}.
\end{align*}
We thus have that if we choose $K$, $\tilde c$ satisfying
\begin{align*}
  (2K+1)\tilde C\tilde c<1,\quad\tilde c<1,
\end{align*}
then $\Phi$ is a contraction mapping on $\mathcal B(N,K)$ with $T<\tilde c$. Next we show that $\mathcal B(N,K)$ is invariant for appropriate $K$ and $\tilde c$. Denote by $Y(t)\in\mathcal B(N,K)$ the input and $\tilde Y(t)$ the output, then Lemma \ref{coeff.} and direct calculation yield
\small
\begin{align*}
  |\tilde Y^{ij}(t)|&\leq\mathbb E_t\bigg[\frac1N|\tilde U^{ij}(T)|+\int_t^T\bigg|\partial^2_{x_ix_j}\tilde H^{R_1}_N(s,X(s),\nabla_xV^{\varepsilon,R_1}_N(s,X(s)))+\sum_{k=1}^NY^{ik}(s)\partial^2_{p_kx_j}\tilde H^{R_1}_N\big(s,X(s),\nabla_xV^\varepsilon_N(s,X(s))\big)\notag\\
  &\quad\sum_{k=1}^N\partial^2_{x_ip_k}\tilde H^{R_1}_N\big(s,X(s),\nabla_xV^\varepsilon_N(s,X(s))\big)Y^{kj}(s)+\sum_{k,l=1}^NY^{ik}(s)\partial^2_{p_kp_l}\tilde H^{R_1}_N\big(s,X(s),\nabla_xV^\varepsilon_N(s,X(s))\big)Y^{lj}(s)\bigg|ds\bigg]\notag\\
  &\leq\delta_{ij}\tilde CN^{-1}+\tilde C\tilde cN^{-1}(\delta_{ij}+N^{-1})+\tilde C\tilde cKN^{-2}\sum_{k=1}^N(\delta_{ik}+N^{-1})+\tilde C\tilde cKN^{-2}\sum_{k=1}^N(\delta_{kj}+N^{-1})\notag\\
  &\quad+\tilde C\tilde cK^2N^{-2}\sum_{k,l=1}^N(\delta_{ik}+N^{-1})(\delta_{lj}+N^{-1})\notag\\
  &=\delta_{ij}\tilde CN^{-1}+\tilde C\tilde cN^{-1}(\delta_{ij}+N^{-1})+4\tilde C\tilde cKN^{-2}+4\tilde C\tilde cK^2N^{-2}.
\end{align*}
\normalsize
It is easy to see that we can choose $K$, $\tilde c$ such that
 \begin{align*}
   K=K\big(f,\lambda^{-\frac12},L,(\lambda-\lambda_0)^{-1}\big),\quad\tilde c=\tilde c\big(f,\lambda^{-\frac12},L,(\lambda-\lambda_0)^{-1}\big),
 \end{align*}
 and
\begin{align*}
  KN^{-1}(\delta_{ij}+N^{-1})>\delta_{ij}\tilde CN^{-1}+\tilde C\tilde cN^{-1}(\delta_{ij}+N^{-1})+4\tilde C\tilde cKN^{-2}+4\tilde C\tilde cK^2N^{-2}.\end{align*}
Then we have for such $K$, $\tilde c$ that $\mathcal B(N,K)$ is invariant.

Since $\Phi$ is contractive and invariant on $(\mathcal B(N,K),\|\cdot\|_\infty)$, which is a Banach space, it follows that $\Phi$ admits a fixed point in $\mathcal B(N,K)$ when $T<\tilde c$. Note that processes in $\mathcal B(N,K)$ are all bounded. Therefore the aforementioned fixed point in $\mathcal B(N,K)$ is nothing but the matrix process in \eqref{def-matrix-Y} and we may take $\tilde C_4=K$. Consider $t=t_0$ in \eqref{def-matrix-Y}, then we have \eqref{short-time-2nd-deri} from \eqref{cal B-N-1}.
\end{proof}
We can see from the proof above that $\tilde C_4$ actually depends on $U', U''$ rather than $U$.
\subsubsection{Global in time estimates}
In this subsection we focus on the global estimates for any given $T>0$ with sufficiently smooth data. As will be seen, the global estimates rely heavily on the convexity assumption (with respect to $x$) on the Hamiltonian {$\tilde H_N(t,x,p)$} in \eqref{H-N-R}. However, the truncation of $L$, $U$ might break the convexity of {$\tilde H_N(t,x,p)$}. Therefore, we need to pass $R_1$ to infinity in \eqref{glb-1} and consider
\begin{align}\label{glb-3}
\left\{\begin{aligned}&\partial_tV^\varepsilon_N+\frac{\sigma^2}2\sum_{i,j=1}^N\partial^2_{x_ix_j}V^\varepsilon_N+\frac{\varepsilon^2}2\sum_{i=1}^N\partial^2_{x_ix_i}V^\varepsilon_N+\tilde H_N(t,x,\nabla_xV^\varepsilon_N)=0,\\
&V^\varepsilon_N(T,x_1,\ldots,x_N)=\frac1N\sum_{i=1}^NU(x_i),
\end{aligned}\right.
\end{align}
Here
\begin{align}\label{H-N-R}
  \tilde H_N(t,x,p):=\inf_{\theta\in\mathbb R}\bigg\{\frac\lambda2\big|\theta\big|^2+\sum_{i=1}^Nf\bigg(t,\theta,x_i,\frac1N\sum_{j=1}^N\delta_{x_j}\bigg)p_i+\frac1 N\sum_{i=1}^NL(x_i)\bigg\}.
\end{align}
Similarly to \eqref{glb-2}, we would like to take $\partial^2_{x_kx_l}$ in \eqref{glb-3} and analysis the resulting system. To do so, we show the validity of taking derivatives in the next proposition.
\begin{proposition}\label{1st-deri-2-1}
Suppose that the assumptions of Lemma \ref{1st-deri-2-0-1} take place.
The PDE \eqref{glb-3} admits a unique classical solution $V^{\varepsilon,R_1}_N\in C\left([0,T]\times\mathbb R^N\right)$ where $V^{\varepsilon,R_1}_N,\partial_tV^{\varepsilon,R_1}_N, \partial_{x_i}V^{\varepsilon,R_1}_N,\partial^2_{x_ix_j}V^{\varepsilon,R_1}_N, 1\leq i,j\leq N$ are bounded. For $0<\gamma<1$ and $1\leq i,j\leq N$, $V^\varepsilon_N, \partial_{x_i}V^\varepsilon_N,\partial^2_{x_ix_j}V^\varepsilon_N\in C^{1+\frac\gamma2,2+\gamma}_{loc}\left([0,T)\times\mathbb R^N\right)$. And for $\varphi\in\{V^\varepsilon_N,\partial_tV^\varepsilon_N, \partial_{x_i}V^\varepsilon_N, \partial^2_{x_ix_j}V^\varepsilon_N\}$, $1\leq i,j\leq N$, $\varphi$ has polynomial growth in $x$:
\begin{align}\label{VN-growth-1}
  |\varphi(t,x)|\leq \breve{C}(1+|x|)^7,\quad(t,x)\in[0,T)\times\mathbb R^N.
\end{align}
Here the constant $\breve C$ depends only on $f,\ L,\ U,\ \sigma,\ \varepsilon$. Moreover, the solution $V^{\varepsilon,R_1}_N$ to \eqref{glb-1} satisfies
\begin{align*}
  \lim_{R_1\to+\infty}(V^{\varepsilon,R_1}_N,\partial_tV^{\varepsilon,R_1}_N,\partial_{x_i}V^{\varepsilon,R_1}_N,\partial^2_{x_ix_j}V^{\varepsilon,R_1}_N)(t,x)=(V^\varepsilon_N,\partial_{t}V^\varepsilon_N,\partial_{x_i}V^\varepsilon_N,\partial^2_{x_ix_j}V^\varepsilon_N)(t,x),
\end{align*}
where the convergence is {locally} uniform on $[0,T]\times\mathbb R^N$.

As a result, {for} the first order derivatives {of} $V^{\varepsilon,R_1}_N$, we also have
\begin{align}\label{1st-deri-scaling-2}
    \big|\partial_{x_i}V^\varepsilon_N(t,x)\big|&\leq\frac{\tilde C_2(C^L_{11}+C^U_{11})}N\bigg(1+|x_i|^2+\frac1N\sum_{j=1}^N|x_j|^2\bigg)^\frac12+\frac{\tilde C_2(C^L_{10}+C^U_{10})}N,
\end{align}
where the constant $\tilde C_2$ is from Lemma \ref{1st-deri}.
\end{proposition}
\begin{proof}
  Let $V^{\varepsilon,R_1}_N$ be the solution to \eqref{glb-1} in Lemma \ref{1st-deri-2-0}. According to Theorem 4.7.2 and Theorem 4.7.4 in \cite{Krylov1980} as well as the growth condition \eqref{trun-deri}, we have that for $\varphi\in \{V^{\varepsilon,R_1}_N,\partial_tV^{\varepsilon,R_1}_N, \partial_{x_i}V^{\varepsilon,R_1}_N, \partial^2_{x_ix_j}V^{\varepsilon,R_1}_N\}$, $1\leq i,j\leq N$, $\varphi$ has polynomial growth in $x$:
\begin{align*}
  |\varphi(t,x)|\leq \breve{C}(1+|x|)^7,\quad(t,x)\in[0,T)\times\mathbb R^N,
\end{align*}
where the constant $\breve C$ depends only on $f,\ L,\ U,\ \sigma,\ \varepsilon$ and is independent of $R_1$\footnote{According to Theorem 4.7.2 and Theorem 4.7.4 in \cite{Krylov1980} for $L$, $U$ with growth rate $(1+|x|^m)$, the estimates on the second order derivatives are of growth rate $(1+|x|^{3m+1})$. Here in our case, since $L$ and $U$ grow at most quadratically, we have $m=2$.}

Similar to \eqref{regularity}, we may view the solution of \eqref{glb-1} as the solution of the constant coefficients PDE
\begin{align}\label{regularity-1}
    \partial_tV^{\varepsilon,R_1}_N(t,x)+\frac{\sigma^2}2\sum_{i,j=1}^N\partial^2_{x_ix_j}V^{\varepsilon,R_1}_N(t,x)+\frac{\varepsilon^2}2\sum_{i=1}^N\partial^2_{x_ix_i}V^{\varepsilon,R_1}_N(t,x)=g^{R_1}(t,x),
  \end{align}
  where
  \begin{align*}
  g^{R_1}(t,x):=-\tilde H^{R_1}_N(t,x,\nabla_xV^{\varepsilon,R_1}_N).
\end{align*}
In view of Corollary 4.7.8 in \cite{Krylov1980} as well as Lemma \ref{1st-deri-2-0} and Lemma \ref{coeff.}, $g^{R_1}(t,x)$ is locally Lipschitz continuous with respect to $x$ with Lipschitz constant independent of $R_1$ while $g^{R_1}(t,x)$ is locally $\frac\gamma2-$H\"{o}lder continuous ($0<\gamma<1$) with respect to $t$ with H\"{o}lder constant independent of $R_1$. It then follows that $\partial_tV^{\varepsilon,R_1}_N(t,x)$ and $\partial^2_{x_ix_j}V^{\varepsilon,R_1}_N(t,x)$, $1\leq i,j\leq N$ are locally H\"{o}lder continuous in $(t,x)$ with H\"{o}lder constant independent of $R_1$. According to Arzel\`{a}--Ascoli Theorem, we may pass $R_1$ to infinity in \eqref{glb-1} and obtain the limit of $V^{\varepsilon,R_1}_N$ as the solution $V^\varepsilon_N\in C^{1+\frac\gamma2,2+\gamma}_{loc}\left([0,T)\times\mathbb R^N\right)\cap C\left([0,T]\times\mathbb R^N\right)$ of \eqref{glb-3}. We remark that because of the uniqueness of solutions to this last problem, there is no need to consider sub-sequential limits in the Arzel\`a--Ascoli theorem. Moreover, we have \eqref{1st-deri-scaling-2} and \eqref{VN-growth-1} for $\varphi\in\{V^\varepsilon_N, \partial_tV^{\varepsilon}_N, \partial_{x_i}V^\varepsilon_N, \partial^2_{x_ix_j}V^\varepsilon_N\}$, $1\leq i,j\leq N$.

In order to show higher regularity of $V^\varepsilon_N$, we may take $\partial_{x_i}$, $(1\leq i\leq N)$ in \eqref{glb-3} and obtain the PDE satisfies by $\partial_{x_i}V^\varepsilon_N$. Notice that $\partial_{x_i}\big(\tilde H_N(x,\nabla_xV^\varepsilon_N(t,x))\big)\in C^{\frac\gamma2,\gamma}_{loc}\left([0,T)\times\mathbb R^N\right)$, then it follows that $\partial_{x_i}V^\varepsilon_N\in C^{1+\frac\gamma2,2+\gamma}_{loc}\left([0,T)\times\mathbb R^N\right)\cap C\left([0,T]\times\mathbb R^N\right)$. Thanks to Lemma \ref{coeff.} we may let $R_1$ go to infinity in \eqref{hamilton-001} to obtain that $\partial^2_{x_ix_j}\big(\tilde H_N(x,\nabla_xV^\varepsilon_N(t,x))\big)$ is bounded and $\partial^2_{x_ix_j}\big(\tilde H_N(x,\nabla_xV^\varepsilon_N(t,x))\big)\in C^{1+\frac\gamma2,2+\gamma}_{loc}\left([0,T)\times\mathbb R^N\right)$. Thus we can further take  $\partial^2_{x_ix_j}$, $(1\leq i,j\leq N)$ in \eqref{glb-3} and repeat the previous procedure once more to show that $\partial^2_{x_ix_j}V^\varepsilon_N\in C^{1+\frac\gamma2,2+\gamma}_{loc}\left([0,T)\times\mathbb R^N\right)\cap C\left([0,T]\times\mathbb R^N\right)$.
\end{proof}
{ We note here that \eqref{1st-deri-scaling-2} is equivalent to $\big(x,\nabla_xV^\varepsilon_N(t,x)\big)\in\mathcal A_{N,\tilde C_2,\tilde C_2},\ (t,x)\in[0,T]\times\mathbb R^N.$} 
Now we make the following assumption on convexity of the data and set out for the global in time estimates.

\noindent{\bf Hypothesis (R1).} Suppose \noindent{\bf Hypothesis (R)} parameters $\Gamma_1,\ \Gamma_2$  and the following:
\begin{enumerate}
\item $U$ is convex;
\item
{$\tilde H_N$ defined in \eqref{H-N-R} is convex in the $x$-variable, and satisfies in particular 
\begin{align}\label{convexity-assumption}
 \nabla^2_{xx}\tilde H_N\left(t,x,p\right)\geq0,\ (t,x,p)\in[0,T]\times\mathbb R^N\times\mathbb R^N,
\end{align}
in the sense of positive definite matrices.
}
\end{enumerate}

\begin{remark}
  { As we will see in the analysis below, the second assumption in {\bf Hypothesis (R1)} on the convexity of the Hamiltonian in the $x$-variable can be relaxed as eventually it will be used only restricted to the set $\left\{(t,x,p)\in[0,T]\times\mathbb R^N\times\mathbb R^N: p = \nabla_xV^\varepsilon_N(t,x)\right\}.$
  Such relaxation will also be shown to be particularly meaningful in the case study of Section \ref{adaptation}. The assumption} on the convexity of the Hamiltonian is quite common in mean field control problems. { In general, infinite dimensional HJB equations for mean field control are degenerate despite the presence of non-degenerate noise, and convexity conditions are used to establish the a priori estimates for global classical solutions. In our context, the convexity condition is used to established uniform in $N$ estimates in Lemma \ref{prior-bound}--Proposition \ref{dY/dx-1} below.} In our model, since the control is centralized and the dynamics of the particles { are} more complicated, this assumption could no longer be guaranteed in a simple way. According to direct calculation, with the notation $\mu=\frac1N\sum_{j=1}^N\delta_{x_j}$, we find
  \begin{align*}
    \partial^2_{x_kx_l}\tilde H_N(t,x,p)&=\delta_{kl}\frac1NL''(x_k)+\partial^2_{\theta x}f(t,\theta^*,x_k,\mu)p_k\partial_{x_l}\theta^*+\delta_{kl}\partial^2_{xx}f(t,\theta^*,x_k,\mu)p_k+\frac1N\partial^2_{x\mu}f(t,\theta^*,x_k,\mu,x_l)p_k\notag\\
    &\quad+\frac1N\partial^2_{x\mu}f(t,\theta^*,x_l,\mu,x_k)p_l+\frac1N\sum_{i=1}^N\partial^2_{\mu\theta}f(t,\theta^*,x_i,\mu,x_k)p_i\partial_{x_l}\theta^*\notag\\
    &\quad+\delta_{kl}\frac1N\sum_{i=1}^N\partial^2_{\mu\tilde x}f(t,\theta^*,x_i,\mu,x_k)p_i+\frac1{N^2}\sum_{i=1}^N\partial^2_{\mu\mu}f(t,\theta^*,x_i,\mu,x_k,x_l)p_i.
  \end{align*}
  We can see from the above that one possible way to ensure the convexity of $\tilde H_N(t,x,p)$ in $x$ is to assume an affine structure on $f$. 
  
  Indeed, set the parameters in \eqref{sample-dyn} and \eqref{obj-func-2} as follows
\begin{align*}
{f\left(t,\theta,x,\mu\right)=\theta+x+\int_{\mathbb R}y\mu(dy)},\ { L''(x)\ge 0,\ U''(x)\geq0},\ \forall x\in\mathbb R,\ \sigma=\lambda=1.
\end{align*}
Then {\bf Hypothesis (R1)} can be easily verified. { A more sophisticated example that fulfils {\bf Hypothesis (R1)} will be analyzed in Section \ref{adaptation}, where $f$ is nonlinear and $L,U$ are increasing.

We also note here that the convexity condition \eqref{convexity-assumption} is the analogy of displacement monotonicity in mean field games. Inspired by the literature on mean field games, there are at least two possible ways to relax the requirement on convexity. One possible way is to consider linear quadratic framework, see \cite{MinLi2023} where mean field games of controls with non-monotonic data is studied. The other possible way is to introduce the $\lambda$-monotonicity condition or involve canonical transformations, see \cite{Mou2022MFGC, BanMes:25-FMS, Liu26, BanMes:25-BLMS}.}
\end{remark}

Similar to the last subsection, the key estimate in this subsection is from the BSDE of Riccati type \eqref{riccati} below. The following lemma is devoted to estimating the terms appearing in \eqref{riccati}.
\begin{lemma}\label{coeff.-1}
  { Suppose {\bf Hypothesis (R1)} with certain parameters $\Gamma_1,\ \Gamma_2$ in \eqref{hypo-regu-1}, then for $\tilde H_N$ in \eqref{H-N-R} and $(t,x)\in[0,T]\times\mathbb R^N$,
  \begin{align}\label{hamilton-002}
   &\big|\partial_{x_i}\tilde H_N\big(t,x,\nabla_xV^\varepsilon_N(t,x)\big)\big|,\big|\partial^2_{x_ip_j}\tilde H_N\big(t,x,\nabla_xV^\varepsilon_N(t,x)\big)\big|\leq\tilde C_3N^{-1},\notag\\
   &\big|\partial^2_{x_ix_j}\tilde H_N\big(t,x,\nabla_xV^\varepsilon_N(t,x)\big)\big|\leq\tilde C_3N^{-1}(\delta_{ij}+N^{-1}),\notag\\
    &\big|\partial_{p_i}\tilde H_N\big(t,x,\nabla_xV^\varepsilon_N(t,x)\big)\big|,\big|\partial^2_{p_ip_j}\tilde H_N\big(t,x,\nabla_xV^\varepsilon_N(t,x)\big)\big|\leq\tilde C_3,\quad1\leq i,j\leq N.
  \end{align}
  As a result, there exists a constant $\tilde C_5=\tilde C_5(f,\lambda^{-\frac12},T,L,U,(\lambda-\lambda_0)^{-1})$ such that
  \begin{align*}
    0\leq \nabla^2_x\tilde H_N\big(t,x,\nabla_xV^\varepsilon_N(t,x)\big)\leq\frac{\tilde C_5}N I_N.
  \end{align*}}
\end{lemma}
\begin{proof}
   In view of Lemma \ref{coeff.}, the constant $\tilde C_3$ in \eqref{hamilton-001} is independent of $R_1$. Therefore we can let $R_1$ go to infinity and obtain \eqref{hamilton-002} according to definitions in \eqref{glb-H} and \eqref{H-N-R}. Furthermore, in view of the second inequality in \eqref{hamilton-002}, we can deduce the existence of $\tilde C_5$ such that for any { $\xi\in\mathbb R^N$, $(t,x)\in[0,T]\times\mathbb R^N$,
  \begin{align*}
    \sum_{i,j=1}^N\partial^2_{x_ix_j}\tilde H_N\big(t,x,\nabla_xV^\varepsilon_N(t,x)\big)\xi_i\xi_j\leq\frac12\sum_{i,j=1}^N\partial^2_{x_ix_j}\tilde H_N\big(t,x,\nabla_xV^\varepsilon_N(t,x)\big)(\xi^2_i+\xi^2_j)\leq\frac{\tilde C_5}N|\xi|^2.
  \end{align*}}
  Combining the above with {\bf Hypothesis (R1)} we have the last inequality.
\end{proof}
With the preparation above, if we further assume that $\partial^2_{x_ix_j}V^\varepsilon_N(t,x)$ $(1\leq i,j\leq N)$ are bounded, we would then obtain a refined estimate on the bound of $\partial^2_{x_ix_j}V^\varepsilon_N(t,x)$ $(1\leq i,j\leq N)$ in \eqref{global-eigen.}. This is obtained via the BSDE of Riccati type in \eqref{riccati} below where the convexity assumption plays a key role.

\begin{lemma}\label{prior-bound}
  Let the constants $\Gamma^0_1,\ \Gamma^0_2$ be such that $\big(x,\nabla_xV^\varepsilon_N(t,x)\big)\in\mathcal A_{N,\Gamma^0_1,\Gamma^0_2}$ for all $\varepsilon,R_1>0,\ (t,x)\in[0,T]\times\mathbb R^N$. Suppose that {\bf Hypothesis (R1)} holds with $(\Gamma_1,\Gamma_2)=(\Gamma^0_1,\Gamma^0_2)$ in \eqref{hypo-regu-1}, and that there exist positive constants $\delta$ and $\breve C$ {(which could depend on $N$ and $\varepsilon$)} such that for $(t,x)\in[T-\delta,T]\times\mathbb R^N$, it holds that $\partial^2_{x_ix_j}V^\varepsilon_N(t,x)$ $(1\leq i,j\leq N)$ are bounded by constant $\breve C$. Then there exists a constant $\tilde C_6=\tilde C_6(f,\lambda^{-\frac12},T,L,U,(\lambda-\lambda_0)^{-1})$ (independent of $\breve C$ and $\delta$) such that for $\xi\in\mathbb R^N$ and $(t,x)\in[T-\delta,T]\times\mathbb R^N$,
  \begin{align}\label{global-eigen.}
    0\leq\sum_{i,j=1}^NV^{\varepsilon,ij}_N(t,x)\xi_i\xi_j\leq\frac{\tilde C_6}N|\xi|^2,
  \end{align}
  {In particular, $\partial^2_{x_ix_j}V^\varepsilon_N(t,x)$ $(1\leq i,j\leq N)$ are bounded by $\frac{\tilde C_6}N$ for $(t,x)\in[T-\delta,T]\times\mathbb R^N$.}
\end{lemma}
\begin{proof}
  Without the loss of generality, we show \eqref{global-eigen.} when $t=T-\delta$. For $x=(x_1,\ldots,x_N)\in\mathbb R^N$, consider
\begin{align}\label{SDE-X-i-1}
  dX^i(t)=\partial_{p_i}H(t,X(t),\nabla_xV^\varepsilon_N(t,X(t)))dt+\sigma dW^i(t)+\varepsilon dW^0(t), X^i_{T-\delta}=x_i,
\end{align}
as well as 
\begin{align}\label{FBSDE-Y}
Y^{kl}(t)=V^{\varepsilon,kl}_N(t,X(t)),\quad(t,x)\in[T-\delta,T]\times\mathbb R^N.
\end{align}
According to \eqref{1st-deri-scaling-2} and {\bf Hypothesis (R1)}, it is easy to show that
\begin{align}\label{linear-growth}
  \big|\partial_{p_i}H(t,X(t),\nabla_xV^\varepsilon_N(t,X(t)))\big|\leq C(1+|X(t)|).
\end{align}
for some constant $C$. Hence \eqref{SDE-X-i-1} admits a weak solution satisfying
\begin{align*}
  \mathbb E\bigg[\max_{0\leq t\leq T}|X(t)|^\kappa\bigg]\leq C(1+|x|^\kappa),\quad\forall \kappa\geq1.
\end{align*}
Moreover, since $t\in[T-\delta,T]$, we have by assumption that
\begin{align*}
  |Y^{kl}(t)|\leq\breve C,\quad1\leq k,l\leq N.
\end{align*}

Given the estimates above and Proposition \ref{1st-deri-2-1}, we may differentiate \eqref{glb-3} with respect to $x_i$, $x_j$ ($1\leq i,j\leq N$) and obtain an analog of \eqref{glb-2}, then we can deduce from the {assumption on the} boundedness of the matrix process $Y(t)$ that it satisfies the Riccati type equation
\begin{align}\label{riccati}
  Y(t)&=\mathbb E_t\bigg[\frac1N\tilde U(T)+\int_t^T\big[\nabla^2_{xx}\tilde H_N\big(X(s),\nabla_xV^\varepsilon_N(s,X(s))\big)+Y(s)\nabla^2_{xp}\tilde H_N\big(s,X(s),\nabla_xV^\varepsilon_N(s,X(s))\big)\notag\\
  &\quad+\nabla^2_{px}\tilde H_N\big(s,X(s),\nabla_xV^\varepsilon_N(s,X(s))\big)Y(s)+Y(s)\nabla^2_{pp}\tilde H_N\big(X(s),\nabla_xV^\varepsilon_N(s,X(s))\big)Y(s)\big]ds\bigg].
\end{align}
Here we recall that the first term $\frac1N\tilde U(T)$ on the right hand side is defined similarly to that in \eqref{tilde-UT}. Define $\Phi(s)$ satisfying
\begin{align}\label{riccati-0}
  &\Phi(t)=I_N-\int_{T-\delta}^t\Phi(s)\bigg[\frac12Y(s)\nabla^2_{pp}\tilde H_{N}(s,X(s),\nabla_xV^\varepsilon_N(s,X(s)))+\nabla^2_{px}\tilde H_N\big(s,X(s),\nabla_xV^\varepsilon_N(s,X(s))\big)\bigg]ds,\notag\\
  &t\in[T-\delta,T].
\end{align}
Note here that $\Phi(t)$, $t\in[T-\delta,T]$ is bounded because $Y(t),\ \nabla^2_{pp}\tilde H_{N}$ and $\nabla^2_{px}\tilde H_N$ in the right hand side above are bounded. According to the above and \eqref{riccati}, we may write the dynamics of $Y(t)$, $\Phi(t)$ as
\begin{align*}
 dY(t)&= {\bf -}\bigg[\nabla^2_{xx}\tilde H_N\big(X(t),\nabla_xV^\varepsilon_N(t,X(t))\big)+Y(t)\nabla^2_{xp}\tilde H_N\big(t,X(t),\nabla_xV^\varepsilon_N(t,X(t))\big)\notag\\
  +&\nabla^2_{px}\tilde H_N\big(t,X(t),\nabla_xV^\varepsilon_N(t,X(t))\big)Y(t)+Y(t)\nabla^2_{pp}\tilde H_N\big(X(t),\nabla_xV^\varepsilon_N(t,X(t))\big)Y(t)\bigg]dt+\sum_{i=0}^NZ^i(t)dW^i(t),\notag\\
  d\Phi(t)&=-\Phi(t)\bigg[\frac12Y(t)\nabla^2_{pp}\tilde H_{N}(t,X(t),\nabla_xV^\varepsilon_N(t,X(t)))+\nabla^2_{px}\tilde H_N\big(t,X(t),\nabla_xV^\varepsilon_N(t,X(t))\big)\bigg]dt,
\end{align*}
where $\int_{T-\delta}^tZ^i(s)dW^i(s)$ $(0\leq i\leq N)$ are BMO martingales. Then It\^{o}'s formula gives
\begin{align*}
 d(\Phi(t)Y(t)\Phi(t)^\top)&=(d\Phi(t))Y(t)\Phi(t)^\top+\Phi(t)(dY(t))\Phi(t)^\top+\Phi(t)Y(t)(d\Phi(t)^\top)\notag\\
 &={{\bf -}}\Phi(t)\nabla^2_{xx}\tilde H_{N}\big(t,X(t),\nabla_xV^\varepsilon_N(t,X(t))\big)\Phi^\top(t)dt+\sum_{i=0}^N\Phi(t)Z^i(t)\Phi^\top(t)dW^i(t).
\end{align*}
Since $\Phi(t)$ is bounded, we may present the above as
\begin{align}\label{riccati-1}
  \Phi(t)Y(t)\Phi^\top(t)=\mathbb E(t)\bigg[\frac1N\Phi(t)\tilde U(T)\Phi^\top(t)+\int_t^T\Phi(s)\nabla^2_{xx}\tilde H_{N}\big(s,X(s),\nabla_xV^\varepsilon_N(s,X(s))\big)\Phi^\top(s)ds\bigg].
\end{align}
According to {\bf Hypothesis (R1)}, we have
\begin{align*}
 \tilde U(T)\geq0,\quad\nabla^2_{xx}\tilde H_{N}\big(s,X(s),\nabla_xV^\varepsilon_N(s,X(s))\big)\geq0.
\end{align*}
Here the ordering relation $\geq$ is used in the sense of positive semi-definite matrices. Hence
\begin{align*}
 \frac1N\Phi(t)\tilde U(T)\Phi^\top(t)\geq0,\quad\Phi(s)\nabla^2_{xx}\tilde H_{N}\big(s,X(s),\nabla_xV^\varepsilon_N(s,X(s))\big)\Phi^\top(s)\geq0,\quad s\in[t,T].
\end{align*}
Moreover,
\begin{align}\label{riccati-2}
 \mathbb E(t)\bigg[\frac1N\Phi_T\tilde U(T)\Phi^\top_T+\int_t^T\Phi(s)\nabla^2_{xx}\tilde H_{N}\big(s,X(s),\nabla_xV^\varepsilon_N(s,X(s))\big)\Phi^\top(s)ds\bigg]\geq0.
\end{align}
In other words, the right hand side of \eqref{riccati-1} is a (random) positive semi-definite matrix. Now we may take $t=T-\delta$ in \eqref{riccati-1} and combine \eqref{riccati-0}, \eqref{riccati-2} to obtain $Y_{T-\delta}\geq0$. In view of \eqref{FBSDE-Y}, we have $\nabla^2_{xx}V^\varepsilon_N\geq0$ and hence
\begin{align*}
 \sum_{i,j=1}^NV^{\varepsilon,ij}_N(T-\delta,x)\xi_i\xi_j\geq0.
\end{align*}
One the other hand, according to {\bf Hypothesis (R1)} and \eqref{H-N-R}, we have
\begin{align*}
  \nabla^2_{xx}\tilde H_N\big(s,X(s),\nabla_xV^\varepsilon_N(s,X(s))\big),\ -\nabla^2_{pp}\tilde H_N\big(s,X(s),\nabla_xV^\varepsilon_N(s,X(s))\big)\geq0.
\end{align*}
Hence for any $\alpha\in\mathbb R^N$ satisfying $|\alpha|=1$,
\begin{align*}
  0&\leq\alpha^\top Y(t)\alpha\leq\mathbb E_t\bigg[\frac1N\alpha^\top\tilde U(T)\alpha+\int_t^T\Big[\alpha^\top Y(s)\nabla^2_{xp}\tilde H_N\big(s,X(s),\nabla_xV^\varepsilon_N(s,X(s))\big)\alpha\notag\\
  &\quad+\alpha^\top\nabla^2_{px}\tilde H_N\big(s,X(s),\nabla_xV^\varepsilon_N(s,X(s))\big)Y(s)\alpha+\alpha^\top\nabla^2_{xx}\tilde H_N\big(s,X(s),\nabla_xV^\varepsilon_N(s,X(s))\big)\alpha\Big] ds\bigg].
\end{align*}
Moreover, in view of Lemma \ref{coeff.-1}  and $Y(s)\geq0$, we have
\begin{align*}
 &\quad\alpha^\top Y(s)\nabla^2_{xp}\tilde H_N\big(s,X(s),\nabla_xV^\varepsilon_N(s,X(s))\big)\alpha\leq\big|Y(s)\alpha\big|\cdot\big|\nabla^2_{xp}\tilde H_N\big(s,X(s),\nabla_xV^\varepsilon_N(s,X(s))\big)\alpha\big|\notag\\
 &\leq\tilde C_3\big|Y(s)\alpha\big|\cdot\big|\alpha\big|\leq\tilde C_3\sup_{|\beta|=1}\beta^\top Y(s)\beta,
\end{align*}
as well as
\begin{align*}
  \alpha^\top\nabla^2_{xx}\tilde H_N\big(s,X(s),\nabla_xV^\varepsilon_N(s,X(s))\big)\alpha\leq\frac{\tilde C_5}N|\alpha|^2=\frac{\tilde C_5}N,\quad\frac1N\alpha^\top\tilde U(T)\alpha\leq\frac{C^U_{20}}N,
\end{align*}
where we recall that $C^U_{20}$ is from \eqref{assumption-L-U}. Hence
\begin{align*}
  \sup_{|\beta|=1}\beta^\top Y(t)\beta\leq\frac{C^U_{20}}N+\frac{\tilde C_5T}N+2\tilde C_3\mathbb E_t\bigg[\int_t^T\bigg(\sup_{|\beta|=1}\beta^\top Y(s)\beta\bigg)ds\bigg].
\end{align*}
Therefore we may deduce the existence of $\tilde C_6=\tilde C_6(f,\lambda^{-\frac12},T,L,U,(\lambda-\lambda_0)^{-1})$ such that
\begin{align*}
\mathbb E\bigg[ \sup_{|\beta|=1}\beta^\top Y(t)\beta\bigg]\leq\frac{\tilde C_6}N,\quad t\in[T-\delta,T].
\end{align*}
The inequality above implies that
\begin{align*}
  \sum_{i,j=1}^NV^{\varepsilon,ij}_N(t,x)\xi_i\xi_j\leq\frac{\tilde C_6}N|\xi|^2,\quad\xi\in\mathbb R^N,\quad(t,x)\in[T-\delta,T]\times\mathbb R^N,
\end{align*}
and we have completed the proof.
\end{proof}
\begin{remark}\label{explanation}
\begin{enumerate}
\item  To see why we confine ourselves to the case where $\partial^2_{x_ix_j}V^\varepsilon_N(t,x)$ $(1\leq i,j\leq N)$ are bounded, one might turn to the definition of the matrix valued process $\Phi_t$. If we do not assume that $Y(t)$ is bounded, then we cannot ensure the integrability of $\Phi_t$. Without the integrability of $\Phi_t$, we can not do the calculations in \eqref{riccati} and below, since they all involve taking conditional expectation.
 \item For now, in this Lemma \ref{prior-bound}, the existence of the constants $\delta$ and $\breve C$ is merely an assumption. But we know from Proposition \ref{dY/dx-1-0} and Proposition \ref{1st-deri-2-1} that $\delta$ indeed exists and is at least $\tilde c$, so does $\breve C$. In the next Proposition \ref{dY/dx-1}, we will use the refined estimate \eqref{global-eigen.} to show that $\delta=T$. Moreover, showing that $\delta=T$ will then in turn gives us the refined estimate \eqref{global-eigen.} on $[0,T]$.

\end{enumerate}
\end{remark}

We finish this section with the next proposition where the extra assumption on boundedness in Lemma  \ref{prior-bound} is removed. The main idea is to take advantage of the refined estimate in \eqref{global-eigen.} while utilizing a suitable `continuity' method.
\begin{proposition}\label{dY/dx-1}
  { Let constants $\Gamma^0_1,\ \Gamma^0_2>0$ be such that $\big(x,\nabla_xV^\varepsilon_N(t,x)\big)\in\mathcal A_{N,\Gamma^0_1,\Gamma^0_2}$ for all $\varepsilon>0,\ (t,x)\in[0,T]\times\mathbb R^N$. Suppose that {\bf Hypothesis (R1)} holds with $(\Gamma_1,\Gamma_2)=(\Gamma^0_1,\Gamma^0_2)$}. There exists a constant $\tilde C_6=\tilde C_6(f,\lambda^{-\frac12},T,L,U,(\lambda-\lambda_0)^{-1})$ such that for $1\leq i,j\leq N$,
  \begin{align*}
    0\leq\sum_{i,j=1}^NV^{\varepsilon,ij}_N(t,x)\xi_i\xi_j\leq\frac{\tilde C_6}N|\xi|^2,\quad\xi\in\mathbb R^N,\ (t,x)\in[0,T]\times\mathbb R^N.
  \end{align*}
\end{proposition}
\begin{proof}
{ Let $\gamma_0>0$ be sufficiently small such that $\lambda>(1+\gamma_0)\lambda_0$ for $\lambda_0$ from \eqref{hypo-regu-2}. By definition, {\bf Hypothesis (R1)} also holds with $(\Gamma_1,\Gamma_2)=\big((1+\gamma_0)\Gamma^0_1,(1+\gamma_0)\Gamma^0_2\big)$.} We have from Proposition \ref{dY/dx-1-0} that, for $0\leq T-t\leq\tilde c$ and $x\in\mathbb R^N$, $\partial^2_{x_ix_j}V^{\varepsilon,R_1}_N(t,x)$, $1\leq i,j\leq N$ are uniformly bounded by $\tilde C_4+1$  independent of $R_1$. In view of the convergence of $\partial^2_{x_ix_j}V^{\varepsilon,R_1}_N(t,x)$ to $\partial^2_{x_ix_j}V^\varepsilon_N(t,x)$, as $R_{1}\to+\infty$ in Proposition \ref{1st-deri-2-1}, we obtain that $\partial^2_{x_ix_j}V^\varepsilon_N(t,x)$, $1\leq i,j\leq N$ are bounded on $(t,x)\in[T-\tilde c,T]\times\mathbb R^N$. Therefore we have both \eqref{1st-deri-scaling-2} and \eqref{global-eigen.} on $(t,x)\in[T-\tilde c,T]\times\mathbb R^N$ from Proposition \ref{1st-deri-2-1} and Lemma \ref{prior-bound}. 

Next we replace $\frac1N\sum_{i=1}^NU_{R_1}(x_i)$ { with $V^\varepsilon_N\big(T-\tilde c,\rho_{N,R}(x)\big)$} in \eqref{glb-1}, \eqref{glb-2} (we impose some specific properties on $\rho$ below) and consider the following coupled PDE system on a time interval $(T-\tilde c - c,T-\tilde c)$, where $c>0$ is a small number which will be specified later (written for $\hat V^{\varepsilon,R_1}_N$)
\begin{align}\label{glb-1-1}
\left\{\begin{aligned}&\partial_t\hat V^{\varepsilon,R_1}_N+\frac{\sigma^2}2\sum_{i,j=1}^N\partial^2_{x_ix_j}\hat V^{\varepsilon,R_1}_N+\frac{\varepsilon^2}2\sum_{i=1}^N\partial^2_{x_ix_i}\hat V^{\varepsilon,R_1}_N+\tilde H^{R_1}_N(t,x,\nabla_x\hat V^{\varepsilon,R_1}_N)=0,\\
&\hat V^{\varepsilon,R_1}_N(T-\tilde c,x_1,\ldots,x_N)={ V^\varepsilon_N\big(T-\tilde c,\rho_{N,R}(x)\big)},
\end{aligned}\right.
\end{align}
as well as
\begin{align}\label{glb-2-1}
\left\{\begin{aligned}&\partial_t\hat V^{\varepsilon,R_1,kl}_N+\frac{\sigma^2}2\sum_{i,j=1}^N\partial^2_{x_ix_j}\hat V^{\varepsilon,kl}_N+\frac{\varepsilon^2}2\sum_{i=1}^N\partial^2_{x_ix_i}\hat V^{\varepsilon,kl}_N+\partial^2_{x_kx_l}\tilde H^{R_1}_N(x,\nabla_x\hat V^{\varepsilon,R_1}_N)\\
&+\sum_{i=1}^N\partial_{p_i}\tilde H^{R_1}_N(x,\nabla_x\hat V^{\varepsilon,R_1}_N)\partial_{x_i}\hat V^{\varepsilon,R_1,kl}_N+\sum_{i,j=1}^N\partial^2_{p_ip_j}\tilde H^{R_1}_N(x,\nabla_x\hat V^{\varepsilon,R_1}_N)\hat V^{\varepsilon,R_1,ki}_N\hat V^{\varepsilon,R_1,jl}_N\\
& +\sum_{i=1}^N\partial^2_{x_lp_i}\tilde H^{R_1}_N (x,\nabla_x\hat V^{\varepsilon,R_1}_N)\hat V^{\varepsilon,R_1,ki}_N+ \sum_{i=1}^N\partial^2_{x_kp_i}\tilde H^{R_1}_N (x,\nabla_x\hat V^{\varepsilon,R_1}_N)\hat V^{\varepsilon,R_1,li}_N\\
&=0,\\
&\hat V^{\varepsilon,R_1,kl}_N(T-\tilde c,x)={ \partial^2_{x_kx_l}\big[V^\varepsilon_N\big(T-\tilde c,\rho_{N,R}(x)\big)\big]},\quad1\leq k,l\leq N.
\end{aligned}\right.
\end{align}
{ Here $\rho_{N,R}:\ \mathbb R^N\to\mathbb R^N$ is a truncation that is smooth and satisfies for $x=(x_1,\ldots,x_N)\in\mathbb R^N$,
\begin{align}\label{truncation}
 \rho_{N,R}(x)=\big(\rho_R(x_1),\ldots,\rho_R(x_N)\big),\quad0\leq\rho'_R\leq1,\quad\rho_R(y)\in\left\{\begin{aligned}&\{-R-1\},\quad y\in(-\infty,-R-2),\\
 &[-R-1,-R],\quad y\in[-R-2,R],\\
 &\{y\},\quad y\in[-R,R],\\
 &[R,R+1],\quad y\in[R,R+2],\\
 &\{R+1\},\quad y\in(R+2,+\infty).
 \end{aligned}\right.
\end{align}
In view of the definition of $\rho_{N,R}$, since $\big(x,\nabla_xV^\varepsilon_N(t,x)\big)\in\mathcal A_{N,\Gamma^0_1,\Gamma^0_2}$, $(t,x)\in[0,T]\times\mathbb R^N$, 
\begin{align*}
 G_N\big(x,\Gamma^0_1\big)\leq{ \partial_{x_k}\big[V^\varepsilon_N\big(T-\tilde c,\rho_{N,R}(x)\big)\big]}\leq G_N\big(x,\Gamma^0_2\big).
\end{align*}
Following a similar contraction argument to that in Lemma \ref{1st-deri} and Lemma \ref{forward-scaling-eqn}, we obtain that for sufficiently small $\gamma>0$, there exists time duraion $\hat c$ depending on $\gamma$ such that,\ \\
\begin{align*}
 &G_N\big(x,(1+\gamma)\Gamma^0_1\big)\leq Y^i_N(s)=\partial_{x_i}\hat V^{\varepsilon,R_1}_N(s,x)\leq G_N\big(x,(1+\gamma)\Gamma^0_2\big),\ s\in[T-\tilde c-\hat c,T-\tilde c],\\
 &1\leq i\leq N,\quad x\in\mathbb R^N.
\end{align*}
Let $\gamma<\gamma_0$. According to the above inequality and our assumption that {\bf Hypothesis (R1)} holds with $(\Gamma_1,\Gamma_2)=(\Gamma^0_1,\Gamma^0_2)$, we may apply Lemma \ref{coeff.} and show that the estimates in \eqref{hamilton-001} holds when $(t,x,p)$ takes the following value:
\begin{align*}
 (t,x,p)=\big(s,x,\nabla_x\hat V^{\varepsilon,R_1}_N(s,x)\big),\quad(s,x)\in[T-\tilde c-\hat c,T-\tilde c]\times\mathbb R^N.
\end{align*}
and the bound therein is replaced with $\tilde C_3+\hat\gamma$ for some $\hat\gamma>0$.}

Next, we may use a contraction method similar to the one in the proof of Proposition \ref{dY/dx-1-0} to show the existence of $c>0$, depending on $\tilde C_2(C^L_{10}+C^U_{10})$ in \eqref{1st-deri-scaling-2}, $\tilde C_6$ in \eqref{global-eigen.} { as well as $N,\ \gamma,\ \hat\gamma$}, such that the solution $\hat V^{\varepsilon,R_1,kl}_N$ to \eqref{glb-2-1} is unique and bounded on $(t,x)\in[T-\tilde c-c,T-\tilde c]\times\mathbb R^N$ uniformly in $R_1$. We may also argue similarly to Proposition \ref{1st-deri-2-1} to obtain that
\begin{align*}
  \lim_{R_1\to+\infty}\partial^2_{x_kx_l}\hat V^{\varepsilon,R_1,kl}_N(t,x)=\partial^2_{x_kx_l}V^{\varepsilon,kl}_N(t,x),
\end{align*}
where $(t,x)\in[T-\tilde c-c,T-\tilde c]\times\mathbb R^N,\ 1\leq k,l\leq N$. In particular, we have shown that $\partial^2_{x_ix_j}V^\varepsilon_N(t,x)$, $1\leq i,j\leq N$ are bounded on $t\in[T-\tilde c-c,T]$. Then Proposition \ref{1st-deri-2-1} and Lemma \ref{prior-bound} again yield both \eqref{1st-deri-scaling-2} and \eqref{global-eigen.} on $(t,x)\in[T-\tilde c-c,T-\tilde c]\times\mathbb R^N$. 

It is important to notice that  $V^\varepsilon_N\big(T-\tilde c,\rho_{N,R}(x)\big)$, as the terminal condition of \eqref{glb-2-1}, is only used to show the boundedness of $\partial^2_{x_ix_j}V^\varepsilon_N(t,x)$ but not \eqref{global-eigen.}. We are relying the convexity of the final datum only after passing to the imit $R_1\to+\infty$.

Now we can replace $U_{R_1}(x)$ with $V^\varepsilon_N(T-\tilde c-c,x)$ in \eqref{glb-1}, \eqref{glb-2} and repeat the procedure above to prove \eqref{1st-deri-scaling-2} and \eqref{global-eigen.} on $(t,x)\in[T-\tilde c-2c,T-\tilde c-c]\times\mathbb R^N$. After finite such repetition we can show \eqref{global-eigen.} on $(t,x)\in[0,T]\times\mathbb R^N$.
\end{proof}

\subsection{Convergence of auxiliary problems}\label{Convergence of auxiliary problems}
In this section, we study the original problem associated to the HJB equation \eqref{HJB-V-N}. Thanks to the uniform estimates in the previous sections, we may obtain the desired solution by extracting a convergent subsequence from the families $(V^{\varepsilon,R_1}_N)_{{\varepsilon,R_{1}}}$ and $(V^{\varepsilon}_N)_{\varepsilon}$ which solve \eqref{glb-1} and \eqref{glb-3}, respectively. More importantly, the resulting limits inherit the estimates (uniform in $N$) satisfied by $V^{\varepsilon,R_1}_N$ and $V^{\varepsilon}_N$.

For short time, we have the following result on the well-posedness {of \eqref{HJB-V-N}} as well as the corresponding estimates.
\begin{theorem}\label{veri}
  Suppose that the assumptions of Lemma \ref{1st-deri-2-0-1} take place.
  Let $\tilde c>0$ given in Proposition \ref{dY/dx-1-0}. For $T<\tilde c$, the original HJB equation \eqref{HJB-V-N} admits a solution $V_N\in W^{1,2,\infty}_{loc}([0,T]\times\mathbb R^N)$, satisfying for $1\leq i,j\leq N$, $(t,x)\in[0,T]\times\mathbb R^N$,
\begin{align}\label{VN-scaling-1}
    \left|\partial_{x_i}V_N(t,x)\right|&\leq\frac{\tilde C_2(C^L_{11}+C^U_{11})}N\bigg(1+|x_i|^2+\frac1N\sum_{k=1}^N|x_k|^2\bigg)^\frac12+\frac{\tilde C_2(C^L_{10}+C^U_{10})}N,
\end{align}
and
\begin{align}\label{VN-scaling-2}
\left|\partial^2_{x_ix_j}V_N(t,x)\right|\leq\tilde C_4N^{-1}(\delta_{ij}+N^{-1})\quad{{\rm{a.e.}}}
\end{align}
  Moreover, such solution $V_N$ is characterized by the value function in \eqref{def-V-N-O} and thus it is unique. The unique optimal feedback function is
\begin{align}\label{def-theta-*}
\theta^*_N\big(t,x\big)&:=\lim_{R_1\to+\infty}\theta^{R_1}_N\big(t,x,\nabla_xV_N(t,x)\big)\notag\\
&\in\arg\min_{\theta\in\Theta}\bigg\{\frac\lambda2\big|\theta\big|^2+\sum_{i=1}^N f\bigg(t,\theta,x_i,\frac1N\sum_{j=1}^N\delta_{x_j}\bigg)\partial_{x_i}V_N(t,x)\bigg\},
\end{align}
where $\theta^{R_1}_N(t,p,q)$ is defined in \eqref{theta-*}.
\end{theorem}
\begin{proof}
Rewrite \eqref{glb-1} as follow
\begin{align*}%\label{HJB-V-N-2}
&-\frac{\varepsilon^2}2\sum_{i=1}^N\partial^2_{x_ix_i}V^{\varepsilon,R_1}_N=\partial_tV^{\varepsilon,R_1}_N+\frac{\sigma^2}2\sum_{i,j=1}^N\partial^2_{x_ix_j}V^{\varepsilon,R}_N\notag\\
  &\quad+\inf_{\theta\in\Theta}\bigg\{\frac\lambda2\big|\theta\big|^2+\sum_{i=1}^Nf\bigg(t,\theta,x_i,\frac1N\sum_{j=1}^N\delta_{x_j}\bigg)\partial_{x_i}V^{\varepsilon,R_1}_N\bigg\}+\frac1 N\sum_{i=1}^NL_{R_1}(x_i).
\end{align*}
In view of the uniform estimates in Lemma \ref{1st-deri-2-0} and Proposition \ref{dY/dx-1-0} 
let
\begin{align*}
\varepsilon\to0+,\ R_1\to+\infty,
\end{align*}
we immediately have the existence of $V_N\in W^{1,2,\infty}_{loc}([0,T)\times\mathbb R^N)$ such that on any compact subset of $[0,T)\times\mathbb R^N$, $V^{\varepsilon,R_1}_N$ and $\nabla_xV^{\varepsilon,R_1}_N$ converge (up to a subsequence) uniformly to $V_N$ and $DV_N$ whereas $\partial_tV^{\varepsilon,R_1}_N$, $\nabla^2_xV^{\varepsilon,R_1}_N$ converges weakly to $\partial_tV_N$, $\nabla^2_xV_N$. Moreover, $V_N$ also satisfies the corresponding local estimates \eqref{1st-deri-scaling-2}, \eqref{short-time-2nd-deri} of $V^{\varepsilon,R_1}_N$, hence \eqref{VN-scaling-1} and \eqref{VN-scaling-2} is valid.

According to \eqref{short-time-2nd-deri},
\begin{align*}
\left|\frac{\varepsilon^2}2\sum_{i=1}^N\partial^2_{x_ix_i}V^{\varepsilon,R_1}_N\right|\leq\frac{\varepsilon^2}2\tilde C_4.
\end{align*}
Sending $\varepsilon$ to $0+$, $R_1$ to $+\infty$ in \eqref{glb-1}, we get
\begin{align}\label{HJB-V-N-3}
  \partial_tV_N+\frac{\sigma^2}2\sum_{i,j=1}^N\partial^2_{x_ix_j}V_N+\inf_{\theta\in\Theta}\bigg\{\frac\lambda2\big|\theta\big|^2+\sum_{i=1}^Nf\bigg(t,\theta,x_i,\frac1N\sum_{j=1}^N\delta_{x_j}\bigg)\partial_{x_i}V_N\bigg\}+\frac1 N\sum_{i=1}^NL(x_i)=0,
\end{align}
in the distributional sense.

  To show the uniqueness, it suffices to establish the verification result that any solution $V_N\in W^{1,2,\infty}_{loc}([0,T]\times\mathbb R^N)$ satisfying \eqref{VN-scaling-1} and \eqref{VN-scaling-2} equals the value function in \eqref{def-V-N-O}. Consider any $\theta\in\mathcal U^{ad}$, as well as the corresponding $X^{\theta,i}(t)$ in \eqref{sample-dyn} and $X^{\varepsilon,\theta,i}(t)$ in \eqref{sample-dyn-1}. The generalized It\^{o}'s formula (see e.g. \cite{Krylov1980}) gives that, for any bounded domain $D\subset\mathbb R^N$, denoting by $\tau_D$ the corresponding exit time of $\mathbf X^{\varepsilon,\theta}_N(t)$,
\begin{align*}
&\quad V_N(T\wedge\tau_D,\mathbf X^{\varepsilon,\theta}_N(T\wedge\tau_D))\notag\\
    &=V_N(0,\mathbf X^{\varepsilon,\theta}_N(0))+\int_0^{T\wedge\tau_D}\mathcal L^\varepsilon_tV_N(t)dt+\sigma \sum_{i=1}^N\int_0^{T\wedge\tau_D}\partial_{x_i}V_N(t,\mathbf X^{\varepsilon,\theta}_N(t))dW^0(t)\notag\\
&\quad+\varepsilon\sum_{i=1}^N\int_0^{T\wedge\tau_D}\partial_{x_i}V_N(t,\mathbf X^{\varepsilon,\theta}_N(t))dW^i(t)
\end{align*}
For the ease of notation, we have adopted the notation
\begin{align*}
    &\mathcal L^\varepsilon_tV_N(t)=\partial_tV_N(t,\mathbf X^{\varepsilon,\theta}_N(t))+\frac{\sigma^2}2\sum_{i,j=1}^N\partial^2_{x_ix_j}V_N(t,\mathbf X^{\varepsilon,\theta}_N(t))+\frac{\varepsilon^2}2\sum_{i=1}^N\partial^2_{x_ix_i}V_N(t,\mathbf X^{\varepsilon,\theta}_N(t))\notag\\
    &\quad+\sum_{i=1}^Nf\bigg(t, \theta(t),X^{\varepsilon,\theta,i}_N(t),\frac{1}{N}\sum_{j=1}^N\rho(X^{\varepsilon,\theta,j}_N(t))\bigg)\partial_{x_i}V_N(t,\mathbf X^{\varepsilon,\theta}_N(t)).
\end{align*}
  According to Proposition \ref{dY/dx-1-0} as well as the convergence of $\nabla_xV^{\varepsilon,R_1}_N(t,x)$ to $\nabla_xV_N(t,x)$, $\nabla_xV_N(t,x)$ is continuous and uniformly bounded on $D$. Hence
\begin{align}\label{ineq-0}
    \mathbb E\left[V_N(T\wedge\tau_D,\mathbf X^{\varepsilon,\theta}_N(T\wedge\tau_D))\right]=V_N(0,\mathbf X^{\varepsilon,\theta}_N(0))+\mathbb E\bigg[\int_0^{T\wedge\tau_D}\mathcal L^\varepsilon_tV_N(t)dt\bigg].
\end{align}
According to Theorem 2.10.2 in \cite{Krylov1980} and \eqref{VN-scaling-2}, \eqref{HJB-V-N-3},
\begin{align}\label{ineq-1}
    \mathbb E\bigg[\int_0^{T\wedge\tau_D}\mathcal L^\varepsilon_tV_N(t)dt\bigg]&\leq\frac{\varepsilon^2}2\mathbb E\left[\int_0^{T\wedge\tau_D}\sum_{i=1}^N\partial^2_{x_ix_i}V_N(t,\mathbf X^{\varepsilon,\theta}_N(t))dt\right]\notag\\
    &\quad-\mathbb E\bigg[\frac\lambda2\int_0^{T\wedge\tau_D}\big|\theta(t)\big|^2dt+\frac1 N\sum_{i=1}^N\int_0^{T\wedge\tau_D}L(X^{\varepsilon,\theta,i}_N(t))dt\bigg]\notag\\
    &\leq\frac{\varepsilon^2}2\tilde C_4-\mathbb E\bigg[\frac\lambda2\int_0^{T\wedge\tau_D}\big|\theta(t)\big|^2dt+\frac1 N\sum_{i=1}^N\int_0^{T\wedge\tau_D}L(X^{\varepsilon,\theta,i}_N(t))dt\bigg].
\end{align}
Plug \eqref{ineq-1} into \eqref{ineq-0}, and let $D$ extend to $\mathbb R^N$, the monotone convergence theorem yields that
\begin{align}\label{ineq-3}
    \mathbb E\left[\frac1 N\sum_{i=1}^NU\big(X^{\varepsilon,\theta,i}_N(T)\big)+\frac1 N\sum_{i=1}^N\int_0^TL\big(X^{\varepsilon,\theta,i}_N(t)\big)dt+\frac\lambda2\int_0^T\big|\theta(t)\big|^2dt\right]\leq V_N\big(0,\mathbf X^{\varepsilon,\theta}_N(0)\big)+\frac{\varepsilon^2}2\tilde C_4.
\end{align}
Sending $\varepsilon$ to $0+$ and noticing the convergence of $X^{\varepsilon,\theta,i}$ to $X^{\theta,i}$, we have
\begin{align}\label{ineq-4}
    &\quad\mathbb E\left[\frac1 N\sum_{i=1}^NU\big(X^{\theta,i}_N(T)\big)+\frac1 N\sum_{i=1}^N\int_0^TL\big(X^{\theta,i}_N(t)\big)dt+\frac\lambda2\int_0^T\big|\theta(t)\big|^2dt\right]\leq V_N\big(0,\mathbf X^\theta_N(0)\big).
\end{align}
  On the other hand, consider the candidate optimal feedback control $\theta^*_N\big(t,x,\nabla_xV_N(t,x)\big)$. We first claim that the corresponding system
\begin{align}\label{optimal-sde}
    dX^{*,i}_N(t) = f\bigg(t,\theta^*_N(t,X^*_N),X^{*,i}_N(t),\frac{1}{N}\sum_{j=1}^N\delta_{X^{*,j}_N(t)}\bigg)dt + \sigma dW^0(t),\ i=1,\ldots,N.
\end{align}
  admits a unique solution for any initial data $x_1,\ldots,x_N$, $N\geq1$. In fact, it is easy to see from \eqref{def-theta-*} and Lemma \ref{theta-Lip} that $\theta^*_N(t,p,q)$ is locally Lipschitz continuous with respect to $(p,q)\in {\mathcal A}_N$. In the same time, $(x,\nabla_xV_N(t,x))\in {\mathcal A}_N$ and $V_N\in W^{1,2,\infty}_{loc}([0,T]\times\mathbb R^N)$. Therefore, after composition,
\begin{align*}
x\mapsto f\bigg(t,\theta^*_N(t,x,\nabla_xV_N(t,x)),x_i,\frac{1}{N}\sum_{j=1}^N\delta_{x_j}\bigg),\quad i=1,\ldots,N,
\end{align*}
  is locally Lipschitz continuous. The local Lipschitz continuity then gives the strong uniqueness of the solution. Notice that we have got \eqref{VN-scaling-1}, the weak existence can be deduced from \eqref{def-theta-*} and from the linear growth property that
\begin{align*}
\bigg|f\bigg(t,\theta^*_N(t,x,\nabla_xV_N(t,x)),x_1,\frac{1}{N}\sum_{j=1}^N\rho(x_j)\bigg)\bigg|\leq C_N(1+|x|),
\end{align*}
for some constant $C_N$.

  Having shown the well-posedness of \eqref{optimal-sde}, the first ``$\leq$'' in \eqref{ineq-1} becomes ``$=$''. The estimates in \eqref{VN-scaling-2} then enable us to replace the ``$\leq$'' in \eqref{ineq-4} with ``$\geq$'', implying that $\theta^*_N$ is optimal and $V_N$ is the value function.
\end{proof}
Using the same method as in Theorem \ref{veri} and combining with the uniform estimates in {Proposition \ref{1st-deri-2-1}}, Proposition \ref{dY/dx-1}, we can prove the following result for long time.
\begin{theorem}\label{veri-1}
  Suppose that the assumptions of Proposition \ref{dY/dx-1} take place.
  The original HJB equation \eqref{HJB-V-N} admits a solution $V_N\in W^{1,2,\infty}_{loc}([0,T]\times\mathbb R^N)$, satisfying for $1\leq i,j\leq N$, $(t,x)\in[0,T]\times\mathbb R^N$,
\begin{align}\label{VN-scaling-3}
    \left|\partial_{x_i}V_N(t,x)\right|&\leq\frac{\tilde C_2(C^L_{11}+C^U_{11})}N\bigg(1+|x_i|^2+\frac1N\sum_{k=1}^N|x_k|^2\bigg)^\frac12+\frac{\tilde C_2(C^L_{10}+C^U_{10})}N,
\end{align}
and
\begin{align}\label{VN-scaling-4}
\quad0\leq\sum_{i,j=1}^N\partial^2_{x_ix_j}V_N(t,x)\xi_i\xi_j\leq\frac{\tilde C_6}N|\xi|^2\quad\text{a.e.}.
\end{align}
  Moreover, such solution $V_N$ is characterized by the value function in \eqref{def-V-N-O} and thus unique. An optimal feedback function is
\begin{align*}
\theta^*_N\big(t,x\big)\in\arg\min_{\theta\in\Theta}\bigg\{\frac\lambda2\big|\theta\big|^2+\sum_{i=1}^Nf\bigg(t,\theta,x_i,\frac1N\sum_{j=1}^N\delta_{x_j})\bigg)\partial_{x_i}V_N(t,x)\bigg\}.
\end{align*}
\end{theorem}
\section{Discussion on the convergence rate}\label{Discussion on convergence rate}

In this section we discuss the convergence rate for the value functions $V_N$ as well as the minimizer $\theta^*_N$ where the number of samples $N$ goes to infinity. In terms of neural SDEs, the convergence of $V_N$ above is instantly interpreted as the convergence of minima of objective functionals, while we may  use the convergence of $\theta^*_N$ above to yield pathwise convergence results that imply the convergence of optimal parameters obtained via neural SDE with $N$ samples (see Proposition \ref{cauchy-sequence-1-0} and Proposition \ref{cauchy-sequence-0-0} below). We recall that for sufficiently large $N$, the conclusion in Theorem \ref{veri} holds as long as $T<\tilde c$, while the conclusion in Theorem \ref{veri-1} holds for any $T>0$.

We first show the interesting fact that the value function $V_N$ of Problem \ref{problem} is actually the finite dimensional projection of a function ${\cal V}$ defined on the set of probability measure. The following lemma implies that $\mathcal V$ in Definition \ref{def-emp-meas} is well-defined.
\begin{lemma}\label{well-define}
   Suppose {\bf Hypothesis (R)}. Let $V_N$ be the value function in \eqref{def-V-N-O}. For samples $x_1,\ldots,x_N\in\mathbb R$ and $y_1,\ldots,y_M\in\mathbb R$, (for $M,N\in\mathbb N$) suppose that
\begin{align*}
\frac1N\sum_{i=1}^N\delta_{x_i}=\frac1M\sum_{i=1}^M\delta_{y_i},
\end{align*}
then for $t\in[0,T]$, $T>0$,
\begin{align*}
V_N(t,x_1,\ldots,x_N)=V_M(t,y_1,\ldots,y_M).
\end{align*}
\end{lemma}

\begin{proof}
In view of \eqref{approximation-1} and \eqref{approximation-2}, it suffices to show
\begin{align*}
  V^{0,R}_N(t,x_1,\ldots,x_N)=V^{0,R}_M(t,y_1,\ldots,y_M),
\end{align*}
for any $R_1,R_2>0$. Here we have defined the value function
\begin{align}\label{def-V-N-0}
V^{0,R}_N(t,x_1,\ldots,x_N):=\inf_{\theta\in\mathcal U^{ad}_{t,R_2}}J^{0,R_1}_N(\theta,t,x_1,\ldots,x_N).
\end{align}
Note that
\begin{align*}
\frac1{NM}\sum_{i=1}^NM\delta_{x_i}=\frac1N\sum_{i=1}^N\delta_{x_i}=\frac1M\sum_{i=1}^M\delta_{y_i}=\frac1{NM}\sum_{i=1}^MN\delta_{y_i}.
\end{align*}
Since the left hand side and the right hand side have the same sample size, it holds that
\begin{align*}
\left\{x_1,\ldots,x_N,x_1,\ldots,x_N,\ldots,x_1,\ldots,x_N\right\}=\left\{y_1,\ldots,y_M,y_1,\ldots,y_M,\ldots,y_1,\ldots,y_M\right\}.
\end{align*}
  Here the left hand side above consists of $M$ duplicates of $\left\{x_1,\ldots,x_N\right\}$, while the right hand side above consists of $N$ duplicates of $\left\{y_1,\ldots,y_M\right\}$. According to the variational definition of $V^{0,R}_{MN}$, it is easy to check the symmetric feature that
\begin{align*}
V^{0,R}_{NM}(t,x_1\mathbf1^\top_M,\ldots,x_N\mathbf1^\top_M)=V^{0,R}_{NM}(t,y_1,\ldots,y_M,y_1,\ldots,y_M,\ldots,y_1,\ldots,y_M),
\end{align*}
{where
\begin{align*}
  \mathbf1_M:={\underbrace{(1,\ldots,1)}_{M\ -\ \text{times}}}^\top.
\end{align*}}
Therefore it suffices to show that
\begin{align}\label{duplication-equality}
V^{0,R}_N(t,x_1,\ldots,x_N)=V^{0,R}_{NM}(t,x_1\mathbf1^\top_M,\ldots,x_N\mathbf1^\top_M).
\end{align}
For any continuous $\theta\in\mathcal U^{ad}_{t,R_2}$, define the following particle systems
\begin{align*}\left\{\begin{aligned}
  &d\tilde X^{(k-1)M+l}_{NM}(s)=f\bigg(s,\theta(s),\tilde X^{(k-1)M+l}_{NM}(s),\frac1{NM}\sum_{i=1}^{NM}\delta_{\tilde X^i_{NM}(s)}\bigg)ds+\sigma dW^0(s),\\
  &\tilde X^{(k-1)M+l}_{NM}(t)=x_k,\quad1\leq k\leq N,\ 1\leq l\leq M.
  \end{aligned}\right.
\end{align*}
Now that $\theta$ is a bounded process,  the solution admits strong uniqueness. Taking advantage of the symmetry and the strong uniqueness, it is easy to verify that for $s\in[t,T]$, the only solution to the above SDE satisfies
\begin{align}\label{group}
  \tilde X^{(k-1)M+l_1}_{NM}(s)=\tilde X^{(k-1)M+l_2}_{NM}(s),\quad1\leq k\leq N,\ 1\leq l_1,l_2\leq M.
\end{align}
Denote by
\begin{align*}
  X^k_N(s):=\tilde X^{(k-1)M+1}_{NM}(s),\quad1\leq k\leq N.
\end{align*}
In view of \eqref{group}, we have
\begin{align*}
  \frac1{MN}\sum_{i=1}^{NM}\delta_{\tilde X^i_{NM}(t)}=\frac1N\sum_{i=1}^N\delta_{X^i_N(t)}.
\end{align*}
Moreover, $(X^1_N(s),\ldots,X^N_N(s))$ uniquely solves
\begin{align*}\left\{\begin{aligned}
  &dX^i_N(s)=f\bigg(s,\theta(s),X^i_N(s),\frac1N\sum_{i=1}^N\delta_{X^i_N(s)}\bigg)ds+\sigma dW^0(s),\\
  &X^i_N(t)=x_i,\quad1\leq i\leq N.
  \end{aligned}\right.
\end{align*}
Therefore
\begin{align*}
  &\quad J^{0,R_1}_N(t,\theta,x_1,\ldots,x_N)=\mathbb E\bigg[\frac1 N\sum_{i=1}^N\int_t^TL_{R_1}\big(X^i_N(s)\big)ds+\frac1N\sum_{i=1}^NU_{R_1}\big(X^i_N(T)\big)+\frac\lambda2\int_t^T\left|\theta(s)\right|^2ds\bigg]\notag\\
  &=\mathbb E\bigg[\frac1 {NM}\sum_{i=1}^{NM}\int_t^TL_{R_1}\big(\tilde X^i_{NM}(s)\big)ds+\frac1{NM}\sum_{i=1}^{NM}U_{R_1}\big(\tilde X^i_{NM}(T)\big)+\frac\lambda2\int_t^T\left|\theta(s)\right|^2ds\bigg]\notag\\
  &=J^{0,R_1}_{NM}\big(t,\theta,x_1\mathbf1^\top_M,\ldots,x_N\mathbf1^\top_M\big).
\end{align*}
Since $\theta$ is taken arbitrarily from $\mathcal U^{ad}_{t,R_2}$, we have \eqref{duplication-equality}.
\end{proof}
Let us now turn to the proof of Theorem \ref{extend-U} with the estimates in Proposition \ref{1st-deri-2-1}.
\begin{proof}[Proof of Theorem \ref{extend-U}]
Up to a duplication, we may assume $\mu_1$ and $\mu_2$ admit the following representation
\begin{align*}
\mu_1=\frac1N\sum_{i=1}^N\delta_{x_i},\ \mu_2=\frac1N\sum_{i=1}^N\delta_{y_i}.
\end{align*}
Then
\begin{align*}
\mathcal W_2(\mu_1,\mu_2)=\min_{\sigma}\bigg(\frac1N\sum_{i=1}^N|x_i-y_{\sigma(i)}|^2\bigg)^{\frac12},
\end{align*}
where the minimum is taken over all permutation on $\{1,\ldots,N\}$. Up to a permutation, we may further assume that
\begin{align*}
\mathcal W_2(\mu_1,\mu_2)=\bigg(\frac1N\sum_{i=1}^N|x_i-y_i|^2\bigg)^{\frac12}.
\end{align*}
Denote by
\begin{align*}
\mathbf x_N:=(x_1,\ldots,x_N),\ \mathbf y_N:=(y_1,\ldots,y_N),
\end{align*}
then
\begin{align*}
&\quad|{\cal V}(t,\mu_1)-{\cal V}(t,\mu_2)|=|V_N(\mathbf x_N)-V_N(\mathbf y_N)|.
\end{align*}
Let $g:[0,1]\to\mathbb R$ be defined as
\begin{align*}
  g(\gamma):=V_N(t,\gamma\mathbf x_N+(1-\gamma)\mathbf y_N).
\end{align*}
 In view of the uniform estimates in Lemma \ref{1st-deri} and convergence in \eqref{approximation-2}, $V_N$ has weak derivative $\nabla_xV_N$ satisfying
\begin{align*}
    \big|\partial_{x_i}V_N(t,x)\big|&\leq\frac{\tilde C_2(C^L_{11}+C^U_{11})}N\bigg(1+|x_i|^2+\frac1N\sum_{j=1}^N|x_j|^2\bigg)^\frac12+\frac{\tilde C_2(C^L_{10}+C^U_{10})}N,\ (t,x)\in[0,T]\times\mathbb R^N.
\end{align*}
Therefore,
\begin{align}\label{differ-U-1}
{|V_N(t,\mathbf x_N)-V_N(t,\mathbf y_N)|\leq\int_0^1|g'(\gamma)|d\gamma}\leq\sum_{i=1}^N\int_0^1|\partial_{x_i}V_N(t,\gamma\mathbf x_N+(1-\gamma)\mathbf y_N)|\cdot|x_i-y_i|d\gamma,
\end{align}
where
\begin{align*}
 &\quad|\partial_{x_i}V_N(t,\gamma\mathbf x_N+(1-\gamma)\mathbf y_N)|\notag\\
&\leq\frac{\tilde C_2(C^L_{11}+C^U_{11})}N\bigg(1+|\gamma x_i+(1-\gamma)y_i|^2+\frac1N\sum_{j=1}^N|\gamma x_j+(1-\gamma)y_j|^2\bigg)^{\frac12}+\frac{\tilde C_2(C^L_{10}+C^U_{10})}N\notag\\
  &\leq\frac{\tilde C_2(C^L_{11}+C^U_{11})}N\bigg[1+|\gamma x_i+(1-\gamma)y_i|+\bigg(\frac1N\sum_{j=1}^N|\gamma x_j+(1-\gamma)y_j|^2\bigg)^{\frac12}\bigg]+\frac{\tilde C_2(C^L_{10}+C^U_{10})}N.
\end{align*}
Direct calculation gives
\begin{align*}
  &\quad|x_i-y_i|\int_0^1|\gamma x_i+(1-\gamma)y_i|d\gamma=\frac{|x_i-y_i|}{2(x_i-y_i)}(|x_i|x_i-|y_i|y_i)\leq\frac12|x_i-y_i|^2,\\
  &\quad|x_i-y_i|\int_0^1\bigg(\frac1N\sum_{j=1}^N|\gamma x_j+(1-\gamma)y_j|^2\bigg)^{\frac12}d\gamma\leq|x_i-y_i|\bigg(\frac1N\sum_{j=1}^N\int_0^1|\gamma x_j+(1-\gamma)y_j|^2d\gamma\bigg)^{\frac12}\notag\\
&\leq|x_i-y_i|\bigg(\frac1{2N}\sum_{j=1}^N|x_j|^2+\frac1{2N}\sum_{j=1}^N|y_j|^2\bigg)^{\frac12}.
\end{align*}
We notice that in the previous computations we have assumed that $x_{j}\neq y_{j}$, otherwise the inequalities are trivially true.
Plug the above inequalities into \eqref{differ-U-1},
\begin{align*}
|{\cal V}(t,\mu_1)-{\cal V}(t,\mu_2)|&\leq\frac{\tilde C_2(C^L_{11}+C^U_{11}+C^L_{10}+C^U_{10})}N\sum_{i=1}^N|x_i-y_i|+\frac{\tilde C_2(C^L_{11}+C^U_{11})}{2N}\sum_{i=1}^N|x_i-y_i|^2\notag\\
&\quad+\frac{\tilde C_2(C^L_{11}+C^U_{11})}N\bigg(\sum_{i=1}^N|x_i-y_i|\bigg)\bigg(\frac1{2N}\sum_{j=1}^N|x_j|^2+\frac1{2N}\sum_{j=1}^N|y_j|^2\bigg)^{\frac12}\notag\\
&\leq\tilde C_2(C^L_{11}+C^U_{11}+C^L_{10}+C^U_{10})\mathcal W_2(\mu_1,\mu_2)+\frac{\tilde C_2(C^L_{11}+C^U_{11})}2\mathcal W^2_2(\mu_1,\mu_2)\notag\\
&\quad+\frac{\tilde C_2(C^L_{11}+C^U_{11})}2\mathcal W_2(\mu_1,\mu_2)\cdot\bigg(\int_{\mathbb R}y^2\mu_1(dy)+\int_{\mathbb R}y^2\mu_2(dy)\bigg).
\end{align*}
Note that, as an implication of the law of large numbers, for any $\mu\in\mathcal P_2(\mathbb R)$, there exists a sequence of empirical measures $(\mu^N)_{N\geq1}$ such that $\lim_{N\to+\infty}\mathcal W_2(\mu^N,\mu)=0.$ For arbitrary $(\mu^N)_{N\geq1}$ that converges to $\mu$ in $\mathcal P_2(\mathbb R)$, the above inequality then yields that $\big(\mathcal V(t,\mu^N)\big)_{N\geq1}$ is a Cauchy sequence. Since the mentioned $(\mu^N)_{N\geq1}$ could be arbitrary, the following extension of $\mathcal V(t,\cdot)$ at $\mu$ is well-defined:
\begin{align*}
 \mathcal V(t,\mu):=\lim_{N\to+\infty}\mathcal V(t,\mu^N),
\end{align*}
and the uniqueness of extension is implied by the continuity of $\mathcal V$.
\end{proof}

In view of Theorem \ref{extend-U} and Corollary \ref{VN-con-rate}, the definition domain of ${\cal V}$ can be extended to $\mathcal P_2(\mathbb R)$. Moreover, they reveal the convergence (at a specific rate) of $V_N(t,x_1,\ldots,x_N)$ to $\mathcal V(t,\mu)$ whenever $\frac1N\sum_{i=1}^N\delta_{x_i}$ converges in $\mathcal P_2(\mathbb R)$. It is thus natural to further consider the convergence of feedback control function $\theta^*(t,x_1,\ldots,x_N)$, as well as the corresponding convergence rate.

Consider the empirical measure
\begin{align*}
\mu^N=\frac1N\sum_{i=1}^N\delta_{x_i},
\end{align*}
and introduce the notation
\begin{align}\label{deri-def}
D_{\mu}^{(N)}{\cal V}(t,\mu^N,x_i):=N\partial_{x_i}V_N(t,x_1,\ldots,x_N),\ i=1,\ldots,N.
\end{align}
In view of the symmetric property of $V_N$, $D_{\mu}^{(N)}{\cal V}(t,\mu^N,x_i)$ above is well-defined. Next we show that $\mathcal V$ is differentiable in the measure variable at $(t,\mu^N)$ and
\begin{align}\label{equivalent-deri}
 \partial_\mu {\cal V}(t,\mu^N,x_i)=D_{\mu}^{(N)}{\cal V}(t,\mu^N,x_i).
\end{align}

\begin{lemma}\label{well-define-1}
   Let $V_N$ be the value function in \eqref{def-V-N-O}. For samples $x_1,\ldots,x_N\in\mathbb R$ and $y_1,\ldots,y_M\in\mathbb R$, suppose that
\begin{align*}
\mu^N:=\frac1N\sum_{i=1}^N\delta_{x_i}=\frac1M\sum_{i=1}^M\delta_{y_i}=:\nu^M,
\end{align*}
then for $t\in[0,T]$, $T>0$ and any bounded continuous function $\varphi:\ \mathbb R\ \to\ \mathbb R$, we have
\begin{align*}%\label{value-equal}
\int_{\mathbb R}\varphi(y)D^{(N)}_\mu {\cal V}(t,\mu^N,y)\mu^N(dy)=\int_{\mathbb R}\varphi(y)D^{(M)}_\mu {\cal V}(t,\nu^M,y)\nu^M(dy),
\end{align*}
as well as
\begin{align}\label{equivalent-deri-1}
\lim_{\epsilon\to0+}\epsilon^{-1}\bigg({\cal V}\big(t,(I+\epsilon\varphi)\sharp\mu^N\big)-{\cal V}(t,\mu^N)\bigg)=\int_{\mathbb R}\varphi(y)D^{(N)}_\mu {\cal V}(t,\mu^N,y)\mu^N(dy).
\end{align}
As a result, \eqref{equivalent-deri} is valid.

\end{lemma}
\begin{proof}
  Similarly to the comments before \eqref{duplication-equality}, it suffices to show that
\begin{align}\label{fdbk-eq}
  \sum_{i=1}^N\varphi(x_i)\partial_{x_i}V_N(t,x_1,\ldots,x_N)=\sum_{i=1}^{MN}\varphi(y_i)\partial_{y_i}V_{NM}(t,y_1,\ldots,y_{NM}),
\end{align}
where $(y_1,\ldots,y_{NM})=(x_1\mathbf1^\top_M,\ldots,x_N\mathbf1^\top_M)$, in other words, $y_{(i-1)+k}=x_i$, $1\leq i\leq N$.
According to Lemma \ref{well-define},
\begin{align*}
  V_N(t,x_1,\ldots,x_N)=V_{NM}(t,x_1\mathbf1^\top_M,\ldots,x_N\mathbf1^\top_M)=V_{NM}(t,y_1,\ldots,y_{NM}).
\end{align*}
Take the { weak} derivative with respect to $x_i$ and obtain
\begin{align*}
  \partial_{x_i}V_N(t,x_1,\ldots,x_N)=\sum_{k=1}^M\partial_{y_{(i-1)+k}}V_{NM}(t,y_1,\ldots,y_{NM}).
\end{align*}
Using the equality above and noticing that $y_{(i-1)+k}=x_i$, $1\leq i\leq N$, we can show that \eqref{fdbk-eq} is true. 

To show \eqref{equivalent-deri-1}, we may plug in \eqref{def-U} and \eqref{deri-def}. Then \eqref{equivalent-deri-1} and \eqref{equivalent-deri} follows.
\end{proof}
According to Lemma \ref{well-define-1}, we may present the optimal feedback function $\theta^*_N$ in such a way that
\begin{align}\label{measure-theta}
\theta^*(t,\mu^N)&:=\theta^*_N\big(t,x,\nabla_xV_N(t,x)\big)\notag\\
  &=\argmin_{\theta\in\Theta}\bigg\{\frac\lambda2|\theta|^2+\int_{\mathbb R}f(t,\theta,y,\mu^N)\partial_\mu {\cal V}(t,\mu^N,y)\mu^N(dy)\bigg\}.
\end{align}
Similar to Theorem \ref{extend-U}, we can show the Lipschitz continuity of $\theta^*(t,\mu^N)$ in \eqref{measure-theta}, which implies the convergence rate of the optimal feedback function.

\begin{theorem}\label{lip-lift-theta} 
Suppose that the assumptions of Lemma \ref{1st-deri-2-0-1} take place.
Let $\mu_1$, $\mu_2$ be two empirical measures and $\theta^*(t,\mu)$ be defined as in \eqref{measure-theta}. Then for $T<\tilde c$, where $\tilde c$ is from Proposition \ref{dY/dx-1-0},
\begin{align}\label{theta-convergence}
|\theta^*(t,\mu_1)-\theta^*(t,\mu_2)|\leq\tilde C_8\mathcal W_1(\mu_1,\mu_2).
\end{align}
Here
\begin{align*}
{\tilde C_8:=(\lambda-\lambda_0)^{-1}C^Q+(\lambda-\lambda_0)^{-1}\|f_\theta\|_\infty\tilde C_6.}
\end{align*}
As a result, $\theta^*(t,\cdot)$ can be uniquely extended to a Lipschitz continuous mapping on $(\mathcal P_2(\mathbb R),\mathcal W_{1})$.
\end{theorem}
\begin{proof}
Up to a duplication, we may assume that $\mu_1$ and $\mu_2$ have the same sample size. Denote by
\begin{align*}
\mu_1=\frac1N\sum_{i=1}^N\delta_{x_i},\ \mu_1=\frac1N\sum_{i=1}^N\delta_{y_i}.
\end{align*}
    It is easy to see from \eqref{measure-theta} that $\theta^*(t,x_1,\ldots,x_N)$ on the right hand side remains unchanged after a permutation of the input $\{x_1,\ldots,x_N\}$. Hence up to a permutation, we may assume that
\begin{align*}
\mathcal W_1(\mu_1,\mu_2)=\frac1N\sum_{i=1}^N|x_i-y_i|.
\end{align*}
Define
\begin{align*}
g(\gamma):=\theta^*(t,\gamma\mathbf x_N+(1-\gamma)\mathbf y_N),\ \gamma\in[0,1].
\end{align*}
According to \eqref{theta-weak-deri} and \eqref{VN-scaling-2},
\begin{align*}
&\quad|\theta^*(t,\mu_1)-\theta^*(t,\mu_2)|=|\theta^*(t,x_1,\ldots,x_N)-\theta^*(t,y_1,\ldots,y_N)|\notag\\
&=|g(1)-g(0)|\leq\int_0^1|g'(\gamma)|d\gamma\notag\\
      &\leq\sum_{l=1}^N|x_l-y_l|\int_0^1\bigg|\frac\partial{\partial p_l}\theta^*+\sum_{k=1}^N\frac\partial{\partial q_k}\theta^*\cdot\partial^2_{kl}V_N\bigg|(t,\gamma\mathbf x_N+(1-\gamma)\mathbf y_N)d\gamma\notag\\
&\leq\frac{{(\lambda-\lambda_0)^{-1}}C^Q+{(\lambda-\lambda_0)^{-1}}\|f_\theta\|_\infty\tilde C_6}N\sum_{l=1}^N|x_l-y_l|\notag\\
&=({(\lambda-\lambda_0)^{-1}}C^Q+{(\lambda-\lambda_0)^{-1}}\|f_\theta\|_\infty\tilde C_6)\mathcal W_1(\mu_1,\mu_2).
\end{align*}
\end{proof}
Now we have the convergence rate of feedback function as the sample size grows to infinity.
\begin{corollary}\label{theta-con-rate}
Suppose that the assumptions of Theorem \ref{lip-lift-theta} take place and suppose that $\mu^N:=\frac1N\sum_{i=1}^N\delta_{x_i}\to\mu$ in $(\mathcal P_2(\mathbb R),\mathcal W_{2}),$ as $N\to+\infty$. Then for $T<\tilde c$ where $\tilde c$ is from Proposition \ref{dY/dx-1-0},
\begin{align*}
\lim_{N\to+\infty}\theta^*_N\big(t,x,\nabla_xV_N(t,x)\big)=\theta^*(t,\mu),
\end{align*}
at a rate
\begin{align*}
|\theta^*_N\big(t,x,\nabla_xV_N(t,x)\big)-\theta^*(t,\mu)|\leq\tilde C_8\mathcal W_1(\mu^N,\mu).
\end{align*}
\end{corollary}
\begin{proof}
  This is directly from \eqref{measure-theta} and Theorem \ref{lip-lift-theta}.
\end{proof}
Another consequence of Theorem \ref{lip-lift-theta} is the pathwise convergence at certain rates.
\begin{proposition}\label{cauchy-sequence-0-0}
 Let $X^*_N = (X^{1,*}_N(s),\dots,X^{N,*}_N(s))_{s\in[0,T]}$, $N\geq1$ be the optimal path of Problem \ref{problem}, with $t=0$ and initial values $x^{(1)}_N,x^{(2)}_N,\ldots,x^{(N)}_N$. Suppose that the assumptions of Theorem \ref{lip-lift-theta} take place and suppose that
  $\frac1N\sum_{i=1}^N\delta_{x^{(i)}_N}\to\mu\ \text{in}\ (\mathcal P_2(\mathbb R),\mathcal W_{2})$, as $N\to+\infty$. Then for $T<\tilde c$ where $\tilde c$ is from Proposition \ref{dY/dx-1-0}, there exists an adapted limit process $(\theta^*,\mu^*)$, where $\theta^*(s)\in\Theta$ and $\mu^*(s)\in\mathcal P_1(\mathbb R)$, $0\leq s\leq T$, such that $\mu^*(0)=\mu$ and 
 \begin{align}
 &\max_{s\in[t,T]}\mathcal W_1(\mu^*_N(s),\mu^*(s))\leq\hat C_8\mathcal W_1\big(\mu^*_N(0),\mu^*(0)\big),\notag\\
 \label{cauchy-sequence-0}&\max_{s\in[0,T]}|\theta^*_N\big(s,X^*_N(s),\nabla_xV_N(t,X^*_N(s))\big)-\theta^*(s)|\leq\hat C_8\mathcal W_1(\mu^*_N(0),\mu^*(0)).
\end{align}
Here $\mu^*_N(t):=\frac1N\sum_{i=1}^N\delta_{X^{i,*}_N(t)}$ and $\hat C_{8}>0$ is a constant independent of $N$.
\end{proposition}
\begin{proof}
 In order to show the first inequality in \eqref{cauchy-sequence-0}, it suffices to first show that, for the sample paths $X^*_N$ and $X^*_M$, which corresponds to sample number $N$ and $M$ respctively, it holds that
\begin{align}\label{cauchy-sequence}
 \max_{s\in[0,T]}\mathcal W_1(\mu^*_N(s),\mu^*_M(s))\leq\hat C_8\mathcal W_1\big(\mu^*_N(0),\mu^*_M(0)\big),
\end{align}
and then pass $M$ to infinity. Here we have used the assumption that $\mu^*_N(0)=\frac1N\sum_{i=1}^N\delta_{x^{(i)}_N}$, $N\geq1$, is a Cauchy sequence in $(\mathcal P_2(\mathbb R),\mathcal W_{2})$.

According to Lemma \ref{well-define-1} and \eqref{measure-theta}, define
\begin{align*}
 f^*(s,x,\mu):=f(s,\theta^*(s,\mu),x,\mu),\quad(s,x,\mu)\in[0,T]\times\mathbb R\times\mathcal P(\mathbb R),
\end{align*}
then the optimal path $X^*_N$ satisfies
\begin{align}\label{cauchy-sequence-1}
 dX^{i,*}_N(s)=f^*(s,X^{i,*}_N(s),\mu^*_N(s))ds+\sigma dW(s),\quad1\leq i\leq N.
 \end{align}
 Moreover, according to Theorem \ref{lip-lift-theta}, for $x_1,\ldots,x_N,\tilde x_1,\ldots,\tilde x_N\in\mathbb R$, it holds that
 \begin{align}\label{cauchy-sequence-2}
  \bigg|f^*\bigg(s,x_i,\frac1N\sum_{ij=1}^N\delta_{x_j}\bigg)-f^*\bigg(s,\tilde x_i,\frac1N\sum_{j=1}^N\delta_{\tilde x_j}\bigg)\bigg|\leq\hat C_8|x_i-\tilde x_i|+\frac{\hat C_8}N\sum_{j=1}^N|x_j-\tilde x_j|,\quad1\leq i\leq N.
 \end{align}
Given the dynamics in \eqref{cauchy-sequence-1}, we may proceed in a similar fashion to the proof in Lemma \ref{well-define} and show that 
\begin{align*}
 \mu^*_N(s)=\tilde\mu^*_{NM}(s),
\end{align*}
where
\begin{align*}
 \tilde\mu^*_{NM}(s):=\frac1{NM}\sum_{i=1}^{NM}\delta_{\tilde X^{i,*}_{NM}(s)},\quad\tilde\mu^*_{NM}(0)=\frac1{NM}\sum_{i=1}^NM\delta_{x^{(i)}_N},
\end{align*}
with
\begin{align*}
 d\tilde X^{i,*}_{NM}(s)=f^*(t,\tilde X^{i,*}_{NM}(s),\tilde\mu^*_{NM}(s))ds+\sigma dW(s),\quad1\leq i\leq NM.
 \end{align*}
 Therefore, up to a duplication, showing \eqref{cauchy-sequence} is equivalent to showing that
 \begin{align}\label{cauchy-sequence-3}
  \max_{s\in[0,T]}\mathcal W_1(\mu^*_N(s),\tilde\mu^*_N(s))\leq\hat C_8\mathcal W_1\big(\mu^*_N(0),\tilde\mu^*_N(0)\big),\quad N\geq1,
 \end{align}
where
\begin{align*}
 \tilde\mu^*_N(s):=\frac1N\sum_{i=1}^N\delta_{\tilde X^{i,*}_N(s)},\quad\tilde\mu^*_N(0)=\frac1N\sum_{i=1}^N\delta_{\tilde x^{(i)}_N},
\end{align*}
with
\begin{align*}
 d\tilde X^{i,*}_N(s)=f^*(t,\tilde X^{i,*}_N(s),\tilde\mu^*_N(s))ds+\sigma dW(s),\quad1\leq i\leq N.
 \end{align*}
 Here $\tilde x^{(1)}_N,\ldots,\tilde x^{(N)}_N$ are $N$ arbitrary numbers from $\mathbb R$.
 But in view of \eqref{cauchy-sequence-2}, subtracting the above and \eqref{cauchy-sequence-1} as well as the standard Gr\"onwall's inequality then yields \eqref{cauchy-sequence-3}, which further implies \eqref{cauchy-sequence}.
 
 Having obtained \eqref{cauchy-sequence}, we may use Theorem \ref{lip-lift-theta} to further show that for any $t\in[0,T]$,
 \begin{align}\label{theta-convergence-1}
  &\quad|\theta^*_N\big(s,X_N(s),\nabla_xV_N(s,X_N(s))\big)-\theta^*_M\big(s,X_M(s),\nabla_xV_M(s,X_M(s))\big)|\notag\\
  &=|\theta^*(s,\mu^*_N(s))-\theta^*(s,\mu^*_M(s))|\leq\hat C_8\mathcal W_1(\mu^*_N(0),\mu_M(0)).
 \end{align}
Passing $M$ to infinity in the above yields the second inequality in \eqref{cauchy-sequence-0}.
\end{proof}
The path $\theta^*_N\big(s,X_N(s),\nabla_xV_N(s,X_N(s))\big)$, $s\in[0,T]$ in Proposition \ref{cauchy-sequence-0-0} actually corresponds to the optimal parameters obtained via the neural SDE with $N$ samples. Hence we may interpret Proposition \ref{cauchy-sequence-0-0} in such a way that the aforementioned parameters converge, at a certain rate, as long as the empirical distributions of inputs converge as $N$ tends to infinity.

In addition to the above convergence for short time, we can also obtain the global convergence under assumptions on convexity. We first do some preparation in Lemma \ref{Holder-cont.} then present the main results in Theorem \ref{lip-lift-theta-1}. Denote by
\begin{align*}
  \mu^N:=\frac1N\sum_{i=1}^N\delta_{x_i},\quad\nu^N:=\frac1N\sum_{i=1}^N\delta_{y_i}.
\end{align*}
\begin{lemma}\label{Holder-cont.}
  Suppose that the assumptions of Proposition \ref{dY/dx-1} take place.
  Then  \begin{align}\label{2nd-error-0}
  \frac1N\sum_{i=1}^N\big|\partial_{\mu}{\cal V}(t,\mu^N,x_i)-\partial_{\mu}{\cal V}(t,\nu^N,y_i)\big|^2\leq\frac{\tilde C_6}N\sum_{i=1}^N|x_i-y_i|^2.
  \end{align}
\end{lemma}
\begin{proof}
  In view of Theorem \ref{veri-1}, for $(t,x)\in[0,T]\times\mathbb R^N$ we have
  \begin{align*}
    \nabla^2_xV_N(t,x)=\bigg(\nabla^2_xV_N(t,x)^{\frac12}\bigg)^\top\nabla^2_xV_N(t,x)^{\frac12},
  \end{align*}
  for some matrix $\nabla^2_xV_N(t,x)^{\frac12}\in\mathbb R^{N\times N}$ such that for any $\alpha\in\mathbb R^N$ with $|\alpha|=1$,
  \begin{align*}
    |\nabla^2_xV_N(t,x)^{\frac12}\alpha|\leq\sqrt{\frac{\tilde C_6}N}|\alpha|.
  \end{align*}
  Therefore for any $\alpha,x,y\in\mathbb R^N$,
  \begin{align*}
    &\quad\langle\alpha,\nabla_xV_N(t,x)-\nabla_xV_N(t,y)\rangle\notag\\
    &=\int_0^1\bigg\langle\nabla^2_xV_N(t,y+s(x-y))^{\frac12}\alpha,\nabla^2_xV_N(t,y+s(x-y))^\frac12(x-y)\bigg\rangle ds\notag\\
    &\leq\frac{\tilde C_6}N|\alpha|\cdot|x-y|.
  \end{align*}
  The inequality above implies that
  \begin{align*}
    |\nabla_xV_N(t,x)-\nabla_xV_N(t,y)|\leq\frac{\tilde C_6}N|x-y|,
  \end{align*}
  which is \eqref{2nd-error-0} according to \eqref{deri-def}.
\end{proof}
Lemma \ref{Holder-cont.} tells that we can extend the domain of $\partial_\mu {\cal V}(t,\nu,\cdot)$ from the set of all empirical measures to $\nu\in\mathcal P_2(\mathbb R)$ in some weak sense, which is formalized as follows.
\begin{corollary}\label{extend-derivative}
  For each $t\in[0,T]$, there exists a Lipschitz continuous mapping $\Phi_t$ that maps empirical measures on $\mathbb R$ to $\mathcal P_2(\mathbb R)$ in such a way that for any $\mu=\frac1N\sum_{i=1}^N\delta_{x_i}$,
  \begin{align*}
    \Phi_t(\mu)=\frac1N\sum_{i=1}^N\delta_{\partial_\mu {\cal V}(t,\mu,x_i)}.
  \end{align*}
Therefore, $\Phi_t$ can be uniquely extended to a Lipschitz continuous map $\Phi_t:\mathcal P_2(\mathbb R)\to\mathcal P_2(\mathbb R)$. Moreover, with an abuse of notation, there exists a unique nonnegative function $\Phi_t(\mu,\cdot)\in L^1(\mu;\mathbb R)$ such that for any Borel subset $A\subseteq\mathbb R$:
\begin{align*}
 \Phi_t(\mu)(A)=\int_A\Phi_t(\mu,y)\mu(dy).
\end{align*}
Lastly, for any bounded continuous function $\varphi:\ \mathbb R\ \to\ \mathbb R$, we have
\begin{align}\label{equivalent-deri-1}
\lim_{\epsilon\to0+}\epsilon^{-1}\bigg({\cal V}\big(t,(I+\epsilon\varphi)\sharp\mu\big)-{\cal V}(t,\mu)\bigg)=\int_{\mathbb R}\varphi(y)\Phi_t(\mu)(dy)=\int_{\mathbb R}\varphi(y)\Phi_t(\mu,y)\mu(dy).
\end{align}
\end{corollary}
\begin{proof}
  Consider the following empirical measures
\begin{align*}
  \mu^N=\frac1N\sum_{i=1}^N\delta_{x_i},\quad\nu^N=\frac1N\sum_{i=1}^N\delta_{y_i}.
\end{align*}In view of the symmetric property, there is no loss of generality {in} assuming
\begin{align*}
  \mathcal W_2(\mu^N,\nu^N)=\bigg(\frac1N\sum_{i=1}^N|x_i-y_i|^2\bigg)^\frac12.
\end{align*}
Then we have from Lemma \ref{Holder-cont.} that
\begin{align*}
  &\mathcal W_2\big(\Phi_t(\mu^N),\Phi_t(\nu^N)\big)\leq\sqrt{\tilde C_6}\mathcal W_2(\mu^N,\nu^N).
\end{align*}
Hence $\Phi_t$ is Lipschitz continuous. 

Consider now $A\cap{\rm supp}(\mu)=\emptyset$ as well as i.i.d random variables $(\xi_i)_{i\geq1}$ with law $\mu$. Let $\hat\mu^N=\frac1N\sum_{i=1}^N\delta_{\xi_i}$. It holds  that $\xi_i\notin A,\ i\geq1$ a.s., hence
\begin{align*}
\mathbb E[\Phi_t(\hat\mu^N)(A)]=\mathbb E\bigg[\int_A\partial_\mu {\cal V}(t,\hat\mu^N,y)\hat\mu^N(dy)\bigg]=0.
\end{align*}
In view of the strong law of large numbers and that $\Phi_t$ is Lipschitz continuous, $\Phi_t(\hat\mu^N)(A)\to\Phi_t(\mu)(A)$ a.s.. Bounded convergence theorem yields $\Phi_t(\mu)(A)=0.$ Notice that $A$ is arbitrary such that $A\cap{\rm supp}(\mu)=\emptyset$, we have that the Radon--Nikodym derivative $\Phi_t(\mu,\cdot)\in L^1(\mu;\mathbb R)$. It remains to show \eqref{equivalent-deri-1}. In view of Lemma \ref{well-define-1},
\begin{align*}
 {\cal V}\big(t,(I+\epsilon\varphi)\sharp\mu^N\big)-{\cal V}(t,\mu^N)=\int_0^\varepsilon \int_{\mathbb R}\varphi(y)\Phi_t((I+s\varphi)\sharp\mu^N)(dy)ds.
\end{align*}
Since $\Phi_t(\cdot)$ is a Lipschitz mapping, we may first send $N$ to infinity in the above then take derivative with respect to $\varepsilon$ at  $\varepsilon=0$ and obtain $\eqref{equivalent-deri-1}$.
\end{proof}

Before going into the next convergence result, we would like to discuss Corollary \ref{extend-derivative} from the perspective of PDE. In view of \eqref{equivalent-deri-1}, it is natural to define the derivative $\partial_\mu {\cal V}$ in the following way:
\begin{align*}
 \partial_\mu {\cal V}(t,\mu,y):=\Phi(t,\mu,y),\quad(t,\mu,y)\in[0,T]\times\mathbb P_2(\mathbb R)\times\mathbb R.
\end{align*}
It is easy to see that $\partial_\mu {\cal V}(t,\mu^N,x_1)=NV_N(t,x_1,\ldots,x_N)$. When $\sigma=0$, in view of the projection property in Theorem \ref{extend-U}, it is straightforward that \eqref{intro:2} is true for $\mu=\mu^N$. Thanks to Corollary \ref{extend-derivative} and the above notion of $\partial_\mu {\cal V}$, we may now let $\mu^N\to\mu$ and find \eqref{intro:2} true for $\mu\in\mathcal P_2(\mathbb R)$.

As a result of the distributional difference estimate in Lemma \ref{Holder-cont.}, we prove the Lipschitz continuity of $\theta^*(t,\mu^N)$ for long time $T>0$, which has been reported in Theorem \ref{lip-lift-theta-1}.
%\begin{proof}[Proof of Theorem \ref{theta-convergence-1}]
\begin{proof}[{ Proof of Theorem \ref{lip-lift-theta-1}}]
  Let $\theta^*$, $\hat\theta^*$ denote the optimal feedback function corresponding to $\mu^N$ and $\nu^N$. According to the first order condition,
  \begin{align*}
    &\lambda\theta^*+\frac1N\sum_{i=1}^Nf_\theta(t,\theta^*,x_i,\mu^N)\partial_\mu {\cal V}(t,\mu^N,x_i)=0,\\
    &\lambda\hat\theta^*+\frac1N\sum_{i=1}^Nf_\theta(t,\hat\theta^*,y_i,\nu^N)\partial_\mu{\cal V}(t,\nu^N,y_i)=0.
  \end{align*}
  Subtracting the above and utilizing \eqref{hypo-regu-2}, \eqref{hypo-regu-3}, \eqref{VN-scaling-3} as well as \eqref{tilde-C-7}, we have
  \begin{align*}
    {(\lambda-\lambda_0)}|\theta^*-\hat\theta^*|&\leq\frac{\|f_{\theta x}\|_\infty C^Q}N\sum_{i=1}^N|x_i-y_i|+\frac{\|f_{\mu\theta}\|_\infty\tilde C_{71}}N\sum_{i=1}^N|x_i-y_i|\notag\\
    &\quad+\frac{\|f_\theta\|_\infty}N\sum_{i=1}^N|\partial_\mu{\cal V}(t,\mu^N,x_i)-\partial_\mu{\cal V}(t,\hat\mu^N,y_i)|.
  \end{align*}
  In view of Lemma \ref{Holder-cont.}, by choosing appropriate $(x_i,y_i)$, $i=1,\ldots,N$, we can deduce \eqref{theta-convergence-1} from the above.
\end{proof}
Parallel to Corollary \ref{theta-con-rate} and Proposition \ref{cauchy-sequence-0-0}, we estimate the convergence rate of feedback function and the optimal parameters for long time $T>0$, see Corollary \ref{theta-con-rate-1} and Proposition \ref{cauchy-sequence-1-0}.

{ \section{Example: training with monotone loss functions}\label{adaptation}

In this subsection we show that, for certain $L,\ U$, and $f$, how \noindent{\bf Hypothesis (R1)} can be verified with any $\lambda, T>0$. To ease technical details, for $N=1,2,\ldots$ we consider
\begin{align}\label{mono-0}
\left\{\begin{aligned}
    &dX^{\theta,i}_N(t) = f\big(t,\theta(t),X^{\theta,i}_N(t)\big)dt + \sigma dW^0(t),\ t\in[0,T],\\
&X^{\theta,i}_N(0)=x_i,\ i=1,\ldots,N,
\end{aligned}\right.
\end{align}
together with the objective functional $J_N$ in \eqref{obj-func-2}. Here all components $(X^{\theta,i}_N)_{1\leq i\leq N}$ are coupled via $\big(\theta(t)\big)_{t\in[0,T]}$. In this subsection we further assume
\begin{itemize}
 \item The drift function $f$ has bounded derivatives and satisfies
 \begin{align}\label{mono-1}
  &\partial^2_{\theta\theta}f(t,\theta,y)>0,\quad\partial^2_{yy}f(t,\theta,y)\partial^2_{\theta\theta}f(t,\theta,y)>|\partial^2_{\theta y}f(t,\theta,y)|^2,\quad(t,\theta,y)\in[0,T]\times\mathbb R\times\mathbb R.
 \end{align}
\item The loss functions $L,\ U$ are both convex, Lipschitz continuous and satisfy
\begin{align}\label{mono-2}
 L'(y)\geq0,\ U'(y)\geq0,\quad y\in\mathbb R.
\end{align}
\end{itemize}
In the current situation,
\begin{align*}
 H_N(t,x,p)=\inf_{\theta\in\Theta}\bigg\{\sum_{i=1}^Nf(t,\theta,x_i)p_i+\frac\lambda2|\theta|^2+\frac1N\sum_{i=1}^NL(x_i)\bigg\},\quad(t,x,p)\in[0,T]\times\mathbb R^N\times\mathbb R^N.
\end{align*}
We first show the following:
\begin{lemma}\label{example-1}
 Suppose \eqref{mono-1} and \eqref{mono-2}. Then for $p=(p_1,\ldots,p_N)\in\mathbb R^N$ satisfying $p_i\geq0$,
 \begin{align*}
  D^2_{xx}\tilde H_N(t,x,p)\geq0,\quad(t,x)\in[0,T]\times\mathbb R^N.
 \end{align*}
\end{lemma}
\begin{proof}
 For $\xi=(\xi_1,\ldots,\xi_N)\in\mathbb R^N$, direct calculation yields
 \begin{align*}
  &\xi^\top D^2_{xx}\tilde H_N(t,x,p)\xi\\
  &=-\bigg(\sum_{i=1}^N\partial^2_{\theta\theta}f(t,x_i)p_i+\lambda\bigg)^{-1}\bigg(\sum_{i,j=1}^N\xi_i\partial^2_{\theta x}f(t,x_i)p_i\bigg)^2+\sum_{i=1}^N|\xi_i|^2\partial^2_{xx}f(t,x_i)p_i+\sum_{i=1}^N|\xi_i|^2L''(x_i).
 \end{align*}
According to H\"older's inequality and \eqref{mono-1}, we obtain the desired result with the following:
\begin{align*}
 \bigg(\sum_{i,j=1}^N\xi_i\partial^2_{\theta x}f(t,x_i)p_i\bigg)^2&\leq\bigg(\sum_{i=1}^N\partial^2_{\theta\theta}f(t,x_i)p_i\bigg)\bigg(\sum_{i=1}^N|\xi_i|^2\frac{|\partial^2_{\theta x}f(t,x_i)|^2}{\partial^2_{\theta\theta}f(t,x_i)}p_i\bigg)\\
 &\leq\bigg(\sum_{i=1}^N\partial^2_{\theta\theta}f(t,x_i)p_i+\lambda\bigg)\bigg(\sum_{i=1}^N|\xi_i|^2\partial^2_{xx}f(t,x_i)p_i\bigg),
\end{align*}
where according to our assumptions $\sum_{i=1}^N\partial^2_{\theta\theta}f(t,x_i)p_i+\lambda>0$.
\end{proof}
We also need the following monotonicity:
\begin{lemma}\label{example-2}
 Let $V^\varepsilon_N$ be the solution to \eqref{glb-3} with parameters from \eqref{mono-0}$\sim$\eqref{mono-2}, then for $(t,x)\in[0,T]\times\mathbb R^N,\ 1\leq i\leq N$, it holds that $\partial_{x_i}V^\varepsilon_N(t,x)\geq0.$
\end{lemma}
\begin{proof}
 In view of Proposition \ref{1st-deri-2-1}, $V^\varepsilon_N$ is the classical solution to HJB equation \eqref{glb-3}. We may consider the corresponding stochastic optimal control problem and use the stochastic maximum principle to get the optimal path and the adjoint process. Then for $(t,x)\in[0,T]\times\mathbb R^N$ we obtain the following:
 \begin{align*}
 \left\{\begin{aligned}
 dX^{i,*}_N(s)&=f\big(s,X^{i,*}_N(s),\mu^*_N(s)\big)ds+\sigma dW(s),\\
 dY^{*,i}_N(s)&=-\bigg[\partial_xf\big(s,\theta(s),X^{\theta,i}_N(s)\big)Y^{*,i}_N(s)+\frac1NL'\big(X^{i,*}_N(s)\big)\bigg]ds+\sum_{j=0}^NZ^{ij}_N(s)dW^j(s),\ s\in[t,T],\\
X^*_N(t)&=x,\quad Y^{*,i}_N(T)=\frac1NU'\big(X^{*,i}_N(T)\big),\quad1\leq i\leq N.
 \end{aligned}\right.
 \end{align*}
 Define $\tilde Y^{*,i}_N(s)=\exp\bigg\{\int_t^s\partial_xf\big(s,\theta(s),X^{\theta,i}_N(s)\big)ds\bigg\}Y^{*,i}_N(s)$. Note that $L',U'\geq0$, then
 \begin{align*}
  &\quad\partial_{x_i}V^\varepsilon_N(t,x)=\tilde Y^{*,i}_N(t)\\
  &=\mathbb E\bigg[\frac1NU'\big(X^{*,i}_N(T)\big)+\int_t^T\frac1N\exp\bigg\{\int_t^s\partial_xf\big(s,\theta(s),X^{\theta,i}_N(s)\big)ds\bigg\}L'\big(X^{i,*}_N(s)\big)ds\bigg]\geq0.
 \end{align*}
\end{proof}
With the above preparation, we can now show the quantitative convergence results on the value functions and training outcomes of the current neural SDE where the coefficient of the regularizer $\lambda>0$ can be arbitrarily small.
\begin{theorem}
 The the conditions in Proposition \ref{dY/dx-1} are fulfilled with arbitrary $\lambda,\ T>0$. As a consequence, the convergence results in Theorem \ref{extend-U}, Theorem \ref{lip-lift-theta-1} and Proposition \ref{cauchy-sequence-1-0} are valid.
\end{theorem}
\begin{proof}
 Fix arbitrary $\lambda>0$. In view of Lemma \ref{1st-deri} and Lemma \ref{example-2}, we can see that $\big(x,\nabla_xV^\varepsilon_N(t,x)\big)\in\mathcal A_{N,0,\tilde C_2}$ for all $\varepsilon>0,\ (t,x)\in[0,T]\times\mathbb R^N$. Moreover, since $L,\ U$ are Lipischitz, $C^L_{11}+C^U_{11}=0$. Note that $\partial^2_{\theta\theta}f>0$ and is bounded, hence for any constant $\lambda_0\in(0,\lambda)$, \eqref{hypo-regu-2} is true for $(t,\theta,x,p)\in[0,T]\times\Theta\times\mathcal{A}_{N,\delta,\tilde C_2}$ with sufficientlyl small $\delta>0$. Since $f$ has bounded derivatives, \eqref{hypo-regu-3} holds trivially with $(\Gamma_1,\Gamma_2)=(\delta,\tilde C_2)$. Hence \noindent{\bf Hypothesis (R)} holds with $(\Gamma_1,\Gamma_2)=(\delta,\tilde C_2)$. Combining Lemma \ref{example-1} and Lemma \ref{example-2}, we obtain $\nabla^2_{xx}\tilde H_N\big(t,x,\nabla_xV^\varepsilon_N(t,x)\big)\geq0$ for $N\geq1,\ \varepsilon>0$, $(t,x)\in[0,T]\times\mathbb R^N$. Therefore \noindent{\bf Hypothesis (R1)} holds with $(\Gamma_1,\Gamma_2)=(\delta,\tilde C_2)$ and the proof is complete.
\end{proof}}

\appendix
\section{Some technical results}\label{sec-App}
Recall \eqref{definition-MNC}. For a matrix valued function
$\tilde A$, where $\tilde A(t)\in\mathbb R^{N\times N}$, $t\in[0,T]$, we define the matrix $|\tilde A|\in\mathbb R^{N\times N}$ in such a way that
\begin{align}
|\tilde A|_{ij}=\max_{t\in[0,T]}|\tilde A_{ij}(t)|,\quad1\leq i,j\leq N.
\end{align}
Further, define the norm $\|\tilde A\|_N$ by
\begin{align}
\|\tilde A\|_N=\max_{1\leq i,j\leq N}N^{\delta_{ij}-1}|\tilde A|_{ij}.
\end{align}

\begin{lemma}\label{multiply-close}
For $N\geq1$, let $A_N\in M_N(C_1)$ and $B_N\in M_N(C_2)$, then $A_NB_N\in M_N(3C_1C_2)$.
\end{lemma}
\begin{proof}
If $i\neq j$, then
\begin{align*}
\sum_{k=1}^N\left|a^N_{ik}b^N_{kj}\right|\leq
C_1\cdot\frac{C_2}N+\frac{C_1}N\cdot C_2+\frac{C_1}N\cdot\frac{C_2}N\cdot(N-2)<\frac{3C_1C_2}N.
\end{align*}
If $i=j$, then
\begin{align*}
\sum_{k=1}^N\left|a^N_{ik}b^N_{kj}\right|\leq
C_1C_2+\frac{C_1}N\cdot\frac{C_2}N\cdot(N-1)<3C_1C_2.
\end{align*}
Hence we have $A_NB_N\in M_N(3C_1C_2)$.
\end{proof}

\begin{lemma}\label{forward-scaling-eqn}
For each $N\geq1$, let $(W(t))_{t\in(0,T)}$ be a real valued standard Brownian motion. Let the $\mathbb R^{N\times N}$-valued processes  $ (X(t))_{t\in(0,T)},\ (A(t))_{t\in(0,T)},\ (B(t))_{t\in(0,T)}$, satisfy
\begin{align}\label{scaling-eqn}
X(t)=X(0)+\int_0^tX(s)A(s)ds+\int_0^tX(s)B(s)dW(s),\quad t\in[0,T].
\end{align}
{Suppose that $X(0)$ satisfies} for $k=1,2,\ldots$
\begin{align*}
\mathbb E\left[|X_{ij}(0)|^{2k}\right]\leq\left\{\begin{aligned}
C_{0,k},\quad i=j,\\
\frac {C_{0,k}}{N^{2k}},\quad i\neq j,
\end{aligned}\right.
\end{align*}
and $|A|,|B|\in M_N(C)$ for some {constant} $C$. Then
\begin{align*}
\mathbb E\left[\max_{0\leq s\leq T}|X_{ij}(s)|^{2k}\right]\leq\left\{\begin{aligned}
\tilde C_k,\quad i=j,\\
\frac {\tilde C_k}{N^{2k}},\quad i\neq j,
\end{aligned}\right.
\end{align*}
where $\tilde C_k=\tilde C_k(C_{0,k},C,T)$ is increasing in $T$ (but independent of $N$).
\end{lemma}
\begin{proof}
 We show that \eqref{scaling-eqn} admits a unique solution $X$, with the required estimates, which is also the fixed point of the mapping $\Phi:\ X\mapsto\tilde X$ defined as follows:
\begin{align}\label{apen-fixed-pn}
\tilde X(t)=X_0+\int_0^tX(s)A(s)ds+\int_0^tX(s)B(s)dW(s),\quad t\in[0,\delta],
\end{align}
where $\delta>0$ is a small enough positive number. Consider to inputs $X^{(1)}$ and $X^{(2)}$, for $1\leq p,q\leq N$,
\begin{align*}
&\quad\Phi_{pq}\big(X^{(1)}\big)(t)-\Phi_{pq}\big(X^{(2)}\big)(t)\notag\\
&=\int_0^t\sum_{k=1}^N\left[X^{(1)}_{pk}(s)-X^{(2)}_{pk}(s)\right]A_{kq}(s)ds+\int_0^t\sum_{k=1}^N\left[X^{(1)}_{pk}(s)-X^{(2)}_{pk}(s)\right]B_{kq}(s)dW(s),
\end{align*}
then according to Burkholder--Davis--Gundy inequality,
\begin{align*}
&\quad\mathbb E\left[\max_{t\in[0,\delta]}|\Phi_{pq}\big(X^{(1)}\big)(t)-\Phi_{pq}\big(X^{(2)}\big)(t)|^{2k}\right]\notag\\
    &\leq C_k\mathbb E\left[\int_0^\delta\bigg|\sum_{\substack{i=1\\i\neq q}}^N\bigg[X^{(1)}_{pi}(s)-X^{(2)}_{pi}(s)\bigg]A_{iq}(s)\bigg|^{2k}ds\right]\notag\\
&\quad+C_k\mathbb E\left[\int_0^\delta\bigg|\bigg[X^{(1)}_{pq}(s)-X^{(2)}_{pq}(s)\bigg]A_{qq}(s)\bigg|^{2k}ds\right]\notag\\
    &\quad+C_k\mathbb E\left[\int_0^\delta\bigg|\sum_{\substack{i=1\\i\neq q}}^N\bigg[X^{(1)}_{pi}(s)-X^{(2)}_{pi}(s)\bigg]B_{iq}(s)\bigg|^{2k}ds\right]\notag\\
&\quad+C_k\mathbb E\left[\int_0^\delta\bigg|\bigg[X^{(1)}_{pq}(s)-X^{(2)}_{pq}(s)\bigg]B_{qq}(s)\bigg|^{2k}ds\right].
\end{align*}
According to Jensen's inequality,
\begin{align}\label{holder}
    \bigg|\sum_{\substack{i=1\\i\neq q}}^N\bigg[X^{(1)}_{pi}(s)-X^{(2)}_{pi}(s)\bigg]A_{iq}(s)\bigg|^{2k}&\leq\frac{C^{2k}}{N^{2k}}\cdot(N-1)^{2k}\cdot\bigg|\frac1{N-1}\sum_{\substack{i=1\\i\neq q}}^N|X^{(1)}_{pi}(s)-X^{(2)}_{pi}(s)|\bigg|^{2k}\notag\\
&\leq C^{2k}\cdot\frac1{N-1}\sum_{\substack{i=1\\i\neq q}}^N\big|X^{(1)}_{pi}(s)-X^{(2)}_{pi}(s)\big|^{2k}.
\end{align}
Here we have used
\begin{align*}
\big|\big[X^{(1)}_{pi}(s)-X^{(2)}_{pi}(s)\big]A_{iq}(s)\big|\leq\left\{\begin{aligned}
C\big|X^{(1)}_{pi}(s)-X^{(2)}_{pi}(s)\big|,\quad i=q,\\
\frac CN\big|X^{(1)}_{pi}(s)-X^{(2)}_{pi}(s)\big|,\quad i\neq q.
\end{aligned}\right.
\end{align*}
Therefore
\begin{align}\label{holder-max}
    &\quad\mathbb E\bigg[\int_0^\delta\bigg|\sum_{\substack{i=1\\i\neq q}}^N\bigg[X^{(1)}_{pi}(s)-X^{(2)}_{pi}(s)\bigg]A_{iq}(s)\bigg|^{2k}ds\bigg]\notag\\
    &\leq\frac{C^{2k}}{N-1}\sum_{\substack{i=1\\i\neq q}}^N\mathbb E\left[\int_0^\delta\big|X^{(1)}_{pi}(s)-X^{(2)}_{pi}(s)\big|^{2k}ds\right]\notag\\
    &\leq\frac{C^{2k}\delta}{N-1}\sum_{\substack{i=1\\i\neq q}}^N\mathbb E\left[\max_{0\leq s\leq\delta}\big|X^{(1)}_{pi}(s)-X^{(2)}_{pi}(s)\big|^{2k}\right]\notag\\
&\leq C^{2k}\delta\max_{1\leq i,j\leq N}\mathbb E\left[\max_{0\leq t\leq\delta}|X^{(1)}_{ij}(t)-X^{(2)}_{ij}(t)|^{2k}\right].
\end{align}
Similar estimates to \eqref{holder} yields
\begin{align*}
    \mathbb E\left[\max_{t\in[0,\delta]}|\Phi_{pq}\big(X^{(1)}\big)(t)-\Phi_{pq}\big(X^{(2)}\big)(t)|^{2k}\right]\leq4\delta C_kC^{2k}\max_{1\leq i,j\leq N}\mathbb E\left[\max_{0\leq t\leq\delta}|X^{(1)}(t)-X^{(2)}(t)|^{2k}\right],
\end{align*}
and thus
\begin{align*}
    &\quad\max_{1\leq i,j\leq N}\mathbb E\left[\max_{t\in[0,\delta]}|\Phi_{ij}\big(X^{(1)}\big)(t)-\Phi_{ij}\big(X^{(2)}\big)(t)|^{2k}\right]\notag\\
    &\leq4\delta C_kC^{2k}\max_{1\leq i,j\leq N}\mathbb E\left[\max_{0\leq t\leq\delta}|X^{(1)}_{ij}(t)-X^{(2)}_{ij}(t)|^{2k}\right],
\end{align*}
Consider
\begin{align}\label{delta-condition}
\delta<\frac1{8C_kC^{2k}}.
\end{align}
  For the sake of later iterations, we note here that the choice of $\delta$ in \eqref{delta-condition} is independent of the bound $C_0$ of initial data.

In view of \eqref{delta-condition}, $\Phi$ is a contraction mapping. Next, we claim that $\Phi$ maps the following set
\begin{align}\label{inva-space}
    \mathcal X:=\left\{X:\ X{\rm\ is\ matrix\ valued\ process\ and\ }\mathbb E\left[\max_{0\leq s\leq \delta}|X_{ij}(s)|^{2k}\right]\leq M_k(N^{-2k}+\delta_{ij})\right\}.
\end{align}
into itself for some $M_k$.

To see the claim, consider $1\leq p,q\leq N$ and $p\neq q$,
\begin{align*}
&\quad\mathbb E\left[\max_{t\in[0,\delta]}|\Phi_{pq}\big(X\big)(t)|^{2k}\right]\notag\\
    &\leq C_k\mathbb E\big[\big|X_{pq}(0)\big|^{2k}\big]+C_k\mathbb E\bigg[\int_0^\delta\bigg|\sum_{\substack{i=1\\i\neq p,q}}^NX_{pi}(s)A_{iq}(s)\bigg|^{2k}ds\bigg]+C_k\mathbb E\bigg[\int_0^\delta\big|X_{pp}(s)A_{pq}(s)\big|^{2k}ds\bigg]\notag\\
    &\quad+C_k\mathbb E\bigg[\int_0^\delta\big|X_{pq}(s)A_{qq}(s)\big|^{2k}ds\bigg]+C_k\mathbb E\bigg[\int_0^\delta\bigg|\sum_{\substack{i=1\\i\neq p,q}}^NX_{pi}(s)B_{iq}(s)\bigg|^{2k}ds\bigg]\notag\\
    &\quad+C_k\mathbb E\bigg[\int_0^\delta\big|X_{pp}(s)B_{pq}(s)\big|^{2k}ds\bigg]+C_k\mathbb E\bigg[\int_0^\delta\big|X_{pq}(s)B_{qq}(s)\big|^{2k}ds\bigg].
\end{align*}
In view of \eqref{holder} and \eqref{inva-space},
\begin{align*}
    &\quad\mathbb E\bigg[\int_0^\delta\bigg|\sum_{\substack{i=1\\i\neq p,q}}^NX_{pi}(s)A_{iq}(s)\bigg|^{2k}ds\bigg]\leq\frac{C^{2k}}{N-2}\mathbb E\bigg[\sum_{\substack{i=1\\i\neq p,q}}^N\int_0^\delta\big|X_{pi}(s)\big|^{2k}ds\bigg]\notag\\
    &\leq\frac{\delta C^{2k}}{N-2}\sum_{\substack{i=1\\i\neq p,q}}^N\mathbb E\bigg[\max_{0\leq s\leq\delta}\big|X_{pi}(s)\big|^{2k}\bigg]\leq\frac{\delta C^{2k}M_k}{N^{2k}}.
\end{align*}
Combining the two inequalities above together, we arrive at
\begin{align*}
    \mathbb E\left[\max_{t\in[0,\delta]}|\Phi_{pq}\big(X\big)(t)|^{2k}\right]\leq\frac{C_kC^{2k}_0}{N^{2k}}+\frac{6\delta C_kM_kC^{2k}}{N^{2k}}.
\end{align*}
Similarly,
\begin{align*}
&\quad\mathbb E\left[\max_{t\in[0,\delta]}|\Phi_{pp}\big(X\big)(t)|^{2k}\right]\notag\\
    &\leq C_k\mathbb E\big[\big|X_{pp}(0)\big|^{2k}\big]+C_k\mathbb E\bigg[\int_0^\delta\bigg|\sum_{\substack{i=1\\i\neq p}}^NX_{pi}(s)A_{ip}(s)\bigg|^{2k}ds\bigg]+C_k\mathbb E\bigg[\int_0^\delta\big|X_{pp}(s)A_{pp}(s)\big|^{2k}ds\bigg]\notag\\
    &\quad+C_k\mathbb E\bigg[\int_0^\delta\bigg|\sum_{\substack{i=1\\i\neq p}}^NX_{pi}(s)B_{ip}(s)\bigg|^{2k}ds\bigg]+C_k\mathbb E\bigg[\int_0^\delta\big|X_{pp}(s)B_{pp}(s)\big|^{2k}ds\bigg]\notag\\
&\leq C_kC^{2k}_0+\frac{2\delta C_kM_kC^{2k}}{N^{2k}}+2\delta C_kM_kC^{2k}.
\end{align*}
Let $\delta$ and $M_k$ satisfy
\begin{align*}
\delta<\frac1{12C_kC^{2k}},\quad M_k>3C^{2k}_0.
\end{align*}
Note again that the choice of $\delta$ is still independent of $C_0$. Then estimate above implies that
\begin{align*}
\mathbb E\left[\max_{0\leq s\leq \delta}|\Phi(X)_{ij}(s)|^{2k}\right]\leq M_k(N^{-2k}+\delta_{ij}).
\end{align*}
  In other words, contraction mapping $\Phi$ maps $\mathcal X$ into itself. Hence the only fixed point of $\Phi$ lies in $\mathcal X$.

  To conclude the lemma, notice that the choice of $\delta$ is independent of $C_0$, therefore we can separate $[0,T]$ into $[0,\delta]$, $[\delta,2\delta]$, $[2\delta,3\delta]$, $\ldots$, then go over the procedure above repeatedly and obtain the desired results.
\end{proof}

\begin{proof}[Proof of Lemma \ref{theta-Lip}:]
 Note that for each $(t,x,p)\in[0,T]\times \mathcal A_N$, $\theta^*:=\theta^{R_1}_N(t,x,p)$ minimizes the strictly convex function $H^{R_1}_N\left(t,x,p,\theta\right)$ with respect to  $\theta\in\Theta$, hence
\begin{align*}
\langle\partial_\theta H^{R_1}_N\left(t,x,p,\theta^*\right),\theta-\theta^*\rangle\geq0,\ \theta\in\Theta.
\end{align*}
Similarly, for another pair of $(\hat x,\hat p)\in \mathcal A_N$ and $\hat\theta^*:=\theta^{R_1}_N(t,\hat x,\hat p)$,
\begin{align*}
\langle\partial_\theta H^{R_1}_N(t,\hat x,\hat p,\hat\theta^*),\theta-\hat\theta^*\rangle\geq0,\ \theta\in\Theta.
\end{align*}
Therefore we have by taking $\theta=\hat\theta^*,\ \theta^*$ that
\begin{align}\label{foc-0}
    0\geq\big\langle\partial_\theta H^{R_1}_N(t,x,p,\theta^*)-\partial_\theta H^{R_1}_N(t,\hat x,\hat p,\hat\theta^*),\theta^*-\hat\theta^*\big\rangle.
\end{align}
on the other hand,
\begin{align}\label{foc-1}
{ \partial_\theta H_N(t,x,p,\theta^*)-\partial_\theta H_N(t,\hat x,\hat p,\hat\theta^*)=I+II,}
\end{align}
where
\begin{align}
    \label{differ-I}&I:=\sum_{i=1}^Nf_\theta\bigg(t,\theta^*,x_i,\frac1N\sum_{j=1}^N\delta_{x_j}\bigg)p_i-\sum_{i=1}^Nf_\theta\bigg(t,\theta^*,\hat x_i,\frac1N\sum_{j=1}^N\delta_{\hat x_j}\bigg)\hat p_i,\\
    &{ II:=\lambda(\theta^*-\hat\theta^*)+\sum_{i=1}^N\bigg[f_\theta\bigg(t,\theta^*,\hat x_i,\frac1N\sum_{j=1}^N\delta_{\hat x_j}\bigg)-f_\theta\bigg(t,\hat\theta^*,\hat x_i,\frac1N\sum_{j=1}^N\delta_{\hat x_j}\bigg)\bigg]p_i.}\notag
\end{align}
{ For $i=1,2,\ldots,N,$
\begin{align*}
 &\quad\sum_{i=1}^N\bigg[f_\theta\bigg(t,\theta^*,\hat x_i,\frac1N\sum_{j=1}^N\delta_{\hat x_j}\bigg)-f_\theta\bigg(t,\hat\theta^*,\hat x_i,\frac1N\sum_{j=1}^N\delta_{\hat x_j}\bigg)\bigg]p_i\\
 &=\sum_{i=1}^Np_i\int_0^1\partial^2_{\theta\theta}f\bigg(t,\hat\theta^*+t(\theta-\theta^*),\hat x_i,\frac1N\sum_{j=1}^N\delta_{\hat x_j}\bigg)(\theta-\theta^*)dt
\end{align*}}
According to \eqref{hypo-regu-2} in the assumption, it holds for some constant $\lambda_0>0$ that
{ \begin{align}\label{foc-2}
(\theta-\theta^*)^\top\bigg[\sum_{i=1}^Np_i\partial^2_{\theta\theta}f\bigg(t,\hat\theta^*+t(\theta-\theta^*),\hat x_i,\frac1N\sum_{j=1}^N\delta_{\hat x_j}\bigg)\bigg](\theta-\theta^*)\geq-\lambda_0|\theta-\theta^*|^2,\quad t\in[0,1].
\end{align}}
Plugging \eqref{foc-1}, \eqref{foc-2} into \eqref{foc-0}, and using the Cauchy-Schwartz inequality, we have that
\begin{align}\label{differ-theta}
|\theta^*-\hat\theta^*|\leq(\lambda-\lambda_0)^{-1}|I|.
\end{align}
  According to \eqref{differ-I}, $I$ is the difference of the following function (w.r.t. $(x,p)\in\mathbb R^N\times\mathbb R^N$)
\begin{align*}
\sum_{i=1}^Nf_\theta\bigg(t,\hat\theta^*,x_i,\frac1N\sum_{j=1}^N\delta_{x_j}\bigg)p_i,
\end{align*}
which implies the local Lipschitz continuity of $\theta^{R_1}_N(t,x,p)$ with respect to $(x,p)\in \mathcal A_N$.

In view of the local Lipschitz continuity, $\theta^{R_1}_N(t,x,p)$ is differentiable almost everywhere. Furthermore, it follows from \eqref{differ-theta} that
\begin{align*}
&\quad\big|\partial_{x_k}\theta^{R_1}_N(t,x,p)\big|\notag\\
    &\leq(\lambda-\lambda_0)^{-1}\bigg|f_{\theta x}\bigg(t,\hat\theta^*,x_k,\frac1N\sum_{j=1}^N\delta_{x_j}\bigg)p_k+\frac1N\sum_{i=1}^N\partial_\mu f_\theta\bigg(t,\hat\theta^*,x_i,\frac1N\sum_{j=1}^N\delta_{x_j}\bigg)(x_k)p_i\bigg|.
\end{align*}
In view of \eqref{hypo-regu-3},
\begin{align*}
\big|\partial_{x_k}\theta^{R_1}_N(t,x,p)\big|\leq\frac{2(\lambda-\lambda_0)^{-1}C^Q}N.
\end{align*}
Similarly we also have
\begin{align*}
\big|\partial_{p_k}\theta^{R_1}_N(t,x,p)\big|\leq(\lambda-\lambda_0)^{-1}\|f_\theta\|_\infty.
\end{align*}
\end{proof}

{\bf Acknowledgements.} HL acknowledges the support provided by NSFC Grant 12531009, Hong Kong RGC Grant GRF 11311422 and Singapore MOE AcRF Grant R-146-000-271-112. ARM acknowledges the support provided by the EPSRC via the NIA with grant number EP/X020320/1 and by the King Abdullah University of Science and Technology Research Funding (KRF) under Award No. ORA-2021-CRG10-4674.2. CM acknowledges the support provided by NSFC Grant 12522122,  NSFC/RGC JRS N\_CityU165/25, Hong Kong RGC Grant GRF 11311422 and Hong Kong RGC Grant GRF 11303223. CZ acknowledges the support provided by Singapore MOE (Ministry of Education) AcRF Grant A-8000453-00-00, IoTex Foundation Industry Grant A-8001180-00-00 and NSFC Grant No. 11871364. We would like to thank Matthew Thorpe for carefully reading an earlier version of the manuscript and for his feedback on it. { We thank the anonymous referees for their constructive remarks and questions which led us to improve the presentation of our results.}

\medskip

{\bf Declarations.}

\medskip

{\bf Conflict of interest.} The authors have no financial or non-financial conflicts of interest to disclose.

\bibliographystyle{plain}
%\bibliography{citation}

\end{document}